\documentclass[10pt, twoside]{article}
\usepackage[utf8]{inputenc}

\title{Cartier smoothness in prismatic cohomology}
\author{Tess Vincent Bouis}
\date{}

\usepackage[left=3cm, right=3cm, top=1in, bottom=1in, headheight=2cm]{geometry}

\usepackage{fancyhdr}
\usepackage{amsmath, amscd, amsfonts, amssymb, amsthm,todonotes,eurosym,enumitem,relsize}

\usepackage[utf8]{inputenc}
\usepackage{dirtytalk}
\usepackage[T1]{fontenc}

\usepackage[colorlinks]{hyperref}
\usepackage[figure,table]{hypcap}
\hypersetup{
	bookmarksnumbered,
	pdfstartview={FitH},
	citecolor={black},
	linkcolor={black},
	urlcolor={black},
	pdfpagemode={UseOutlines}
}

\usepackage[object=pgfhan]{pgfornament}

\usepackage{graphicx}
\usepackage{array}
\usepackage{csquotes}
\usepackage{extarrows}

\usepackage{needspace}

\input{xy}
\xyoption{all}
\usepackage{tikz-cd}

\usepackage{sectsty}
\sectionfont{\sc\centering}
\subsectionfont{\bf\large}
\subsubsectionfont{\bf\normalsize}

\usepackage{fancyhdr}

\newlength{\outermargin} 
\newlength{\mar} 
\newlength{\len}
\newlength{\temp}
\newtheorem{theorem}{Theorem}[section]
\newtheorem{lemma}[theorem]{Lemma}
\newtheorem{notation}[theorem]{Notation}
\newtheorem{corollary}[theorem]{Corollary}
\newtheorem{proposition}[theorem]{Proposition}
\newtheorem{definition}[theorem]{Definition}
\newtheorem{construction}[theorem]{Construction}
\newtheorem{theoremintro}{Theorem}

\theoremstyle{definition}
\newtheorem{remark}[theorem]{Remark}
\newtheorem{example}[theorem]{Example}
\newtheorem{examples}[theorem]{Examples}

\DeclareMathOperator{\Z}{\mathbb{Z}}
\DeclareMathOperator{\Q}{\mathbb{Q}}
\DeclareMathOperator{\R}{\mathbb{R}}
\DeclareMathOperator{\N}{\mathbb{N}}
\DeclareMathOperator{\F}{\mathbb{F}}

\renewcommand{\epsilon}{\varepsilon}
\let\l\relax
\newcommand{\l}{\ell}

\usepackage[bbgreekl]{mathbbol}
\DeclareSymbolFontAlphabet{\mathbb}{AMSb}
\DeclareSymbolFontAlphabet{\mathbbl}{bbold}
\newcommand{\Prism}{{\mathlarger{\mathbbl{\Delta}}}}

\begin{document}

\pagestyle{fancy}
\fancyhead[EC]{TESS VINCENT BOUIS}
\fancyhead[OC]{CARTIER SMOOTHNESS IN PRISMATIC COHOMOLOGY}
\fancyfoot[C]{\thepage}

\maketitle

\begin{abstract}

We introduce the notion of a $p$\nobreakdash-Cartier smooth algebra. It generalises that of a smooth algebra and includes valuation rings over a perfectoid base. We give several characterisations of $p$\nobreakdash-Cartier smoothness in terms of prismatic cohomology, and deduce a comparison theorem between syntomic and étale cohomologies under this hypothesis.

\end{abstract}

\tableofcontents

\section{Introduction}\label{sec:intro}

\vspace{-\parindent}
\hspace{\parindent}

Let $p$ be a prime number. The Cartier isomorphism is a fundamental cohomological property of smooth $\F_p$-algebras. It states that for any smooth $\F_p$-algebra $A$, the inverse Cartier map
$$C^{-1} : \left\{  \begin{array}{ll} \Omega^n_{A/\F_p} \longrightarrow \text{H}^n(\Omega^\bullet_{A/\F_p}) \\ f dg_1 \wedge \dots \wedge dg_n \longmapsto f^p g_1^{p-1} \dots g_n^{p-1} dg_1 \wedge \dots \wedge dg_n
\end{array} \right.$$
is an isomorphism of $\F_p$-vector spaces for each $n\geq 0$. Assuming resolution of singularities in characteristic $p$, any $\F_p$-algebra would be locally smooth in the cdh-topology. The local rings in the cdh-topology are valuation rings, so valuation rings of characteristic $p$ should be filtered colimits of smooth $\F_p$-algebras, and in particular should satisfy the Cartier isomorphism. Without assuming resolution of singularities, Gabber proved \cite[Appendix]{kerz_towards_2021} that valuation rings of characteristic $p$ actually satisfy the Cartier isomorphism. Motivated by this result and the general lack of understanding of valuation rings, Kelly and Morrow \cite{kelly_k-theory_2021} introduced {\it Cartier smooth algebras} as $\F_p$\nobreakdash-algebras satisfying the Cartier isomorphism, and study their algebraic $K$\nobreakdash-theory.

We develop here an analogue of this story in mixed-characteristic. The inverse Cartier map, as an extension of the Frobenius morphism, is specific to characteristic $p$, and we define $p$\nobreakdash-Cartier smoothness for general ring homomorphisms essentially in terms of their reductions modulo $p$. This more general notion coincides with that introduced by Kelly--Morrow \cite{kelly_k-theory_2021} in the special case that the base ring is $\F_p$. For an $\F_p$-algebra $R$, we denote by $\phi_R$ its Frobenius endomorphism.

\needspace{5\baselineskip}

\begin{definition}[Cartier smoothness]
    \begin{enumerate}
        \item A morphism $R\rightarrow S$ of $\F_p$-algebras is \emph{Cartier smooth}, or $S$ is a \emph{Cartier smooth} $R$-algebra, if it is cotangent smooth, {\it i.e.}, its cotangent complex $\mathbb{L}_{S/R}$ is a flat $S$-module in degree $0$ given by $\Omega^1_{S/R}$, and if the inverse Cartier map
        $$C^{-1} : \Omega^n_{S/R} \otimes_{R,\phi_R} R \longrightarrow \emph{H}^n(\Omega^{\bullet}_{S/R})$$ is an isomorphism of $R$-modules for each $n\geq 0$.
        \item A morphism $R\rightarrow S$ of commutative rings is \emph{$p$\nobreakdash-Cartier smooth}, or $S$ is a \emph{$p$\nobreakdash-Cartier smooth} $R$\nobreakdash-algebra, if the natural map $S\otimes_R^\mathbb{L}R/p\longrightarrow S/p[0]$ is an equivalence in the derived category $\mathcal{D}(R/p)$ and the reduction $R/p\rightarrow S/p$ modulo $p$ is Cartier smooth.
    \end{enumerate}
\end{definition}

Said another way, a morphism $R \rightarrow S$ of commutative rings is $p$\nobreakdash-Cartier smooth if it is $p$\nobreakdash-cotangent smooth (Definition~\ref{Definitionquasismooth}) and if its reduction $R/p \rightarrow S/p$ modulo $p$ satisfies the Cartier isomorphism. The $p$\nobreakdash-cotangent smoothness hypothesis is necessary not to lose control, when specialising to characteristic~$p$, on the invariants we will study. Any smooth morphism of commutative rings is $p$\nobreakdash-Cartier smooth (for any prime $p$). Regarding valuation rings, an extension of valuation rings ({\it i.e.}, an injective morphism of valuation rings) in mixed-characteristic is in general not $p$\nobreakdash-Cartier smooth. For instance, the morphism $\Z_p \rightarrow \overline{\Z}_p$, where $\overline{\Z}_p$ is the ring of integers of an algebraic closure $\overline{\Q}_p$ of $\Q_p$, is not $p$\nobreakdash-cotangent smooth and does not satisfy the Cartier isomorphism: its cotangent complex is in degree $0$, given by the torsion $\overline{\Z}_p$-module $\overline{\Q}_p/\overline{\Z}_p$, and $\Omega^0_{(\overline{\Z}_p/p)/\F_p} = \overline{\Z}_p/p$ contains nilpotent elements; in fact the de Rham complex of $\F_p \rightarrow \overline{\Z}_p/p$ is zero in positive degrees, and thus does not convey a lot of information. 
However, if the base valuation ring is sufficiently ramified, these obstructions vanish and we can prove the following general result. This generalises a theorem of Gabber in characteristic $p$, which is valid over perfect valuation rings (Theorem~\ref{TheoremGabberandGabberRamero} below). We refer the reader to \cite[Definition~$3.5$]{bhatt_integral_2018} for the definition of perfectoid rings.

\begin{theoremintro}[see Theorem~\ref{mainthmvaluationrings}]\label{mainthmvaluationringsintro}
    Let $C^+$ be a valuation ring whose $p$\nobreakdash-adic completion is a perfectoid ring. Let $E^+$ be a valuation ring extension of $C^+$. Then the morphism $C^+ \rightarrow E^+$ is $p$\nobreakdash-Cartier smooth.
\end{theoremintro}

The proof proceeds by reduction to the case of valuation rings of characteristic $p$, and uses deformation theory.

With the examples of valuation rings and smooth algebras in mind, we study the algebraic $K$\nobreakdash-theory of general $p$\nobreakdash-Cartier smooth algebras over a perfectoid ring. Recall that in certain cases algebraic $K$\nobreakdash-theory is equipped with a motivic filtration with graded pieces $\Z(i)(-) \in \mathcal{D}(\Z)$, for $i\geq 0$, called motivic cohomology. The motivic filtration, with its expected properties, was defined by Voevodsky for smooth varieties over a field. In characteristic $p$ the Beilinson--Lichtenbaum conjecture, proved by Suslin and Voevodsky as a consequence of the Bloch--Kato conjecture \cite{suslin_bloch-kato_1998}, computes the $\l$\nobreakdash-adic part of motivic cohomology in terms of the étale cohomology. More precisely, given a field $k$ of characteristic $p$, a smooth $k$-variety $X$ and a prime $\l$ different from $p$, it states that there is a natural quasi-isomorphism
$$\Z/\l^n(i)(X) \simeq R\Gamma_\emph{Zar}(X,\tau^{\leq i} R\alpha_\ast \mu_{\l^n}^{\otimes i}),$$
where $\mu_{\l^n}$ is the étale sheaf of $\l^n$-roots of unity and $\alpha : (\emph{Sm}/k)_{\emph{ét}} \rightarrow (\emph{Sm}/k)_{\emph{Zar}}$ is the projection of sites.

The Beilinson--Lichtenbaum conjecture can be interpreted as a comparison between Zariski motivic cohomology and étale motivic cohomology. The étale sheaf $\mu_{\l^n}$ is indeed naturally identified with the étale sheafification $\Z/\l^n(i)^\text{ét}$ of the (Zariski) motivic sheaf $\Z/\l^n(i)$, and in fact with the étale sheafification of $K_{2i}(-;\Z/\l^n)$ by rigidity results from \cite{suslin_k-theory_1983,gabber_k-theory_1992}.

In characteristic $p$, to describe the $p$\nobreakdash-adic part of motivic cohomology, and thus the $p$\nobreakdash-adic part of algebraic $K$\nobreakdash-theory, one needs to replace $\mu_{\l^n}$ (which is zero on smooth varieties when $\l=p$) by the logarithmic de Rham--Witt sheaf $W_n\Omega^i_{-,\text{log}}$, see \cite{geisser_k-theory_2000}. This description of $p$\nobreakdash-adic algebraic $K$\nobreakdash-theory in terms of the logarithmic de Rham--Witt sheaf is generalised in \cite{kelly_k-theory_2021} to all Cartier smooth $\F_p$\nobreakdash-algebras.

In mixed-characteristic, the $\l$-adic part of algebraic $K$\nobreakdash-theory can often be computed by Gabber's rigidity theorem \cite{gabber_k-theory_1992}, but much less is known about the $p$\nobreakdash-adic part. The syntomic complexes~$\Z/p^n(i)^\text{syn}$, introduced in \cite[Section $7.4$]{bhatt_topological_2019} and generalised in \cite[Section $8$]{bhatt_absolute_2022}, provide a well-behaved theory of $p$\nobreakdash-adic étale motivic cohomology for arbitrary schemes. On $p$\nobreakdash-henselian $p$\nobreakdash-quasisyntomic rings $S$ (a nonnoetherian generalisation of $p$-henselian local complete intersection rings) the complex $\Z/p^n(i)^\text{syn}(S)$ is concentrated in degrees $[0;i+1]$ by \cite[Theorem $5.1$]{antieau_beilinson_2020}. The top degree~$i+1$ of $\Z/p^n(i)^\text{syn}(S)$ can be computed in terms of the logarithmic de Rham--Witt forms by \cite[Corollary~$5.43$]{antieau_beilinson_2020}, and vanishes étale locally. We provide a description in smaller degrees of $\Z_p(i)^\text{syn}(S)$ in terms of the étale cohomology of the generic fibre of $S$, when $S$ is $p$\nobreakdash-Cartier smooth over a perfectoid ring. Note that Bhatt and Mathew introduce in \cite{bhatt_syntomic_2022} the notion of $F$\nobreakdash-smoothness in terms of absolute prismatic cohomology, and prove the following result for $p$\nobreakdash-torsionfree $F$\nobreakdash-smooth rings. Over a perfectoid base ring, $F$\nobreakdash-smoothness and $p$\nobreakdash-Cartier smoothness coincide by Theorem~\ref{TheoremFsmoothequivalentCSmoverperfectoid}. Also note that Bhatt--Morrow--Scholze \cite{bhatt_topological_2019} proved the following result for smooth $p$-adic formal schemes over mixed-characteristic algebraically closed valuation rings (see Remark~\ref{RemarkcomparisonBMS2smoothoverOC}).

\begin{theoremintro}[Syntomic-étale comparison theorem, see Theorems~\ref{Theoremsyntomicetalecomparison} and \ref{Theoremsyntomicetalecomparison2}]\label{Theoremsyntomicetalecomparisonintro}
    Let $R$ be a perfectoid ring, and $S$ a $p$\nobreakdash-Cartier smooth $R$-algebra. Then for any integers $i\geq 0$ and $n\geq 1$, the map $$\Z/p^n(i)^\emph{syn}(S) \longrightarrow R\Gamma_{\emph{ét}}(\emph{Spec}(S[\tfrac{1}{p}]),\mu_{p^n}^{\otimes i})$$ is an isomorphism on cohomology in degrees $<i-1$, and is injective on $\emph{H}^{i-1}$. If the perfectoid ring~$R$ is $p$\nobreakdash-torsionfree, this map is an isomorphism on cohomology in degrees $<i$, and is injective on $\emph{H}^i$.
\end{theoremintro}

The proof relies on the comparison between étale cohomology and prismatic cohomology \cite[Theorem~$9.1$]{bhatt_prisms_2022}, and a property for $p$\nobreakdash-Cartier smooth algebras of the Nygaard filtration on relative prismatic cohomology (Theorem~\ref{TheoremCartiersmoothMainintro}\,$(\mathcal{N}^{\geq})$ below). To prove this property, we extend various results on the relative prismatic cohomology of smooth algebras to general $p$\nobreakdash-Cartier smooth algebras. It is a priori surprising that these properties are true assuming only a property on the special fibre (namely, the Cartier isomorphism); we prove that some of them even characterise $p$\nobreakdash-Cartier smoothness.

\begin{theoremintro}[see Theorem~\ref{TheoremCartiersmoothMain}]\label{TheoremCartiersmoothMainintro}
    Let $(A,I)$ be a bounded prism, and $S$ a $p$\nobreakdash-cotangent smooth $A/I$\nobreakdash-alge\-bra. The following are equivalent:
    \begin{description}[align=left]
        \item[$(CSm)$] $S$ is $p$\nobreakdash-Cartier smooth over $A/I$.
        \item[$(\mathbb{L}\Omega=\widehat{\mathbb{L}\Omega})$] The canonical map $(\mathbb{L}\Omega_{S/(A/I)})^\wedge_p \rightarrow (\widehat{\mathbb{L}\Omega}_{S/(A/I)})^\wedge_p$ is an equivalence in the derived category $\mathcal{D}(A/I)$, where $(\widehat{\mathbb{L}\Omega}_{S/(A/I)})^\wedge_p$ is the Hodge-completion of the $p$\nobreakdash-completed derived de Rham complex.
        \item[$(\mathbb{L}\Omega = \Omega)$] The counit map $(\mathbb{L}\Omega_{S/(A/I)})^\wedge_p \rightarrow (\Omega_{S/(A/I)})^\wedge_p$ is an equivalence in the derived category $\mathcal{D}(A/I)$, where $(\mathbb{L}\Omega_{S/(A/I)})^\wedge_p$ is the $p$\nobreakdash-completed derived de Rham complex.
        \item[$(dR)$] The de Rham comparison map $\Prism^{(1)}_{S/A} \otimes_{A}^\mathbb{L} A/I \rightarrow (\Omega_{S/(A/I)})^\wedge_p$ is an equivalence in the derived category $\mathcal{D}(A/I)$.
        \item[$(\Prism=\widehat{\Prism})$] The canonical map $\emph{can} : \Prism^{(1)}_{S/A} \rightarrow \widehat{\Prism}^{(1)}_{S/A}$ is an equivalence in the derived category $\mathcal{D}(A)$, where $\widehat{\Prism}^{(1)}_{S/A}$ is the Nygaard-completion of the prismatic complex.
        \item[$(L\eta)$] The Frobenius map $\tilde{\phi} : \Prism^{(1)}_{S/A} \rightarrow L\eta_I \Prism_{S/A}$ is an equivalence in the derived category $\mathcal{D}(A)$.
        \item[$(\mathcal{N}^{\geq})$] The Frobenius map $\tau^{\leq i}\phi : \tau^{\leq i} \mathcal{N}^{\geq i} \Prism^{(1)}_{S/A} \rightarrow \tau^{\leq i} I^i\Prism_{S/A}$ is an equivalence in the derived category $\mathcal{D}(A)$ for all $i\geq 0$.
    \end{description}
\end{theoremintro}

\subsection*{Outline.}

\vspace{-\parindent}
\hspace{\parindent}

In Section~\ref{SectionPrismaticCohomology} we study the relative prismatic cohomology of $p$\nobreakdash-Cartier smooth algebras and prove Theorem~\ref{TheoremCartiersmoothMainintro} (Theorem~\ref{TheoremCartiersmoothMain}). We also provide a comparison between $p$\nobreakdash-Cartier smoothness and $F$\nobreakdash-smooth\-ness over a perfectoid base (subsection \ref{subsectionFsmoothness}). In Section~\ref{sectionvaluationrings} we prove Theorem~\ref{mainthmvaluationringsintro} on the $p$\nobreakdash-Cartier smoothness of certain valuation ring extensions (Theorem~\ref{mainthmvaluationrings}). In Section~\ref{Sectionsyntomicetalecomparison} we prove the syntomic-étale comparison Theorem~\ref{Theoremsyntomicetalecomparisonintro} (Theorems~\ref{Theoremsyntomicetalecomparison} and \ref{Theoremsyntomicetalecomparison2}), applying the results of Section~\ref{SectionPrismaticCohomology} to the case of a perfectoid base.

\subsection*{Notation.}

\vspace{-\parindent}
\hspace{\parindent}

Given a commutative ring $R$, denote by $\mathcal{D}(R)$ the derived $\infty$-category of $R$-modules, in which we use the cohomological convention for the indices. Also denote by $\mathcal{DF}(R) := \text{Fun}((\Z,\geq)^\text{op},\mathcal{D}(R))$ the filtered derived $\infty$-category of $R$-modules, where the filtrations are decreasing. Given a filtered complex $\text{Fil}^\star\, C \in \mathcal{DF}(R)$ and for each integer $n \in \Z$, let $\text{gr}^n\, C \in \mathcal{D}(R)$ be the homotopy cofibre of the transition map $\text{Fil}^{n+1}\,C \rightarrow \text{Fil}^n\,C$. A filtered complex $\text{Fil}^\star\, C \in \mathcal{DF}(R)$ is said to be complete if the limit $\text{lim}_n\, \text{Fil}^n\, C$ vanishes in the derived $\infty$-category $\mathcal{D}(R)$. For a nonzerodivisor $d$ in $R$, the (derived) reduction $C/d$ modulo $d$ of a complex $C \in \mathcal{D}(R)$ is the homotopy cofibre of the map $d : C \rightarrow C$ induced by multiplication by $d$, {\it i.e.}, $C/d := C \otimes_R^{\mathbb{L}} R/d$. If $C$ is given by an $R$-module $M$ in degree~$0$, we also call derived reduction of $M$ modulo $d$ the complex $M \otimes_R^{\mathbb{L}} R/d$; it is given by $M/d$ (the classical reduction of $M$ modulo $d$) in degree $0$ and by $M[d]$ (the $d$\nobreakdash-torsion of $M$) in degree $-1$. The (derived) $d$\nobreakdash-adic completion of a complex $C \in \mathcal{D}(R)$ is the derived limit of the complexes $C/d^n$ over $n\geq 1$. An $R$-module $M$ is said to have bounded $d^\infty$-torsion if there exists an integer $n\geq 1$ such that $M[d^m]=M[d^n]$ for all $m \geq n$; this assumption guarantees that the derived $d$\nobreakdash-adic completion of $M$ is in degree $0$, given by the classical $d$\nobreakdash-adic completion of $M$. Given a complex $C \in \mathcal{D}(R)$ and some integers $a,b \in \Z$, $a \leq b$, there is a naturally associated complex $\tau^{[a;b]} C \in \mathcal{D}(R)$ which is concentrated in degrees $[a;b]$ and with the same cohomology groups as $C$ in degrees $[a;b]$. A complex $C \in \mathcal{D}(R)$ is said to have Tor-amplitude in degrees $[a;b]$ if for any $R$-module $N$, the complex $C \otimes_R^\mathbb{L} N$ is in degrees~$[a;b]$. Similarly, $C$ is said to have $d$\nobreakdash-complete Tor-amplitude in degrees $[a;b]$ if $C \otimes_R^\mathbb{L} R/d \in \mathcal{D}(R/d)$ has Tor-amplitude in degrees~$[a;b]$. Given an ideal $I \subset R$, a complex $C \in \mathcal{D}(R)$ is $I$-adically complete if it is $d$\nobreakdash-adically complete for every element $d \in I$.

Given a commutative ring $R$, an $\infty$-category $\mathcal{D}$ which admits sifted colimits ({\it e.g.}, $\mathcal{D}(R)$ or $\mathcal{DF}(R)$), and a functor $$\mathcal{F} : \text{Sm}_R := \{\text{smooth }R\text{-algebras}\} \longrightarrow \mathcal{D},$$ define
$$\begin{array}{ll} \mathbb{L}\mathcal{F} : &R\text{-Alg} \longrightarrow \mathcal{D} \\ &S \longmapsto \underset{P\rightarrow S}{\text{colim}} \text{ } \mathcal{F}(P),\end{array}$$
where the colimit is taken over all free $R$-algebras $P$ mapping to $S$. The functor $\mathbb{L}\mathcal{F}$ is called the left Kan extension (from free $R$-algebras) of $\mathcal{F}$. For instance, the cotangent complex $\mathbb{L}_{-/R} := \mathbb{L}\Omega^1_{-/R}$ is the left Kan extension of the module of Kähler differentials $\Omega^1_{-/R}$, and the derived de Rham complex $\mathbb{L}\Omega_{-/R}$ is the left Kan extension of the de Rham complex $\Omega_{-/R}$. The Hodge filtration on the de Rham complex is given for each $n \in \Z$ by $\text{Fil}_\text{Hod}^n\, \Omega_{-/R} := \Omega_{-/R}^{\geq n}$; the Hodge filtration on the derived de Rham complex is defined as the left Kan extension of this filtration, and is not always complete.

\subsection*{Acknowledgments.}

\vspace{-\parindent}
\hspace{\parindent}

I am very grateful to Matthew Morrow for suggesting this project to me, sharing many insights and for careful readings of this manuscript. I would also like to thank Elden Elmanto and Arnab Kundu for helful discussions, Javier Fresán, Akhil Mathew and Mohamed Moakher for comments on a draft of this paper, and the referee for many helpful comments and corrections. This project has received funding from the European Research Council (ERC) under the European Union’s Horizon 2020 research and innovation programme (grant agreement \hbox{No. 101001474}).

\section{Prismatic cohomology of Cartier smooth algebras}\label{SectionPrismaticCohomology}

\vspace{-\parindent}
\hspace{\parindent}

Let $p$ be a prime number. In this section we introduce the notion of $p$\nobreakdash-Cartier smoothness for morphisms of commutative rings (Definition~\ref{DefinitionCartiersmooth}), and give several characterisations in terms of prismatic cohomology (Theorem~\ref{TheoremCartiersmoothMain}). 

The $p$\nobreakdash-Cartier smooth morphisms generalise smooth morphisms, and behave as if they were smooth from the point of view of $p$\nobreakdash-adic étale motivic cohomology. More precisely, the main interest of this article is in syntomic cohomology (see Section~\ref{Sectionsyntomicetalecomparison}), which provides a general theory of $p$\nobreakdash-adic étale motivic cohomology. Syntomic cohomology can be defined in terms of Frobenius eigenspaces of prismatic cohomology, and this section is devoted to the prismatic cohomology of $p$\nobreakdash-Cartier smooth algebras. Our main result is Theorem~\ref{TheoremCartiersmoothMain}. We also give a comparison with the notion of $F$\nobreakdash-smoothness introduced in \cite{bhatt_syntomic_2022} (see subsection \ref{subsectionFsmoothness}).

\needspace{4\baselineskip}

\subsection{Definitions}

\vspace{-\parindent}
\hspace{\parindent}

For a commutative ring $R$, the cotangent complex $$\mathbb{L}_{-/R} : R\text{-Alg} \longrightarrow \mathcal{D}(R)$$
is the natural derived version of the module of Kähler differentials $$\Omega^1_{-/R} : R\text{-Alg} \longrightarrow R\text{-Mod}.$$
It controls information on the relative prismatic complex via the Hodge--Tate comparison theorem \cite[Theorem~$4.11$]{bhatt_prisms_2022}. The condition that $S \otimes_R^\mathbb{L} R/p \in \mathcal{D}(R/p)$ is in degree $0$ and base change for the cotangent complex imply that the morphism $\mathbb{L}_{S/R} \otimes_R^\mathbb{L}R/p \rightarrow \mathbb{L}_{(S/p)/(R/p)}$ is an equivalence. This allows one to lift properties of the cotangent complex of $R/p \rightarrow S/p$ to similar properties on the $p$\nobreakdash-completed cotangent complex of $R \rightarrow S$.

\begin{definition}[$p$-discreteness]
    A morphism $R\rightarrow S$ of commutative rings is \emph{$p$\nobreakdash-discrete} if the derived tensor product $S\otimes^\mathbb{L}_R R/p \in\mathcal{D}(R/p)$ is concentrated in degree $0$, where it is given by $S/p$.
\end{definition}

\begin{examples}\label{examplepdiscrete}
    A morphism $R \rightarrow S$ of commutative rings is $p$-discrete in the following cases.
    \begin{enumerate}
        \item $R$ and $S$ are $\F_p$-algebras.
        \item The morphism $R \rightarrow S$ is flat.
        \item $R$ and $S$ are $p$-torsionfree. More precisely, if $R$ is $p$-torsionfree, then the morphism $R \rightarrow S$ is $p$-discrete if and only if $S$ is $p$-torsionfree.
    \end{enumerate}
\end{examples}

The following definition axiomatises some properties satisfied by the cotangent complex in the smooth case.

\begin{definition}[Cotangent smoothness, \cite{bhatt_topological_2019}]\label{Definitionquasismooth}
    \begin{enumerate}
        \item A morphism $R\rightarrow S$ of commutative rings is \emph{cotangent smooth}, or $S$ is a \emph{cotangent smooth} $R$-algebra, if the cotangent complex $\mathbb{L}_{S/R}$ has Tor-amplitude in $[0;0]$, i.e., $\Omega^1_{S/R}$ is a flat $S$-module and $\emph{H}_n(\mathbb{L}_{S/R}) = 0$ for all $n>0$.
        \item A morphism $R\rightarrow S$ of commutative rings is \emph{$p$\nobreakdash-cotangent smooth}, or $S$ is a \emph{$p$\nobreakdash-cotangent smooth} $R$-algebra, if it is $p$-discrete and its reduction $R/p\rightarrow S/p$ modulo $p$ is cotangent smooth.\footnote{More generally, a morphism $R\rightarrow S$ of commutative rings, possibly with bounded $p^\infty$-torsion, will be called $p$\nobreakdash-something if it is $p$\nobreakdash-discrete and if its reduction $R/p\rightarrow S/p$ modulo $p$ is something. We distinguish this from similar notions on $p$\nobreakdash-complete rings, often called $p$\nobreakdash-completely something.}
    \end{enumerate}
\end{definition}

\begin{examples}\label{Examplequasismoothness}
    \begin{enumerate}
        \item A smooth morphism of commutative rings is cotangent smooth, because its cotangent complex is in degree $0$, where it is given by the locally free module of differential forms.
        \item A morphism of perfect $\F_p$\nobreakdash-algebras is cotangent smooth, because its cotangent complex is zero. More generally, a $p$\nobreakdash-discrete morphism $R \rightarrow S$ of perfectoid rings is $p$-cotangent smooth. Indeed, the cotangent complex $\mathbb{L}_{(S/p)/(R/p)}$ is zero, as a base change of the cotangent complex $\mathbb{L}_{S^{\flat}/R^{\flat}}$ of the morphism $R^\flat \rightarrow S^{\flat}$ of perfect $\F_p$\nobreakdash-algebras. In particular, a morphism of $p$-torsionfree perfectoid rings is $p$-cotangent smooth.
        \item A filtered colimit of ($p$\nobreakdash-)cotangent smooth algebras is ($p$\nobreakdash-)cotangent smooth, because the cotangent complex commutes with filtered colimits.
        \item Let $C^+$ be a valuation ring with perfect fraction field, and $E^+$ a valuation ring extension of~$C^+$ (see section \ref{sectionvaluationrings}). The morphism $C^+ \rightarrow E^+$ has cotangent complex concentrated in degree $0$ (\cite[Theorem~$6.5.8\,(ii)$]{gabber_almost_2003}). If $C^+$ contains a perfect field, then the morphism $C^+ \rightarrow E^+$ is cotangent smooth (\cite[Corollary~$6.5.21$]{gabber_almost_2003}). If the $p$\nobreakdash-adic completion of $C^+$ is a perfectoid ring, then the morphism $C^+ \rightarrow E^+$ is $p$\nobreakdash-cotangent smooth (Proposition~\ref{Propositionquasismoothnessvaluationrings} below).
    \end{enumerate}
\end{examples}

Several properties of prismatic cohomology in the smooth case --such as the comparison between prismatic cohomology and de Rham cohomology \cite[Corollary~$15.4$]{bhatt_prisms_2022}-- do not hold for general $p$\nobreakdash-cotangent smooth morphisms $R\rightarrow S$. We will prove (Theorem~\ref{TheoremCartiersmoothMain}) that the obstruction to satisfy these properties vanishes exactly when the morphism $R/p \rightarrow S/p$ satisfies the Cartier isomorphism. For an $\F_p$-algebra $R$, we denote by $\phi_R$ its Frobenius endomorphism.

\begin{definition}[Cartier smoothness]\label{DefinitionCartiersmooth}
    \begin{enumerate}
        \item A morphism $R\rightarrow S$ of $\F_p$-algebras is \emph{Cartier smooth}, or $S$ is a \emph{Cartier smooth} $R$-algebra, if it is cotangent smooth and if the inverse Cartier map
        $$C^{-1} : \Omega^n_{S/R} \otimes_{R,\phi_R} R \longrightarrow \emph{H}^n(\Omega^{\bullet}_{S/R})$$ is an isomorphism of $R$-modules for each $n\geq 0$.
        \item A morphism $R\rightarrow S$ of commutative rings is \emph{$p$\nobreakdash-Cartier smooth}, or $S$ is a \emph{$p$\nobreakdash-Cartier smooth} $R$\nobreakdash-algebra, if it is $p$\nobreakdash-discrete and its reduction $R/p\rightarrow S/p$ modulo $p$ is Cartier smooth.
    \end{enumerate}
\end{definition}

\begin{remark}
    A morphism $R \rightarrow S$ of commutative rings is then $p$-Cartier smooth (resp. $p$\nobreakdash-cotangent smooth) if and only if its $p$\nobreakdash-adic completion $R^\wedge_p \rightarrow S^\wedge_p$ is $p$-Cartier smooth (resp. $p$\nobreakdash-cotangent smooth).
\end{remark}

Equivalently, a morphism $R \rightarrow S$ of commutative rings is $p$\nobreakdash-Cartier smooth if it is $p$\nobreakdash-cotangent smooth and if the morphism $R/p \rightarrow S/p$ satisfies the Cartier isomorphism. For a morphism $R\rightarrow S$ of $\F_p$-algebras, recall that the inverse Cartier map $$C^{-1} : \Omega^n_{S/R} \otimes_{R,\phi_R} R \longrightarrow \text{H}^n(\Omega^{\bullet}_{S/R}),$$ for any integer $n\geq 0$, is defined as the unique $R$-linear map satisfying $$C^{-1}(fdg_1\wedge\dots\wedge dg_n \otimes 1)= f^pg_1^{p-1}\dots g_n^{p-1}dg_1\wedge\dots\wedge dg_n.$$ Cartier smoothness when $R=\F_p$ was already studied in \cite{kelly_k-theory_2021}.

\begin{examples} A morphism $R\rightarrow S$ of commutative rings is $p$\nobreakdash-Cartier smooth in the following cases.
	\begin{enumerate}
		\item The morphism $R\rightarrow S$ is smooth. Indeed, $S$ is a flat $R$-algebra and the cotangent complex $\mathbb{L}_{S/R}$ is a flat $S$-module in degree $0$, so the morphism $R\rightarrow S$ is $p$\nobreakdash-cotangent smooth. The Cartier isomorphism is a standard property of smooth morphisms in characteristic $p$ (see for instance \cite[Theorem~$1.2$]{deligne_relevements_1987}).
		\item $S$ is a filtered colimit of smooth $R$-algebras. Indeed, filtered colimits of $p$\nobreakdash-cotangent smooth algebras are $p$\nobreakdash-cotangent smooth, and the inverse Cartier map commutes with filtered colimits.
		\item $R=\F_p$ and $S$ is a valuation ring. The cotangent smoothness is a result of Gabber--Ramero (Example~\ref{Examplequasismoothness}\,$(4)$). The Cartier isomorphism is a result of Gabber (\cite[Corollary~$A.4$]{kerz_towards_2021}).
		\item $R \rightarrow S$ is a $p$\nobreakdash-discrete morphism of perfectoid rings. The cotangent complex of the morphism $R/p \rightarrow S/p$ is zero (Example~\ref{Examplequasismoothness}\,$(2)$), so the Cartier isomorphism is trivial for $n\geq 1$. If $R$ and $S$ are perfect $\F_p$-algebras and $n=0$ the Cartier isomorphism holds; in general the inverse Cartier map for $n=0$ is the reduction modulo $p^\flat$ of the inverse Cartier map associated to the morphism $R^\flat \rightarrow S^\flat$ of perfect $\F_p$-algebras, and is thus an isomorphism.
		\item $R$ is a valuation ring whose $p$\nobreakdash-adic completion is a perfectoid ring, and $S$ is a valuation ring extension of $R$ (Theorem~\ref{mainthmvaluationrings} below).
	\end{enumerate}
\end{examples}

\begin{example}
    An imperfect $\F_p$-algebra $S$ whose cotangent complex $\mathbb{L}_{S/\F_p}$ is trivial, is cotangent smooth but not Cartier smooth. Indeed, it is cotangent smooth because any algebra over the field $\F_p$ is flat and because the zero complex of $S$-modules has Tor-amplitude in $[0;0]$. In degree $0$, the inverse Cartier map for the map $\F_p \rightarrow S$ is the Frobenius map $S \rightarrow S$, and is thus not an isomorphism; so the $\F_p$-algebra $S$ is not Cartier smooth. Remark that imperfect $\F_p$-algebras with trivial cotangent complex exist, by \cite{bhatt_imperfect_2013}.
\end{example}

\begin{lemma}[Base change for Cartier smoothness]\label{LemmabasechangeforCartiersmoothness}
    Let $R \rightarrow S$ be a $p$\nobreakdash-Cartier smooth morphism, and $R'$ an $R$-algebra. Then the morphism $R'\rightarrow S \otimes_R R'$ is $p$\nobreakdash-Cartier smooth.
\end{lemma}

\begin{proof}
    The $p$\nobreakdash-cotangent smoothness is preserved by base change. The Cartier isomorphism depends only on the reduction modulo $p$ of the morphism $R' \rightarrow S \otimes_R R'$, so we can assume that $R$, $R'$ and $S$ are $\F_p$\nobreakdash-algebras. Let $n \geq 0$ be an integer. There is a canonical isomorphism of $R'$-modules $$\Omega^n_{(S \otimes_R R')/R'} \cong \Omega^n_{S/R} \otimes_R R'.$$ By Cartier smoothness of the morphism $R \rightarrow S$, the $R$-module $\text{H}^n(\Omega^\bullet_{S/R})$ is flat, so there is a canonical isomorphism of $R'$-modules $$\text{H}^n(\Omega^\bullet_{S/R}) \otimes_R R' \cong \text{H}^n(\Omega^\bullet_{S/R} \otimes^\mathbb{L}_R R').$$
    Moreover, $\Omega^\bullet_{S/R}$ is a complex of flat $R$-modules by cotangent smoothness of the morphism $R\rightarrow S$, so there is a canonical isomorphism of $R'$-modules $$\text{H}^n(\Omega^\bullet_{S/R} \otimes^\mathbb{L}_R R') \cong \text{H}^n(\Omega^\bullet_{(S\otimes_R R')/R'})$$
    and the morphism $R' \rightarrow S \otimes_R R'$ satisfies the Cartier isomorphism.
\end{proof}

\begin{lemma}[Transitivity of Cartier smoothness]\label{lemmatransitivityCartiersmoothness}
    Let $R \rightarrow S$ and $S \rightarrow T$ be $p$-Cartier smooth morphisms. Then the composite $R \rightarrow T$ is also $p$-Cartier smooth.
\end{lemma}

\begin{proof}
    The $p$\nobreakdash-cotangent smoothness and the Cartier isomorphism for $n=1$ follow from the transitivity sequence for the cotangent complex. The Cartier isomorphism follows in general from the transitivity property of the derived de Rham complex (\cite[Proposition~$3.22$]{bhatt_p-adic_2012}).
\end{proof}

\begin{corollary}\label{CorollaryCartiersmoothtensorproduct}
    Let $R \rightarrow S$ and $R \rightarrow S'$ be $p$\nobreakdash-Cartier smooth morphisms. Then the morphism $R \rightarrow S \otimes_R S'$ is also $p$\nobreakdash-Cartier smooth.
\end{corollary}

\begin{proof}
    The morphism $S \rightarrow S \otimes_R S'$ is $p$\nobreakdash-Cartier smooth by Lemma~\ref{LemmabasechangeforCartiersmoothness}. The morphism $R \rightarrow S \otimes_R S'$ is then $p$-Cartier smooth, as the composite of two $p$-Cartier smooth morphisms (Lemma~\ref{lemmatransitivityCartiersmoothness}).
\end{proof}

\subsection{Review of relative prismatic cohomology}\label{subsectionreviewrelativeprismaticcohomology}

\vspace{-\parindent}
\hspace{\parindent}

Prisms are defined in \cite{bhatt_prisms_2022} as pairs $(A,I)$, where $A$ is a $\delta$-ring (roughly,\footnote{More precisely, a $\delta$-ring is a $\Z_{(p)}$-algebra $A$ equipped with a map $\delta_A : A \rightarrow A$ such that $\phi_A : A \rightarrow A, x \mapsto x^p + p \delta_A(x)$ is a ring homomorphism. If $A$ is a $p$\nobreakdash-torsionfree $\Z_{(p)}$-algebra, a $\delta$-ring structure on $A$ is the same as a ring endomorphism lifting the Frobenius on $A/p$. If $A$ is a general $\Z_{(p)}$-algebra, a $\delta$-ring structure on $A$ is the same as a lift of Frobenius in the derived sense (\cite[Remark~$2.5$]{bhatt_prisms_2022}).} a $\Z_{(p)}$-algebra with a lift of Frobenius $\phi_A : A \rightarrow A$) and $I \subset A$ is an ideal defining a Cartier divisor in $\text{Spec}(A)$, such that $A$ is derived $(p,I)$-complete and $p \in I + \phi_A(I)A$. In all the cases of interest the ideal $I \subset A$ will be a principal ideal generated by a nonzerodivisor. A prism $(A,I)$ is called bounded if the ring~$A/I$ has bounded~$p^\infty$-torsion. Restricting to bounded prisms avoids complications involving derived completions: for instance, the underlying ring $A$ of a bounded prism $(A,I)$ is $(p,I)$-complete in the classical sense.

Following \cite[Section $4$]{bhatt_topological_2019} or \cite[Appendix C]{bhatt_absolute_2022}, a morphism $R \rightarrow S$ of commutative rings is {\it $p$\nobreakdash-quasisyntomic} if it is $p$\nobreakdash-flat and $\mathbb{L}_{(S/p)/(R/p)}$ has Tor-amplitude in $[-1;0]$. For instance, any noetherian local complete intersection ring $S$ is $p$\nobreakdash-quasisyntomic over $\Z$ by \cite[Theorem~$1.2$]{avramov_locally_1999}.

Given a bounded prism $(A,I)$ and a $p$\nobreakdash-quasisyntomic $A/I$-algebra $S$, \cite[Section $4.1$]{bhatt_prisms_2022} defines the prismatic site $(S/A)_\Prism$ as the site having as objects the prisms $(B,IB)$ over $(A,I)$ with a map $S \rightarrow B/IB$, and covers given by flat covers. The relative prismatic complex $\Prism_{S/A} \in \mathcal{D}(A)$ is defined as the cohomology of the sheaf $$\mathcal{O}_\Prism : (S/A)_\Prism \longrightarrow A\text{-Alg}, \text{ }(B,IB) \longmapsto B.$$ Similarly, the Hodge--Tate complex $\overline{\Prism}_{S/A} \in \mathcal{D}(A/I)$ is defined as the cohomology of the sheaf $$\overline{\mathcal{O}}_\Prism : (S/A)_\Prism \longrightarrow A/I\text{-Alg}, \text{ }(B,IB) \longmapsto B/IB.$$ For instance, if $k$ is a perfect field of characteristic $p$ and $S$ is a ($p$\nobreakdash-)quasisyntomic $k$\nobreakdash-algebra, then $(W(k),(p))$ is an example of bounded prism and $\Prism_{S/W(k)}$ recovers the crystalline cohomology of $S$ (\cite[Section $4.6$]{bhatt_absolute_2022}). For a general $A/I$-algebra $S$, prismatic and Hodge--Tate complexes are defined by left Kan extension from the smooth case. The compatibility between the two definitions in the $p$\nobreakdash-quasisyntomic case is proved in \cite[Section $4$]{bhatt_absolute_2022}. Note in particular that $p$\nobreakdash-Cartier smooth algebras are not necessarily $p$\nobreakdash-quasisyntomic, as they are not necessarily $p$-flat.

The prismatic complex $\Prism_{S/A}$ naturally bears a $\phi_A$-semilinear Frobenius endomorphism $$\phi : \Prism^{(1)}_{S/A} \longrightarrow \Prism_{S/A},$$ where $\Prism^{({1})}_{S/A} := \Prism_{S/A} \widehat{\otimes}_{A,\phi_A}^\mathbb{L} A$ is the Frobenius-twisted prismatic complex. Our main result on the prismatic cohomology of $p$\nobreakdash-Cartier smooth algebras describes the image of this Frobenius endomorphism~$\phi$, and is formulated in terms of the functor $L\eta_I$. Following \cite[Section $6$]{bhatt_integral_2018} and for $A$ a commutative ring, $I \subset A$ an ideal defining a Cartier divisor in $\text{Spec}(A)$ and $C \in \mathcal{D}(A)$ a complex, the object $L\eta_I C \in \mathcal{D}(A)$ is defined as follows. The complex $C$ in the derived category $\mathcal{D}(A)$ is represented by a complex $(C^\bullet, d)$ of $A$-modules such that for all $i \in \Z$, $C^i$ is $I$-torsionfree \hbox{({\it i.e.}, $C^i[I]=0$)}. Define the complex $\eta_IC^\bullet$ with terms
$$\eta_I C^i := \{x \in I^iC^i \text{ } \vert \text{ } dx \in I^{i+1}C^{i+1}\}$$ and differential induced by the differential of $C^\bullet$. As an object of the derived category $\mathcal{D}(A)$ this construction does not depend on the choice of $(C^\bullet,d)$ (\cite[Corollary~$6.5$]{bhatt_integral_2018}), and we call this object~$L\eta_IC \in \mathcal{D}(A)$. Following \cite[Section $5$]{bhatt_topological_2019}, an object $\text{Fil}^\star\, C$ of the filtered derived category~$\mathcal{DF}(A)$ is called {\it connective for the Beilinson $t$-structure} if the graded pieces $\text{gr}^i\, C \in \mathcal{D}(A)$ are in degree $\leq i$ for all $i \in \Z$. The {\it connective cover for the Beilinson $t$-structure} of a filtered complex $\text{Fil}^\star\, C$ is the universal connective filtered complex with a map to $\text{Fil}^\star\, C$.

\needspace{5\baselineskip}

\begin{proposition}\label{LemmaLetaI}
    Let $A$ be a commutative ring, $I\subset A$ an ideal defining a Cartier divisor in $\emph{Spec}(A)$ and $C$ an object in the derived category $\mathcal{D}(A)$.
    \begin{enumerate}
        \item (\cite[Lemma~$6.4$]{bhatt_integral_2018}) For each integer $i \in \Z$, there is a canonical isomorphism of $A$-modules $$\emph{H}^i(L\eta_IC) \cong (\emph{H}^i(C)/\emph{H}^i(C)[I]) \otimes_A I^i.$$
        \item (\cite[Proposition~$5.8$]{bhatt_topological_2019}) $L\eta_I C \in \mathcal{D}(A)$ is the complex underlying the connective cover for the Beilinson $t$-structure of the $I$-adically filtered object $I^\star C \in \mathcal{DF}(A)$.
        \item (\cite[Proposition~$6.12$]{bhatt_integral_2018}) There is a natural equivalence
        $$(L\eta_I C) \otimes_A^{\mathbb{L}} A/I \xlongrightarrow{\sim} \emph{H}^\ast(C/I)$$
        in the derived category $\mathcal{D}(A/I)$, where the differential on $\emph{H}^\ast(C/I)$ is the Bockstein operator induced by the $I$-adic filtration on $C$.
    \end{enumerate}
\end{proposition}

The upshot is that the image of the Frobenius $\phi : \Prism^{(1)}_{S/A} \rightarrow \Prism_{S/A}$ will be identified with $L\eta_I \Prism_{S/A}$, when $S$ is a $p$\nobreakdash-Cartier smooth $A/I$-algebra. By the previous proposition, the complex $L\eta_I \Prism_{S/A}$ is characterised by a universal property in terms of the $I$-adic filtration on the prismatic complex $\Prism_{S/A}$. Following \cite[Section $15$]{bhatt_prisms_2022}, define the Nygaard filtration $\mathcal{N}^\star \Prism^{(1)}_{S/A}$ on the Frobenius-twisted prismatic complex $\Prism^{(1)}_{S/A}$ as the preimage\footnote{To make sense of this preimage, one needs to restrict to large $p$\nobreakdash-quasisyntomic $A/I$-algebras --for which the prismatic complex $\Prism_{S/A}$ is concentrated in degree $0$--, and then generalise the definition using descent on the $p$\nobreakdash-quasisyntomic site and left Kan extension.} of the $I$-adic filtration $I^\star \Prism_{S/A}$ via the morphism $\phi$. The Frobenius $$\phi : \Prism^{(1)}_{S/A} \longrightarrow \Prism_{S/A}$$ naturally upgrades to a map of filtered complexes 
\begin{equation}\label{equationfilteredFrobenius}
    \phi : \mathcal{N}^{\geq \star} \Prism^{(1)}_{S/A} \longrightarrow I^\star \Prism_{S/A}
\end{equation}
in the filtered derived category $\mathcal{DF}(A)$. Beware that the filtration $\mathcal{N}^{\geq \star} \Prism^{(1)}_{S/A}$ is in general not complete, and we denote by $\widehat{\Prism}^{(1)}_{S/A} \in \mathcal{D}(A)$ (resp. $\mathcal{N}^i \Prism^{(1)}_{S/A} \in \mathcal{D}(A/I)$, for $i \geq 0$) the Nygaard\nobreakdash-completed prismatic complex (resp. the Nygaard graded pieces). On the Nygaard graded pieces, the Frobenius~(\ref{equationfilteredFrobenius}) can be rewritten as a map $$\phi : \mathcal{N}^i \Prism^{(1)}_{S/A} \longrightarrow \overline{\Prism}_{S/A}\{i\}$$ for each $i \geq 0$, where $$\overline{\Prism}_{S/A}\{i\} := \overline{\Prism}_{S/A} \otimes_{A/I} (I/I^2)^{\otimes i}.$$ To describe the image of this Frobenius map, define the conjugate filtration $\text{Fil}_\star^{\text{conj}}\,\overline{\Prism}_{S/A}\{i\}$ on the Hodge--Tate complex $\overline{\Prism}_{S/A}\{i\}$ as the left Kan extension from smooth $A/I$-algebras of the increasing Postnikov filtration $\tau^{\leq \star} \overline{\Prism}_{-/A}\{i\}$.

\needspace{5\baselineskip}

\begin{theorem}[\cite{bhatt_prisms_2022, bhatt_absolute_2022}]\label{TheoremBS19qsyn}
    Let $(A,I)$ be a bounded prism, and $S$ an $A/I$-algebra.
    \begin{enumerate}
        \item For each $i \geq 0$, the Frobenius map $\phi : \mathcal{N}^{i} \Prism^{(1)}_{S/A} \rightarrow \overline{\Prism}_{S/A}\{i\}$ factors as
        $$\mathcal{N}^i \Prism^{(1)}_{S/A} \xlongrightarrow{\tilde{\phi}} \emph{Fil}_i^\emph{conj}\,\overline{\Prism}_{S/A}\{i\} \xlongrightarrow{\emph{can}} \overline{\Prism}_{S/A}\{i\},$$
        and the map $\tilde{\phi} : \mathcal{N}^i \Prism^{(1)}_{S/A} \rightarrow \emph{Fil}_i^\emph{conj}\,\overline{\Prism}_{S/A}\{i\}$ is an equivalence. In particular, the complex $\mathcal{N}^i \Prism^{(1)}_{S/A} \in \mathcal{D}(A/I)$ is concentrated in degrees $\leq i$.
        \item The Frobenius map $\phi : \Prism^{(1)}_{S/A} \rightarrow \Prism_{S/A}$ factors as $$\Prism^{(1)}_{S/A} \xlongrightarrow{\tilde{\phi}} L\eta_I \Prism_{S/A} \xlongrightarrow{\emph{can}} \Prism_{S/A}.$$
    \end{enumerate}
\end{theorem}

\begin{proof}
    The first part of $(1)$ is a special case of \cite[Proposition~$5.1.1$ and Remark~$5.1.3$]{bhatt_absolute_2022}, where the result is formulated for a general animated commutative $A/I$-algebra $S$. All the objects are left Kan extended from the smooth case, where the result is also given by \cite[Theorem~$15.3$]{bhatt_prisms_2022}. The functor $\tau^{\leq i} \overline{\Prism}_{-/A}\{i\}$ is concentrated in degrees $\leq i$, so its left Kan extension (defined as a sifted colimit) is also concentrated in degrees $\leq i$.
    
    $(2)$ The filtered complex $\mathcal{N}^{\geq \star} \Prism^{(1)}_{S/A} \in \mathcal{DF}(A)$ is connective for the Beilinson $t$-structure by $(1)$. So the map $\phi : \mathcal{N}^{\geq \star} \Prism^{(1)}_{S/A} \rightarrow I^\star \Prism_{S/A}$ factors through the connective cover for the Beilinson $t$-structure of the target. The result follows from Proposition~\ref{LemmaLetaI}\,$(2)$, by looking at the underlying complexes.
\end{proof}

We recall some important features of the relative prismatic complex in the smooth case, following~\cite{bhatt_prisms_2022}.

\needspace{5\baselineskip}

\begin{theorem}[Prismatic cohomology in the smooth case, \cite{bhatt_prisms_2022}]\label{TheoremBS19smooth}
    Let $(A,I)$ be a bounded prism, and $S$ a smooth $A/I$-algebra.
    \begin{enumerate}
        \item ($L\eta_I$ comparison) The Frobenius map $$\tilde{\phi} : \Prism^{(1)}_{S/A} \xlongrightarrow{\sim} L\eta_I \Prism_{S/A}$$
        of Theorem~\ref{TheoremBS19qsyn}\,$(2)$ is an equivalence in the derived category $\mathcal{D}(A)$.
        \item (Hodge--Tate comparison) There is a canonical isomorphism 
        $$(\Omega^\ast_{S/(A/I)})^\wedge_p \xlongrightarrow{\cong} \emph{H}^\ast(\overline{\Prism}_{S/A}\{\ast\}):= \bigoplus_{i \geq 0} \emph{H}^i(\overline{\Prism}_{S/A}\{i\})$$ 
        of differential graded $A/I$-algebras, where the differential on $\emph{H}^\ast(\overline{\Prism}_{S/A}\{\ast\})$ is the Bockstein operator induced by the $I$-adic filtration on $\Prism_{S/A}$.
        \item (de Rham comparison) There is a canonical equivalence $$\Prism^{(1)}_{S/A} \otimes_A^\mathbb{L} A/I \xlongrightarrow{\sim} (\Omega_{S/(A/I)})^\wedge_p$$ in the derived category $\mathcal{D}(A/I)$.
    \end{enumerate}
\end{theorem}

\begin{proof}
    By \cite[Corollary~$4.1.14$]{bhatt_absolute_2022}, the prismatic cohomology of $S$ is the same as the prismatic cohomology of the $p$\nobreakdash-adic completion of $S$. $(1)$ and $(2)$ then respectively follow from \cite[Theorem~$15.3$]{bhatt_prisms_2022} and \cite[Theorem~$4.11$]{bhatt_prisms_2022}. Using Proposition~\ref{LemmaLetaI}\,$(3)$, $(3)$ is a consequence of $(1)$ and $(2)$.
\end{proof}

\subsection{Prismatic cohomology of Cartier smooth algebras}

\vspace{-\parindent}
\hspace{\parindent}

In this subsection we prove Theorem~\ref{TheoremCartiersmoothMainintro} (Theorem~\ref{TheoremCartiersmoothMain}), characterising $p$\nobreakdash-Cartier smooth algebras in terms of their prismatic cohomology. 

Any $p$\nobreakdash-Cartier smooth algebra is $p$\nobreakdash-cotangent smooth by definition. We first extend some properties of smooth algebras to general $p$\nobreakdash-cotangent smooth algebras. Given a morphism $R \rightarrow S$ of commutative rings, denote by $\widehat{\mathbb{L}\Omega}_{S/R}$ the Hodge-completed derived de Rham complex, and by $(\widehat{\mathbb{L}\Omega}_{S/R})^\wedge_p$ its $p$\nobreakdash-adic completion.

\begin{proposition}\label{propositionpquasismoothproperties}
    Let $(A,I)$ be a bounded prism, and $S$ a $p$\nobreakdash-cotangent smooth $A/I$-algebra.
    \begin{enumerate}
        \item The canonical map $(\widehat{\mathbb{L}\Omega}_{S/(A/I)})^\wedge_p \rightarrow (\Omega_{S/(A/I)})^\wedge_p$ is an equivalence in the derived category $\mathcal{D}(A/I)$.
        \item The conjugate filtration $\emph{Fil}^\emph{conj}_\star\, \overline{\Prism}_{S/A}$ on the Hodge--Tate complex $\overline{\Prism}_{S/A} \in \mathcal{D}(A/I)$ coincides with the Postnikov filtration $\tau^{\leq \star} \overline{\Prism}_{S/A}$. In particular for each $i \geq 0$, the Frobenius map $$\phi : \mathcal{N}^i \Prism^{(1)}_{S/A} \longrightarrow \overline{\Prism}_{S/A} \{i\}$$ induces an equivalence
        $$\tilde{\phi} : \mathcal{N}^i \Prism^{(1)}_{S/A} \xlongrightarrow{\sim} \tau^{\leq i} \overline{\Prism}_{S/A}\{i\}$$
        in the derived category $\mathcal{D}(A/I)$.
        \item There is a canonical isomorphism
        $$(\Omega^\ast_{S/(A/I)})^\wedge_p \xlongrightarrow{\cong} \emph{H}^\ast(\overline{\Prism}_{S/A}\{\ast\})$$
        of differential graded $A/I$-algebras, where the differential on $\emph{H}^\ast(\overline{\Prism}_{S/A}\{\ast\})$ is the Bockstein operator induced by the $I$-adic filtration on $\Prism_{S/A}$.
        \item The Frobenius map $\tilde{\phi} : \Prism^{(1)}_{S/A} \rightarrow L\eta_I \Prism_{S/A}$ of Theorem~\ref{TheoremBS19qsyn}\,$(2)$ factors as
        $$\Prism^{(1)}_{S/A} \xlongrightarrow{\emph{can}} \widehat{\Prism}^{(1)}_{S/A} \xlongrightarrow{\tilde{\phi}} L\eta_I \Prism_{S/A},$$
        and the map $\tilde{\phi} : \widehat{\Prism}^{(1)}_{S/A} \rightarrow L\eta_I \Prism_{S/A}$ is an equivalence in the derived category $\mathcal{D}(A)$.
    \end{enumerate}
\end{proposition}

\begin{proof}
    $(1)$ By the derived Nakayama lemma (\cite[091N]{stacks_project_authors_stacks_2019}), it suffices to prove the result after derived reduction modulo $p$. By base change for the Hodge-completed derived de Rham complex and $p$\nobreakdash-discreteness, this is equivalent to proving that the canonical map $$\widehat{\mathbb{L}\Omega}_{(S/p)/(A/(p,I))} \longrightarrow \Omega_{(S/p)/(A/(p,I))}$$ is an equivalence in the derived category $\mathcal{D}(A/(p,I))$. Both sides are complete for the Hodge filtration, so it suffices to prove that this canonical map is an equivalence on the Hodge graded pieces. The corresponding map $$\mathbb{L}^n_{(S/p)/(A/(p,I))}[-n] \longrightarrow \Omega^n_{(S/p)/(A/(p,I))}[-n]$$ is the shift of a wedge power of the counit map $$\mathbb{L}_{(S/p)/(A/(p,I))} \longrightarrow \Omega^1_{(S/p)/(A/(p,I))}[0]$$ for all $n \geq 0$, which is an equivalence by $p$\nobreakdash-cotangent smoothness of the morphism $A/I \rightarrow S$.
    
    $(2)$ By left Kan extension of the Hodge--Tate comparison Theorem~\ref{TheoremBS19smooth}\,$(2)$ (or \cite[Remark~$4.1.7$]{bhatt_absolute_2022}), there is a canonical equivalence
    $$(\mathbb{L}^n_{S/(A/I)})^\wedge_p[-n]\{-n\} \xlongrightarrow{\sim} \text{gr}^\text{conj}_n\, \overline{\Prism}_{S/A}$$
    in the derived category $\mathcal{D}(A/I)$, for each $n\geq 0$. By the derived Nakayama lemma (\cite[091N]{stacks_project_authors_stacks_2019}) and $p$\nobreakdash-cotangent smoothness of the morphism $A/I \rightarrow S$, the complex $(\mathbb{L}^n_{S/(A/I)})^\wedge_p \in \mathcal{D}(A/I)$ is in degree $0$ and the conjugate graded piece $\text{gr}^\text{conj}_n\, \overline{\Prism}_{S/A}$ is thus in degree $n$. By induction on $n\geq 0$ and using the long exact sequence in cohomology groups associated to the homotopy cofibre sequence $$\text{Fil}^\text{conj}_{n-1}\, \overline{\Prism}_{S/A} \longrightarrow \text{Fil}^\text{conj}_n\, \overline{\Prism}_{S/A} \longrightarrow \text{gr}^\text{conj}_n\, \overline{\Prism}_{S/A},$$
    we deduce that $\text{Fil}^\text{conj}_n\, \overline{\Prism}_{S/A}$ belongs to $\mathcal{D}^{[0;n]}(A/I)$, that the induced morphism $$\text{Fil}^\text{conj}_{n-1}\, \overline{\Prism}_{S/A} \longrightarrow \tau^{\leq n-1} \text{Fil}^\text{conj}_n\, \overline{\Prism}_{S/A}$$ is an equivalence and that the morphism $$\text{H}^n(\text{Fil}^\text{conj}_n\,\overline{\Prism}_{S/A}) \longrightarrow \text{H}^n(\text{gr}^\text{conj}_n\, \overline{\Prism}_{S/A})$$ is an isomorphism of $A/I$-modules. The conjugate filtration $\text{Fil}^\text{conj}_\star\, \overline{\Prism}_{S/A}$ is moreover exhaustive, as a left Kan extension of the exhaustive Postnikov filtration $\tau^{\leq \star} \overline{\Prism}_{-/A}$. So the canonical map $$\text{Fil}^\text{conj}_n\, \overline{\Prism}_{S/A} \longrightarrow \tau^{\leq n} \lim\limits_{k} (\text{Fil}^\text{conj}_{n+k}\, \overline{\Prism}_{S/A}) \simeq \tau^{\leq n} \overline{\Prism}_{S/A}$$ is an equivalence for each $n \geq 0$, and the conjugate filtration $\text{Fil}^\text{conj}_\star\, \overline{\Prism}_{S/A}$ coincides with the Postnikov filtration $\tau^{\leq \star} \overline{\Prism}_{S/A}$. The last statement is a consequence of Theorem~\ref{TheoremBS19qsyn}\,$(1)$.
    
    $(3)$ By left Kan extension of the Hodge--Tate comparison Theorem~\ref{TheoremBS19smooth}\,$(2)$ (or \cite[Remark~$4.1.7$]{bhatt_absolute_2022}) and $(2)$, there is for each $n\geq 0$ a canonical isomorphism
    $$(\Omega^n_{S/(A/I)})^\wedge_p \xlongrightarrow{\cong} \text{H}^n(\overline{\Prism}_{S/A}\{n\})$$
    of $A/I$-modules. To prove that the de Rham differential $d$ coincides via these isomorphisms with the Bockstein operator $\beta$ on $\text{H}^\ast(\overline{\Prism}_{S/A}\{\ast\})$, it suffices to prove it when $S$ is a polynomial $A/I$-algebra, where this is Theorem~\ref{TheoremBS19smooth}\,$(2)$.
    
    $(4)$ For each integer $k \geq 0$, the map of filtered complexes $\phi : \mathcal{N}^{\geq \star} \Prism^{(1)}_{S/A} \rightarrow I^\star \Prism_{S/A}$ induces a map of filtered complexes $$\phi : \mathcal{N}^{\geq \star} \Prism^{(1)}_{S/A} / \mathcal{N}^{\geq \star + k} \Prism^{(1)}_{S/A} \longrightarrow I^\star \Prism_{S/A} / I^{\star + k} \Prism_{S/A} \simeq \Prism_{S/A} \otimes_A^{\mathbb{L}} I^\star/I^{\star+k}.$$
    Taking the inverse limit over $k\geq 0$ defines a map of filtered complexes $$\phi : \mathcal{N}^{\geq \star} \widehat{\Prism}^{(1)}_{S/A} \longrightarrow I^\star \Prism_{S/A}.$$
    The canonical map $\mathcal{N}^{\geq \star} \Prism^{(1)}_{S/A} \rightarrow \mathcal{N}^{\geq \star} \widehat{\Prism}^{(1)}_{S/A}$ is an equivalence on graded pieces, so the same argument as in Theorem~\ref{TheoremBS19qsyn}\,$(2)$ proves that $$\Prism^{(1)}_{S/A} \xlongrightarrow{\text{can}} \widehat{\Prism}^{(1)}_{S/A} \xlongrightarrow{\phi} \Prism_{S/A}$$
    factors as $$\Prism^{(1)}_{S/A} \xlongrightarrow{\text{can}} \widehat{\Prism}^{(1)}_{S/A} \xlongrightarrow{\tilde{\phi}} L\eta_I \Prism_{S/A} \xlongrightarrow{\text{can}} \Prism_{S/A},$$
    where the composite $\tilde{\phi} : \Prism^{(1)}_{S/A} \rightarrow L\eta_I \Prism_{S/A}$ of the first two maps is the map of Theorem~\ref{TheoremBS19qsyn}\,$(2)$. Remark that the $I$-adic filtration on $L\eta_I \Prism_{S/A}$ is complete by \cite[Remark~$D.10$]{bhatt_absolute_2022}. To prove that $\tilde{\phi} : \widehat{\Prism}^{(1)}_{S/A} \rightarrow L\eta_I \Prism_{S/A}$ is an equivalence, it suffices by completeness to prove that it is an equivalence on graded pieces. Proposition~\ref{LemmaLetaI}\,$(2)$ (more precisely, \cite[Theorem~$5.4\,(2)$]{bhatt_topological_2019}) identifies the \hbox{$i^\text{th}$ graded} piece of $L\eta_I \Prism_{S/A}$ with the truncation $\tau^{\leq i}$ of the $I$-adic graded piece $\overline{\Prism}_{S/A}\{i\}$. Together with the identification $\mathcal{N}^i \Prism^{(1)}_{S/A} \xrightarrow{\sim} \mathcal{N}^i \widehat{\Prism}^{(1)}_{S/A}$, it then suffices to prove that $$\tilde{\phi} : \mathcal{N}^i \Prism^{(1)}_{S/A} \longrightarrow \tau^{\leq i} \overline{\Prism}_{S/A} \{i\}$$ is an equivalence for all $i \geq 0$, which is $(2)$.
\end{proof}

\begin{remark}
    The factorisation part of Proposition~\ref{propositionpquasismoothproperties}\,$(4)$ and its proof hold more generally for $(A,I)$ a bounded prism and $S$ any $A/I$-algebra.
\end{remark}

\begin{theorem}\label{TheoremCartiersmoothMain}
    Let $(A,I)$ be a bounded prism, and $S$ a $p$\nobreakdash-cotangent smooth $A/I$-algebra. The following are equivalent:
    \begin{description}[align=left]
        \item[$(CSm)$] The inverse Cartier map $$\Omega^n_{(S/p)/(A/(p,I))} \otimes_{A/(p,I),\phi_{A/(p,I)}} A/(p,I) \xlongrightarrow{C^{-1}} \emph{H}^n(\Omega^\bullet_{(S/p)/(A/(p,I))})$$ is an isomorphism of $A/(p,I)$-modules for all $n\geq 0$, {\it i.e.}, $S$ is $p$\nobreakdash-Cartier smooth over $A/I$.
        \item[$(\mathbb{L}\Omega=\widehat{\mathbb{L}\Omega})$] The Hodge-completion map $(\mathbb{L}\Omega_{S/(A/I)})^\wedge_p \rightarrow (\widehat{\mathbb{L}\Omega}_{S/(A/I)})^\wedge_p$ is an equivalence in the derived category $\mathcal{D}(A/I)$.
        \item[$(\mathbb{L}\Omega = \Omega)$] The counit map $(\mathbb{L}\Omega_{S/(A/I)})^\wedge_p \rightarrow (\Omega_{S/(A/I)})^\wedge_p$ is an equivalence in the derived category $\mathcal{D}(A/I)$.
        \item[$(dR)$] The de Rham comparison map $\Prism^{(1)}_{S/A} \otimes^\mathbb{L}_A A/I \rightarrow (\Omega_{S/(A/I)})^\wedge_p$ is an equivalence in the derived category $\mathcal{D}(A/I)$.
        \item[$(\overline{\Prism})$] The canonical map $\Prism^{(1)}_{S/A} \otimes_A^\mathbb{L} A/I \rightarrow \emph{H}^\ast(\overline{\Prism}_{S/A}\{\ast\})$ is an equivalence in the derived category $\mathcal{D}(A/I)$.
        \item[$(\Prism=\widehat{\Prism})$] The Nygaard-completion map $\Prism^{(1)}_{S/A} \rightarrow \widehat{\Prism}^{(1)}_{S/A}$ is an equivalence in the derived category $\mathcal{D}(A)$.
        \item[$(L\eta)$] The Frobenius map $\tilde{\phi} : \Prism^{(1)}_{S/A} \rightarrow L\eta_I \Prism_{S/A}$ is an equivalence in the derived category $\mathcal{D}(A)$.
        \item[$(\mathcal{N})$] The canonical map $\emph{H}^0_\emph{B} : \Prism^{(1)}_{S/A} \otimes_A^\mathbb{L} A/I \rightarrow \emph{H}^\ast(\mathcal{N}^\ast\Prism^{(1)}_{S/A})$ is an equivalence in the derived category $\mathcal{D}(A/I)$, where the differential on $\emph{H}^\ast(\mathcal{N}^\ast\Prism^{(1)}_{S/A})$ is the Bockstein operator induced by the Nygaard filtration on $\Prism^{(1)}_{S/A}$.
        \item[$(\mathcal{N}^{\geq})$] The Frobenius map $\tau^{\leq i} \phi : \tau^{\leq i} \mathcal{N}^{\geq i} \Prism^{(1)}_{S/A} \rightarrow \tau^{\leq i} I^i \Prism_{S/A}$ is an equivalence in the derived category $\mathcal{D}(A)$ for all $i \geq 0$.
    \end{description}
\end{theorem}

\begin{proof}
    $(CSm)\Leftrightarrow(\mathbb{L}\Omega=\Omega)$ By derived Nakayama (\cite[091N]{stacks_project_authors_stacks_2019}), the counit map $$(\mathbb{L}\Omega_{S/(A/I)})^\wedge_p \longrightarrow (\Omega_{S/(A/I)})^\wedge_p$$ is an equivalence in the derived category $\mathcal{D}(A/I)$ if and only if the counit map $$\mathbb{L}\Omega_{(S/p)/(A/(p,I))} \longrightarrow \Omega_{(S/p)/(A/(p,I))}$$ is an equivalence in the derived category $\mathcal{D}(A/(p,I))$, {\it i.e.}, if the induced map $$\text{H}^n(\mathbb{L}\Omega_{(S/p)/(A/(p,I))}) \longrightarrow \text{H}^n(\Omega_{(S/p)/(A/(p,I))})$$ is an isomorphism of $A/(p,I)$-modules for each $n \geq 0$. By \cite[Proposition~$3.5$]{bhatt_p-adic_2012}, the derived de Rham complex $\mathbb{L}\Omega_{(S/p)/(A/(p,I))}$ is equipped with an exhaustive $\mathbb{N}$-indexed increasing filtration $\text{Fil}^\text{conj}_\star\, \mathbb{L}\Omega_{(S/p)/(A/(p,I))}$, whose graded pieces are given by $$C^{-1} : \wedge^n (\mathbb{L}_{(S/p)/(A/(p,I))} \otimes_{A/(p,I), \phi_{A/(p,I)}}^\mathbb{L} A/(p,I))[-n] \xlongrightarrow{\sim} \text{gr}^\text{conj}_n\, \mathbb{L}\Omega_{(S/p)/(A/(p,I))},$$ where $C^{-1}$ is the left Kan extension of the inverse Cartier map. By cotangent smoothness of the morphism $A/(p,I) \rightarrow S/p$, the graded piece $\text{gr}^\text{conj}_n\, \mathbb{L}\Omega_{(S/p)/(A/(p,I))}$ is thus concentrated in degree $n$ for each~$n \geq 0$. Arguing by induction on $n\geq 0$ and by exhaustiveness of the conjugate filtration $\text{Fil}^\text{conj}_\star\, \mathbb{L}\Omega_{(S/p)/(A/(p,I))}$, there is a canonical isomorphism $$\text{H}^n(\text{gr}^\text{conj}_n\, \mathbb{L}\Omega_{(S/p)/(A/(p,I))}) \xlongrightarrow{\cong} \text{H}^n(\mathbb{L}\Omega_{(S/p)/(A/(p,I))}).$$ In particular the counit map $\mathbb{L}\Omega_{(S/p)/(A/(p,I))} \rightarrow \Omega_{(S/p)/(A/(p,I))}$ is an equivalence in the derived category $\mathcal{D}(A/(p,I))$ if and only if the inverse Cartier map 
    $$C^{-1} : \Omega^n_{(S/p)/(A/(p,I))} \otimes_{A/(p,I),\phi_{A/(p,I)}} A/(p,I) \longrightarrow \text{H}^n(\Omega^\bullet_{(S/p)/(A/(p,I))})$$
    is an isomorphism of $A/(p,I)$-modules for all $n\geq 0$.
    
    $(\mathbb{L}\Omega = \widehat{\mathbb{L}\Omega})\Leftrightarrow(\mathbb{L}\Omega=\Omega)\Leftrightarrow(dR)\Leftrightarrow(\overline{\Prism})\Leftrightarrow(\Prism=\widehat{\Prism})\Leftrightarrow(L\eta)$ By \cite[Proposition~$5.2.3$]{bhatt_absolute_2022} there is a commutative diagram
    $$\begin{tikzcd}        
     & \widehat{\Prism}^{(1)}_{S/A} \ar[ddrr,"\tilde{\phi}","\sim"'] \ar[ddd,dashed] & & \\
    \Prism^{(1)}_{S/A} \ar[ur,"\text{can}"] \ar[drrr, "\tilde{\phi}"] \ar[ddd,dashed] & & & \\
     & & & L\eta_I\Prism_{S/A} \ar[ddd,dashed] \\
     & (\widehat{\mathbb{L}\Omega}_{S/(A/I)})^\wedge_p \ar[dr,"\text{can}","\sim"'] & & \\
    (\mathbb{L}\Omega_{S/(A/I)})^\wedge_p \ar[rr] \ar[ur,"\text{can}"] \ar[rrrd] & & (\Omega_{S/(A/I)})^\wedge_p \ar[dr,"\text{HT}","\sim"'] & \\
     & & & \text{H}^\ast(\overline{\Prism}_{S/A}\{\ast\})
    \end{tikzcd}$$
    in the derived category $\mathcal{D}(A)$, where the dashed arrows are derived reduction modulo $I$ (\cite[Proposition~$5.2.5$ and Corollary~$5.2.8$]{bhatt_absolute_2022} and Proposition~\ref{LemmaLetaI}\,$(3)$). The map $(\mathbb{L}\Omega_{S/(A/I)})^\wedge_p \rightarrow \text{H}^\ast(\overline{\Prism}_{S/A}\{\ast\})$ is defined as the composite $$(\mathbb{L}\Omega_{S/(A/I)})^\wedge_p \longrightarrow (\Omega_{S/(A/I)})^\wedge_p \longrightarrow \text{H}^\ast(\overline{\Prism}_{S/A}\{\ast\}).$$ The three equivalences in this diagram hold for any $p$\nobreakdash-cotangent smooth $A/I$-algebra $S$ by Proposition~\ref{propositionpquasismoothproperties}. By derived Nakayama (\cite[091N]{stacks_project_authors_stacks_2019}) and commutativity of this diagram, the conditions $(\mathbb{L}\Omega=\widehat{\mathbb{L}\Omega})$, $(\mathbb{L}\Omega=\Omega)$, $(dR)$, $(\overline{\Prism})$, $(L\eta)$ and $(\Prism=\widehat{\Prism})$ are then equivalent.
    
    $(L\eta)\Leftrightarrow(\mathcal{N})$ There is a commutative diagram
    $$\begin{tikzcd}        
    \Prism^{(1)}_{S/A} \otimes_A^\mathbb{L} A/I \arrow["\text{H}^0_\text{B}"]{r} \arrow["\tilde{\phi}"]{d} & \text{H}^\ast(\mathcal{N}^\ast\Prism^{(1)}_{S/A}) \ar[d,"\tilde{\phi}","\sim"'] \\
    (L\eta_I \Prism_{S/A}) \otimes_A^\mathbb{L} A/I \arrow["\text{H}^0_\text{B}","\sim"']{r} & \text{H}^\ast(\overline{\Prism}_{S/A}\{\ast\})
    \end{tikzcd}$$
    in the derived category $\mathcal{D}(A/I)$, where the maps $\tilde{\phi}$ are defined in Theorem~\ref{TheoremBS19qsyn}, and the functor $\text{H}^0_\text{B}$ (where B refers to the Beilinson $t$-structure) is defined in \cite[Theorem~$5.4$]{bhatt_topological_2019}. More precisely, the functor $\text{H}^0_\text{B}$ sends the filtered complex $$\mathcal{N}^{\geq \star} \Prism^{(1)}_{S/A} \in \mathcal{DF}(A)$$ to $$\text{H}^\ast(\mathcal{N}^\ast \Prism^{(1)}_{S/A}) \in \mathcal{D}(A) \simeq \mathcal{DF}(A)^\heartsuit,$$
    where $\mathcal{DF}(A)^\heartsuit$ is the heart of the filtered derived category $\mathcal{DF}(A)$ for its Beilinson $t$-structure. Because $\mathcal{N}^\star \Prism^{(1)}_{S/A}$ is naturally an object of the derived category $\mathcal{D}(A/I)$, this induces a map $$\Prism^{(1)}_{S/A} \otimes_A^\mathbb{L} A/I \xlongrightarrow{\text{H}^0_\text{B}} \text{H}^\ast(\mathcal{N}^\ast \Prism^{(1)}_{S/A})$$ on the underlying complexes. For any $p$\nobreakdash-cotangent smooth $A/I$-algebra $S$, the right and bottom maps of the previous diagram are equivalences (Propositions~\ref{propositionpquasismoothproperties}\,$(2)$ and \ref{LemmaLetaI}\,$(3)$). By derived Nakayama and commutativity of the diagram, the Frobenius map $$\tilde{\phi} : \Prism^{(1)}_{S/A} \longrightarrow L\eta_I \Prism_{S/A}$$ is an equivalence in the derived category $\mathcal{D}(A)$ if and only if the canonical map $$\text{H}^0_\text{B} : \Prism^{(1)}_{S/A} \otimes_A^\mathbb{L} A/I \longrightarrow \text{H}^\ast(\mathcal{N}^\ast\Prism^{(1)}_{S/A})$$ is an equivalence in the derived category $\mathcal{D}(A/I)$.
    
    $(\mathcal{N}^\geq)\Rightarrow(L\eta)$ Assume that the Frobenius map $$\tau^{\leq i} \phi : \tau^{\leq i} \mathcal{N}^{\geq i} \Prism^{(1)}_{S/A} \longrightarrow \tau^{\leq i} I^i \Prism_{S/A}$$ is an equivalence for each $i\geq 0$, and in particular that the Frobenius map $$\phi : \text{H}^i(\mathcal{N}^{\geq i} \Prism^{(1)}_{S/A}) \longrightarrow \text{H}^i(I^i \Prism_{S/A})$$ is an isomorphism of $A$-modules for each $i \geq 0$. The homotopy cofibre sequence $$\mathcal{N}^{\geq i} \Prism^{(1)}_{S/A} \longrightarrow \mathcal{N}^{\geq i-1} \Prism^{(1)}_{S/A} \longrightarrow \mathcal{N}^{i-1} \Prism^{(1)}_{S/A}$$ induces an exact sequence
    $$\text{H}^{n-1}(\mathcal{N}^{i-1}\Prism^{(1)}_{S/A}) \longrightarrow \text{H}^n(\mathcal{N}^{\geq i} \Prism^{(1)}_{S/A}) \longrightarrow \text{H}^n(\mathcal{N}^{\geq i-1} \Prism^{(1)}_{S/A}) \longrightarrow \text{H}^n(\mathcal{N}^{i-1}\Prism^{(1)}_{S/A})$$ for each integer $n$. By Proposition~\ref{propositionpquasismoothproperties}\,$(2)$, the graded piece $\mathcal{N}^{i-1}\Prism^{(1)}_{S/A}$ is in degrees $[0;i-1]$ and the Frobenius map $\phi : \mathcal{N}^{i-1} \Prism^{(1)}_{S/A} \rightarrow \overline{\Prism}_{S/A}\{i-1\}$ is an isomorphism in degrees $\leq i-1$. In particular the morphism $\text{H}^n(\mathcal{N}^{\geq i}\Prism^{(1)}_{S/A}) \rightarrow \text{H}^n(\mathcal{N}^{\geq i-1} \Prism^{(1)}_{S/A})$ is an isomorphism for all $i,n\geq 0$ satisfying $i<n$, and the canonical morphism $$\text{H}^i(\mathcal{N}^{\geq i-1} \Prism^{(1)}_{S/A}) \longrightarrow \text{H}^i(\mathcal{N}^{\geq 0}\Prism^{(1)}_{S/A}) = \text{H}^i(\Prism^{(1)}_{S/A})$$ is an isomorphism for all $i \geq 0$. The $A$-module $\text{H}^i(L\eta_I\Prism_{S/A})$ is canonically identified with the image of the morphism $\text{H}^i(I^i\Prism_{S/A}) \rightarrow \text{H}^i(I^{i-1}\Prism_{S/A})$ by Proposition~\ref{LemmaLetaI}\,$(1)$. The Frobenius map $$\phi : \mathcal{N}^{\geq \star} \Prism^{(1)}_{S/A} \longrightarrow I^\star \Prism_{S/A}$$ thus induces a map of short exact sequences
    $$\begin{tikzcd}[column sep = small]
    \text{H}^{i-1}(\mathcal{N}^{\geq i-1} \Prism^{(1)}_{S/A}) \ar[r] \ar[d,"\phi"] & \text{H}^{i-1}(\mathcal{N}^{i-1}\Prism^{(1)}_{S/A}) \ar[r] \ar[d,"\phi"] & \text{H}^i(\mathcal{N}^{\geq i} \Prism^{(1)}_{S/A}) \ar[r] \ar[d,"\phi"] & \text{H}^i(\Prism^{(1)}_{S/A}) \ar[r] \ar[d,"\tilde{\phi}"] & 0 \\
    \text{H}^{i-1}(I^{i-1}\Prism_{S/A}) \ar[r] & \text{H}^{i-1}(\overline{\Prism}_{S/A}\{i-1\}) \ar[r] & \text{H}^i(I^i \Prism_{S/A}) \ar[r] & \text{H}^i(L\eta_I\Prism_{S/A}) \ar[r] & 0.
    \end{tikzcd}$$
    The first three vertical maps are isomorphisms by assumption, thus so is the right vertical map, for each $i\geq 0$. So the Frobenius map $\tilde{\phi} : \Prism^{(1)}_{S/A} \rightarrow L\eta_I \Prism_{S/A}$ is an equivalence.
    
    $(L\eta)\Rightarrow(\mathcal{N}^\geq)$ Assume the Frobenius map $\tilde{\phi} : \Prism^{(1)}_{S/A} \rightarrow L\eta_I\Prism_{S/A}$ is an equivalence. We prove by induction on $i\geq 0$ that for every integer $n \leq i$, the Frobenius morphism $$\phi : \text{H}^n(\mathcal{N}^{\geq i} \Prism^{(1)}_{S/A}) \longrightarrow \text{H}^n(I^i\Prism_{S/A})$$ is an isomorphism of $A$-modules. For $i=0$, it suffices to prove that the Frobenius morphism $$\phi : \text{H}^0(\Prism^{(1)}_{S/A}) \longrightarrow \text{H}^0(\Prism_{S/A})$$ is an isomorphism, or equivalently that the canonical map $$\text{H}^0(L\eta_I \Prism_{S/A}) \cong \text{H}^0(\Prism_{S/A}) / \text{H}^0(\Prism_{S/A})[I] \longrightarrow \text{H}^0(\Prism_{S/A})$$ is an isomorphism. For any prism $(B,IB)$ over $(A,I)$, the $A$-algebra $B$ is $I$-torsionfree \cite[Lemma~$3.5$]{bhatt_prisms_2022}, so
    $$\text{H}^0(\Prism_{S/A}) = \lim\limits_{(B,IB) \in (S/A)_\Prism} B$$
    is also $I$-torsionfree (\cite[Theorem~$4.3.6$]{bhatt_absolute_2022}) and the morphism $\text{H}^0(L\eta_I \Prism_{S/A}) \rightarrow \text{H}^0(\Prism_{S/A})$ is the identity. Assume $i$ is now a nonnegative integer for which the result holds. Using the equivalence $$\tilde{\phi} : \Prism^{(1)}_{S/A} \xlongrightarrow{\sim} L\eta_I \Prism_{S/A},$$ the first, second and fourth vertical maps of the previous diagram are isomorphisms, so the morphism $$\phi : \text{H}^i(\mathcal{N}^{\geq i}\Prism^{(1)}_{S/A}) \longrightarrow \text{H}^i(I^i\Prism_{S/A})$$ is also an isomorphism. For $n < i$, the Frobenius $\phi : \mathcal{N}^{\geq \star} \Prism^{(1)}_{S/A} \rightarrow I^\star \Prism_{S/A}$ induces a map of exact sequences
    
    $$\begin{tikzcd}[column sep = tiny]
    \hspace{-0.6cm}
    \text{H}^{n-1}(\mathcal{N}^{\geq i} \Prism^{(1)}_{S/A}) \ar[r] \ar[d,"\phi","\cong"'] & \text{H}^{n-1}(\mathcal{N}^{i}\Prism^{(1)}_{S/A}) \ar[r] \ar[d,"\phi","\cong"'] & \text{H}^{n}(\mathcal{N}^{\geq i+1} \Prism^{(1)}_{S/A}) \ar[r] \ar[d,"\phi"] & \text{H}^{n}(\mathcal{N}^{\geq i}\Prism^{(1)}_{S/A}) \ar[r] \ar[d,"\phi","\cong"'] & \text{H}^n(\mathcal{N}^{i}\Prism^{(1)}_{S/A}) \ar[d,"\phi","\cong"'] \\
    \hspace{-0.6cm}
    \text{H}^{n-1}(I^{i-1}\Prism_{S/A}) \ar[r] & \text{H}^{n-1}(\overline{\Prism}_{S/A}\{i\}) \ar[r] & \text{H}^n(I^{i+1} \Prism_{S/A}) \ar[r] & \text{H}^n(I^{i}\Prism_{S/A}) \ar[r] & \text{H}^n(\overline{\Prism}_{S/A}\{i\})
    \end{tikzcd}$$
    where the isomorphisms are given by Proposition~\ref{propositionpquasismoothproperties}\,$(2)$ and the induction hypothesis. It follows that the morphism $\phi : \text{H}^n(\mathcal{N}^{\geq i} \Prism^{(1)}_{S/A}) \rightarrow \text{H}^n(I^i \Prism_{S/A})$ is also an isomorphism, which concludes the proof. 
\end{proof}

\subsection{Comparison with $F$-smoothness}\label{subsectionFsmoothness}

\vspace{-\parindent}
\hspace{\parindent}

In this subsection we compare the absolute notion of $F$-smoothness, introduced by Bhatt--Mathew \cite{bhatt_syntomic_2022}, with our relative notion of $p$-Cartier smoothness, in the case when the base is perfectoid (Theorem~\ref{TheoremFsmoothequivalentCSmoverperfectoid}). $F$\nobreakdash-smoothness is a variant of ($p$\nobreakdash-adic) smoothness designed to capture smoothness in an absolute sense. For instance, regular rings are $F$\nobreakdash-smooth (\cite[Theorem~$4.15$]{bhatt_syntomic_2022}). 

Absolute prismatic cohomology can be defined using an absolute version of the relative prismatic site introduced in subsection \ref{subsectionreviewrelativeprismaticcohomology}. We recall only the necessary notation, and refer the reader to \cite{bhatt_absolute_2022} for the general theory. Following \cite[Section $4$]{bhatt_topological_2019} or \cite[Appendix $C$]{bhatt_absolute_2022}, a commutative ring $S$ is {\it $p$\nobreakdash-quasisyntomic} if $S$ has bounded $p^\infty$-torsion and $\mathbb{L}_{S/\Z} \otimes_{S}^{\mathbb{L}} S/p \in \mathcal{D}(S/p)$ has Tor-amplitude in~$[-1;0]$. Following \cite[Sections $4$ and $5$]{bhatt_absolute_2022} and for any $p$\nobreakdash-quasisyntomic ring $S$, the absolute prismatic site~$(S)_\Prism$ is the site having as objects the prisms $(B,J)$ with a map $S \rightarrow B/J$ and covers given by flat covers. The absolute prismatic complex $\Prism_S \in \mathcal{D}(\Z_p)$ is the cohomology of the sheaf $$\mathcal{O}_\Prism : (S)_\Prism \longrightarrow A\text{-Alg},\text{ } (B,J) \longmapsto B.$$ It is equipped with a Nygaard filtration $\mathcal{N}^{\geq \star} \Prism_S$ and a map of filtered complexes $$\phi : \mathcal{N}^{\geq \star} \Prism_S \longrightarrow \Prism_S^{[\star]},$$
where the filtered complex $\Prism_S^{[\star]}$ is an absolute version of the $I$-adic filtration on relative prismatic cohomology. This Frobenius map is compatible with the Frobenius map on relative prismatic cohomology when $S$ is defined over a base prism.

\begin{definition}[$F$-smoothness, \cite{bhatt_syntomic_2022}]
    A $p$\nobreakdash-quasisyntomic ring $S$ is \emph{$F$\nobreakdash-smooth} if for each integer~$i \geq 0$, the Nygaard filtration on $\Prism_S\{i\}$ is complete and the homotopy cofibre $\emph{hocofib}(\phi)$ of the Frobenius map 
    $$\phi : \mathcal{N}^i \Prism_S \longrightarrow \overline{\Prism}_S\{i\},$$
    where $\overline{\Prism}_S\{i\} := \emph{gr}^i\, \Prism_S^{[\star]}$,
    has $p$\nobreakdash-complete Tor-amplitude in degrees $\geq i+1$.
\end{definition}

Following \cite[Section $3$]{bhatt_prisms_2022}, a bounded prism $(A,I)$ is {\it perfect} if its Frobenius $\phi_A$ is an isomorphism. The functor $(A,I) \mapsto A/I$ induces an equivalence between the category of perfect prisms and the category of perfectoid rings (\cite[Theorem~$3.10$]{bhatt_prisms_2022}). Given a perfect prism $(A,I)$, a $p$\nobreakdash-quasisyntomic $A/I$-algebra $S$ and an integer $i \geq 0$, the canonical maps $$\Prism_S\{i\} \longrightarrow \Prism_{S/A}\{i\} \longrightarrow \Prism^{(1)}_{S/A}\{i\}$$ are equivalences and their composite can be refined into an equivalence of filtered complexes (\cite[Construction $5.6.1$ and Theorem~$5.6.2$]{bhatt_absolute_2022}) $$\mathcal{N}^{\geq \star} \Prism_S \{i\} \xlongrightarrow{\sim} \mathcal{N}^{\geq \star} \Prism^{(1)}_{S/A} \{i\}.$$

\begin{theorem}\label{TheoremFsmoothequivalentCSmoverperfectoid}
    Let $(A,I)$ be a perfect prism, and $S$ a $p$-discrete $A/I$-algebra. Assume that the ring $S$ is $p$\nobreakdash-quasisyntomic. Then $S$ is $F$\nobreakdash-smooth if and only if $S$ is $p$\nobreakdash-Cartier smooth over $A/I$.
\end{theorem}

\begin{proof}
    Assume first that $S$ is $p$\nobreakdash-Cartier smooth over $A/I$. Let $i \geq 0$ be an integer. The Frobenius map $$\phi : \mathcal{N}^i \Prism_S \longrightarrow \overline{\Prism}_S\{i\}$$ is naturally identified with the Frobenius map 
    $$\phi : \mathcal{N}^i \Prism^{(1)}_{S/A} \longrightarrow \overline{\Prism}_{S/A}\{i\},$$
    whose homotopy cofibre is $\tau^{\geq i+1} \overline{\Prism}_{S/A}\{i\}$ by Theorem~\ref{propositionpquasismoothproperties}\,$(2)$. For all $n\geq i+1$, the cohomology groups $\text{H}^n(\tau^{\geq i+1} \overline{\Prism}_{S/A}\{i\})$ are canonically isomorphic to $(\Omega^n_{S/(A/I)})^\wedge_p$ by the Hodge\nobreakdash--Tate comparison (Proposition~\ref{propositionpquasismoothproperties}\,$(3)$). By $p$-cotangent smoothness of the $A/I$-algebra $S$ the $S$\nobreakdash-modules $(\Omega^n_{S/(A/I)})^\wedge_p$ are $p$-completely flat for all $n$, hence the complex $\tau^{\geq i+1} \overline{\Prism}_{S/A}$ has $p$-complete Tor-amplitude in degrees~\hbox{$\geq i+1$}. Moreover, the Nygaard filtration on $\Prism_S\{i\}$ is canonically identified with the Nygaard filtration on $\Prism_{S/A}^{(1)}\{i\}$ (\cite[Theorem~$5.6.2$]{bhatt_absolute_2022}), where the Nygaard filtration on $\Prism_{S/A}^{(1)}\{i\}$ is defined as the tensor product of the Nygaard filtration on $\Prism_{S/A}^{(1)}$ with the invertible $A$-module $A\{i\}$ (\cite[Construction~$5.6.1$]{bhatt_absolute_2022}). In particular the Nygaard filtration on $\Prism_S\{i\}$ is complete, by Theorem~\ref{TheoremCartiersmoothMain}\,$(\Prism=\widehat{\Prism})$.
    
    Assume now that $S$ is $F$\nobreakdash-smooth. In particular the Frobenius map $$\phi : \mathcal{N}^1 \Prism^{(1)}_{S/A} \longrightarrow \overline{\Prism}_{S/A}\{1\}$$ has homotopy cofibre in degrees $\geq 2$, and $\mathcal{N}^1 \Prism^{(1)}_{S/A}$ is in degrees $\leq 1$ (Theorem~\ref{TheoremBS19qsyn}\,$(1)$), so there is an equivalence $$\tilde{\phi} : \mathcal{N}^1 \Prism^{(1)}_{S/A} \longrightarrow \tau^{\leq 1} \overline{\Prism}_{S/A}\{1\}.$$ The Frobenius factors through the equivalence $$\tilde{\phi} : \mathcal{N}^1\Prism^{(1)}_{S/A} \xlongrightarrow{\sim} \text{Fil}_1^{\text{conj}}\, \overline{\Prism}_{S/A}\{1\}$$ of Theorem~\ref{TheoremBS19qsyn}\,$(1)$, so the canonical map $$\text{Fil}^\text{conj}_1\, \overline{\Prism}_{S/A}\{1\} \xlongrightarrow{\sim} \tau^{\leq 1} \overline{\Prism}_{S/A}\{1\}$$ is an equivalence. Moreover, $\text{gr}^\text{conj}_0\,\overline{\Prism}_{S/A}\{1\}$ is naturally identified with $\text{H}^0(\overline{\Prism}_{S/A}\{1\})$ by the Hodge\nobreakdash--Tate comparison. So, using the Hodge--Tate comparison $\text{gr}^\text{conj}_1\,\overline{\Prism}_{S/A}\{1\} \simeq (\mathbb{L}_{S/(A/I)})^\wedge_p[-1]$, the cotangent complex $(\mathbb{L}_{S/(A/I)})^\wedge_p$ is in degree $0$ and is naturally identified with $\text{H}^1(\overline{\Prism}_{S/A}\{1\})$. Using the $F$\nobreakdash-smoothness condition for $i=0,1$, the complexes $\tau^{\geq 1}\overline{\Prism}_{S/A}\{1\} := \tau^{\geq 1} \overline{\Prism}_{S/A} \otimes_{A/I} I/I^2$ and $\tau^{\geq 2} \overline{\Prism}_{S/A}\{1\}$ have $p$\nobreakdash-complete Tor-amplitude in degrees $\geq 1$ and $\geq 2$ respectively. So $\text{H}^1(\overline{\Prism}_{S/A}\{1\})[-1]$, as the homotopy cofibre of the morphism $$\tau^{\geq 1} \overline{\Prism}_{S/A}\{1\} \longrightarrow \tau^{\geq 2} \overline{\Prism}_{S/A}\{1\},$$ has $p$\nobreakdash-complete Tor-amplitude in degrees $\geq 1$. In particular the cotangent complex $(\mathbb{L}_{S/(A/I)})^\wedge_p$ has $p$\nobreakdash-complete Tor-amplitude in degrees $\geq 0$, or equivalently it is a $p$\nobreakdash-completely flat $S$-module in degree~$0$. So the morphism $A/I \rightarrow S$ is $p$\nobreakdash-cotangent smooth. The Nygaard filtration on $\Prism_{S/A}^{(1)}$ is naturally identified with the Nygaard filtration on $\Prism_S$ (\cite[Theorem~$5.6.2$]{bhatt_absolute_2022}), and is thus complete. So $S$ is $p$\nobreakdash-Cartier smooth over $A/I$ by Theorem~\ref{TheoremCartiersmoothMain}\,$(\Prism=\widehat{\Prism})\Rightarrow(CSm)$.
\end{proof}

\begin{remark}
    In general, given a $p$\nobreakdash-discrete morphism $R \rightarrow S$ of $p$\nobreakdash-quasisyntomic rings, the notions of $F$\nobreakdash-smoothness for $S$ and $p$\nobreakdash-Cartier smoothness for $R\rightarrow S$ do not agree. For instance, the ring $\Z_p[p^{1/p}]$ is regular noetherian and is thus $F$\nobreakdash-smooth (\cite[Theorem~$4.15$]{bhatt_syntomic_2022}), but the morphism $\Z_p \rightarrow \Z_p[p^{1/p}]$ is not $p$\nobreakdash-Cartier smooth (it is not $p$\nobreakdash-cotangent smooth). On the other hand, the identity endomorphism of $R$ is always $p$\nobreakdash-Cartier smooth, but not all $p$-quasisyntomic rings $R$ are $F$\nobreakdash-smooth: any $F$\nobreakdash-smooth $p$\nobreakdash-complete noetherian ring is regular (\cite[Theorem~$4.15$]{bhatt_syntomic_2022}). If $S$ is a $p$\nobreakdash-Cartier smooth $\Z_p$-algebra, then $S$ is $F$\nobreakdash-smooth (\cite[Corollary~$4.17$]{bhatt_syntomic_2022}).
\end{remark}

\begin{remark}
    Let $(A,I)$ be a bounded prism, and $S$ a $p$\nobreakdash-quasisyntomic $A/I$-algebra. Then $S$ is $p$\nobreakdash-Cartier smooth over $A/I$ if and only if it satisfies the following relative version of $F$\nobreakdash-smoothness: for each integer $i \geq 0$ (or equivalently $i \in \{0, 1\}$), the homotopy cofibre of the Frobenius map \hbox{$\phi : \mathcal{N}^i \Prism^{(1)}_{S/A} \rightarrow \overline{\Prism}_{S/A}\{i\}$} has $p$\nobreakdash-complete Tor-amplitude in degrees $\geq i+1$ and the Nygaard filtration on~$\Prism^{(1)}_{S/A}$ is complete.
\end{remark}

\needspace{5\baselineskip}

\section{Valuation rings are Cartier smooth}\label{sectionvaluationrings}

\vspace{-\parindent}
\hspace{\parindent}

Recall that a valuation ring is an integral domain $C^+$ such that for any elements $f$ and $g$ in $C^+$, either $f\in gC^+$ or $g\in fC^+$. The valuation on the fraction field $C$ of $C^+$ is defined as the canonical group homomorphism $v : C^\times \twoheadrightarrow C^\times/(C^+)^\times =: \Gamma_C$, where $\Gamma_C$ is naturally equipped with a structure of totally ordered abelian group, and is called the value group of $C$. The value group $\Gamma_C$ (resp. the valuation ring $C^+$) is said to be discrete if it is isomorphic to the ring of integers $\Z$. The rank of a valuation ring is defined as its number of nonzero prime ideals. In particular, a valuation ring $C^+$ has rank at most $1$ if and only if its value group $\Gamma_C$ can be embedded in the ordered group of real numbers~$\R$. 
A valuation ring extension $E^+$ of $C^+$ is a valuation ring $E^+$ equipped with a flat ring morphism $C^+\rightarrow E^+$. The flat modules over a valuation ring $C^+$ are exactly the torsionfree $C^+$-modules, so a morphism $C^+ \rightarrow E^+$ of valuation rings is flat if and only if it is injective. 

Valuation rings arise in various contexts in $p$\nobreakdash-adic geometry, {\it e.g.}, in \cite{bhatt_projectivity_2017,huber_adic_1996,scholze_perfectoid_2012,bhatt_arc-topology_2021}. The goal of this section is to prove the following result.

\begin{theorem}\label{mainthmvaluationrings}
    Let $C^+$ be a valuation ring whose $p$\nobreakdash-adic completion is a perfectoid ring.
    Let $E^+$ be a valuation ring extension of $C^+$. Then the morphism $C^+ \rightarrow E^+$ is $p$\nobreakdash-Cartier smooth.
\end{theorem}

\begin{corollary}
    Let $C^+$ be a $p$-torsionfree valuation ring whose $p$-adic completion is a perfectoid ring, and $E^+$ be a valuation ring extension of $C^+$. Then for any integers $i\geq 0$ and $n\geq 1$, the map $$\Z/p^n(i)^\emph{syn}(E^+) \longrightarrow R\Gamma_{\emph{ét}}(\emph{Spec}(E^+[\tfrac{1}{p}]),\mu_{p^n}^{\otimes i})$$ is an isomorphism on cohomology in degrees $<i$, and is injective on $\emph{H}^i$.
\end{corollary}

\begin{proof}
    This is a direct consequence of Theorems~\ref{mainthmvaluationrings} and \ref{Theoremsyntomicetalecomparison2}.
\end{proof}

Examples of valuation rings $C^+$ satisfying the hypothesis of Theorem~\ref{mainthmvaluationrings} include the ring of integers $\overline{\Z}_p$ of an algebraic closure $\overline{\Q}_p$ of $\Q_p$, the ring of integers $\mathcal{O}_{\mathbb{C}_p}$ of the $p$\nobreakdash-adic complex numbers $\mathbb{C}_p$ and the ring $\Z_p[p^\infty]$ (Example \ref{Exampleperfectoidvaluationring} below).

\begin{remark}\label{RemarkvaluationringssufficesmodpCSm}
    By Example~\ref{examplepdiscrete}\,$(2)$, an extension of valuation rings is $p$-discrete. In particular, an extension of valuation rings $C^+ \rightarrow E^+$ is $p$\nobreakdash-Cartier smooth if and only if the morphism $C^+/p \rightarrow E^+/p$ is Cartier smooth, {\it i.e.}, if $C^+/p \rightarrow E^+/p$ satisfies the Cartier isomorphism and the cotangent complex $\mathbb{L}_{(E^+/p)/(C^+/p)}$ is a flat $(E^+/p)$-module in degree $0$.
\end{remark}

We will distinguish three cases. If $p$ is invertible in the valuation ring $C^+$, then the rings $C^+/p$ and $E^+/p$ are zero\footnote{There is a slight ambiguity whether the zero ring is perfectoid or not. We include this case to avoid any ambiguity.} and the result is trivial. If $p$ is zero in the valuation ring $C^+$, the result is essentially due to Gabber\nobreakdash--Ramero and Gabber, as we recall now. We will then focus on the remaining case, \hbox{{\it i.e.}, that} of mixed-characteristic.

Assume that $p=0$ in the valuation ring $C^+$. By \cite[Example $3.15$]{bhatt_integral_2018} an $\F_p$-algebra, or equivalently its $p$\nobreakdash-adic completion, is a perfectoid ring if and only if it is perfect ({\it i.e.}, its Frobenius endomorphism is an isomorphism). Theorem~\ref{mainthmvaluationrings} is then due to Gabber--Ramero and Gabber. Remark that Cartier smoothness in characteristic $p$, including its relation to valuation rings and algebraic $K$\nobreakdash-theory, was previously studied in \cite{kelly_k-theory_2021}.

\begin{theorem}[Cartier smoothness of characteristic $p$ valuation rings, \cite{gabber_almost_2003,kerz_towards_2021}]\label{TheoremGabberandGabberRamero}
    Let $C^+$ be a perfect valuation ring of characteristic $p$, and $E^+$ a valuation ring extension of $C^+$. Then the morphism $C^+ \rightarrow E^+$ is Cartier smooth. Equivalently:
    \begin{enumerate}
        \item The cotangent complex $\mathbb{L}_{E^+/C^+}$ is concentrated in degree $0$, and $\Omega^1_{E^+/C^+}$ is a flat $E^+$-module.
        \item The inverse Cartier map $$C^{-1} : \Omega^n_{E^+/C^+} \otimes_{C^+,\phi_{C^+}} C^+ \longrightarrow \emph{H}^n(\Omega^{\bullet}_{E^+/C^+})$$ is an isomorphism of $C^+$-modules for each $n\geq 0$.
    \end{enumerate}
\end{theorem}

\begin{proof}
    $(1)$ The cotangent complexes $\mathbb{L}_{E^+/C^+}$ and $\mathbb{L}_{E^+/\F_p}$ are concentrated in degree $0$ (\cite[Theorem~$6.5.8\,(ii)$]{gabber_almost_2003}) where they are given by $\Omega^1_{E^+/C^+}$ and $\Omega^1_{E^+/\F_p}$, and $\Omega^1_{E^+/\F_p}$ is a torsionfree $E^+$-module (\cite[Corollary~$6.5.21$]{gabber_almost_2003}). There is a homotopy transitivity fibre sequence $$\mathbb{L}_{C^+/\F_p}\otimes^\mathbb{L}_{C^+}E^+ \longrightarrow \mathbb{L}_{E^+/\F_p} \longrightarrow \mathbb{L}_{E^+/C^+},$$
    and $\mathbb{L}_{C^+/\F_p} \simeq 0$ because $C^+$ is a perfect $\F_p$-algebra. So the natural morphism $$\Omega^1_{E^+/\F_p} \longrightarrow \Omega^1_{E^+/C^+}$$ is an isomorphism of $E^+$-modules. Torsionfree modules over a valuation ring are flat, so $\Omega^1_{E^+/C^+}$ is a flat $E^+$-module.
    
    $(2)$ The morphism $\F_p \rightarrow E^+$ satisfies the Cartier isomorphism (\cite[Corollary~$A.4$]{kerz_towards_2021}). By the previous paragraph, there is a canonical isomorphism $$\Omega^1_{E^+/\F_p} \xlongrightarrow{\cong} \Omega^1_{E^+/C^+}$$ of $E^+$-modules, so the morphism $C^+\rightarrow E^+$ also satisfies the Cartier isomorphism.
\end{proof}

Assume now, for the rest of this section, that $p$ is neither invertible nor zero in the valuation ring~$C^+$; we say in this case that $C^+$ is a {\it mixed-characteristic} valuation ring. The hypothesis on $C^+$ in Theorem~\ref{mainthmvaluationrings} can be reformulated as follows.

\begin{lemma}\label{lemmacharacterisationperfectoidvaluationring}
    Let $C^+$ be a mixed-characteristic valuation ring. The following are equivalent:
    \begin{enumerate}
        \item The $p$\nobreakdash-adic completion of the valuation ring $C^+$ is a perfectoid ring.
        \item The ring $C^+/p$ has a nonzero nilpotent element, and its Frobenius endomorphism is surjective.
    \end{enumerate}
\end{lemma}

\begin{proof}
    The second condition depends only on the ring $C^+/p$, so we can assume the valuation ring $C^+$ is $p$\nobreakdash-adically complete. Because $p$ is nonzero in the valuation ring $C^+$, the ring $C^+$ is in particular $p$\nobreakdash-torsionfree.
    
    By \cite[Lemmas $3.9$ and $3.10$]{bhatt_integral_2018}, a $p$\nobreakdash-torsionfree ring $R$ is perfectoid if and only if $R$ is $p$\nobreakdash-adically complete, the Frobenius on $R/p$ is surjective and $up \in R$ admits a compatible system of $p$\nobreakdash-power roots for some unit $u \in R^\times$. From now on, let $C^+$ be a nonzero $p$\nobreakdash-adically complete and $p$\nobreakdash-torsionfree valuation ring, such that the Frobenius on $C^+/p$ is surjective.
    
    $(1) \Rightarrow (2)$ Let $u \in (C^+)^\times$ be a unit such that $up$ admits a compatible system $((up)^{1/p^n})_{n\in \N}$ of $p$\nobreakdash-power roots in the valuation ring $C^+$. Because $C^+$ is $p$\nobreakdash-torsionfree and not the zero ring, $p$ is nonzero in $C^+$ and neither is $(up)^{1/p}$. The element $(up)^{1/p} \in C^+$ thus defines a nonzero nilpotent element in the ring $C^+/p$.
    
    $(2) \Rightarrow (1)$ Assume the valuation ring $C^+$ has an element $\pi \in C^+$ defining a nonzero nilpotent element in~$C^+/p$. Because the Frobenius on $C^+/p$ is surjective, we can assume that $\pi^p$ is nonzero in~$C^+/p$. Because $C^+$ is a valuation ring, this implies that $\pi^p$ divides $p$ and that $p$ divides some power of $\pi$ in $C^+$. In particular the valuation ring $C^+$ is $\pi$-adically complete for this element $\pi \in C^+$. By \cite[Lemma~$3.9$]{bhatt_integral_2018}, there exists a unit $u \in (C^+)^\times$ such that $up \in C^+$ admits a compatible system of $p$\nobreakdash-power roots in $C^+$.
\end{proof}

\begin{example}\label{Exampleperfectoidvaluationring}
    The $p$\nobreakdash-adic completion of the ring $\Z_p[p^{1/p^\infty}]$ is a perfectoid valuation ring, because it satisfies the hypotheses of Lemma~\ref{lemmacharacterisationperfectoidvaluationring}\,$(2)$. In particular, any valuation ring extension $E^+$ of $\Z_p$ containing a compatible system of $p$\nobreakdash-power roots of $p$ is $p$\nobreakdash-Cartier smooth over $\Z_p[p^{1/p^\infty}]$ by Theorem~\ref{mainthmvaluationrings}. Beware that $p$\nobreakdash-adic completion (or even taking the $p$\nobreakdash-adically separated quotient) does not preserve the value group of a valuation ring. For instance, the localisation of the ring $\Z_p[X/p^n,n\geq 0]$ at the ideal $(X/p^n,n\geq0) \subset \Z_p[X/p^n,n\geq 0]$ is a valuation ring, with fraction field $\Q_p(X)$ and with $p$\nobreakdash-adic completion $\Z_p$; an alternative description of this valuation ring is the fiber product $\Z_p \times_{\Q_p} \Q_p[X]_{(X)}$. Similarly, the localisation of the ring $\Z_p[p^{1/p^\infty},X/p^n, n \geq 0]$ at the ideal $(X/p^n,n\geq 0) \subset \Z_p[p^{1/p^\infty},X/p^n, n \geq 0]$ is a valuation ring, with fraction field $\Q_p(p^{1/p^\infty},X)$ and $p$\nobreakdash-adically separated quotient $\Z_p[p^{1/p^\infty}]$.
\end{example}

To prove Theorem~\ref{mainthmvaluationrings} for the morphism $C^+ \rightarrow E^+$, it suffices to prove that the morphism \hbox{$C^+/p \rightarrow E^+/p$} is cotangent smooth ({\it i.e.}, its cotangent complex $\mathbb{L}_{(E^+/p)/(C^+/p)}$ is given by a flat $E^+/p$\nobreakdash-modu\-le in degree $0$) and satisfies the Cartier isomorphism (Remark~\ref{RemarkvaluationringssufficesmodpCSm}). We first prove the cotangent smoothness.

\begin{lemma}\label{Lemmaptorsionfreeimpliesmodpisflat}
    Let $E^+$ be a valuation ring in which $p$ is nonzero, and $M$ be an $E^+$-module. If the $E^+$\nobreakdash-module $M$ is $p$\nobreakdash-torsionfree, then its derived reduction $M \otimes_{E^+}^\mathbb{L} E^+/p$ modulo $p$ is a flat $E^+/p$\nobreakdash-module in degree $0$.
\end{lemma}

\begin{proof}
    Because $p$ is a nonzerodivisor in the valuation ring $E^+$, the derived reduction $M \otimes_{E^+}^\mathbb{L} E^+/p$ modulo $p$ of any $E^+$-module $M$ is concentrated in degree $0$ if and only if $M$ is $p$\nobreakdash-torsionfree. 
    
    Assuming $M$ is a $p$\nobreakdash-torsionfree $E^+$-module, we prove now that $M/p$ is a flat $E^+/p$-module. Consider the $p$\nobreakdash-adically separated quotient $E^+{}'$ of $E^+$, {\it i.e.}, the quotient of $E^+$ by its ideal $\bigcap_{n\geq 0} p^nE^+$. 
    The ring $E^+{}'$ is a valuation ring, because it is an integral domain and satisfies the divisibility condition of valuation rings. 
    Define $M'$ as the $E^+{}'$-module $M \otimes_{E^+} E^+{}'$. The morphism $M \rightarrow M'$ of $E^+$-modules is surjective, with kernel given by a quotient of the $E^+$-module $M \otimes_{E^+} \bigcap_{n\geq 0} p^n E^+$. The $E^+$-module $M \otimes_{E^+} \bigcap_{n\geq 0} p^n E^+$ is $p$\nobreakdash-divisible, so the morphism $M \rightarrow M'$ becomes an isomorphism after reduction modulo $p$. Because $E^+{}'$ is a $p$\nobreakdash-adically separated valuation ring, any $p$\nobreakdash-torsionfree $E^+{}'$-module is torsionfree and thus flat. In particular the $E^+{}'$-module $M'$ is flat, and the $E^+{}'/p$-module $M'/p$ is also flat. So the $E^+/p$-module $M/p$ is flat.
\end{proof}

\begin{proposition}[Cotangent smoothness of the morphism $C^+/p\rightarrow E^+/p$]\label{Propositionquasismoothnessvaluationrings}
    Let $C^+$ be a mixed-characteristic valuation ring whose $p$\nobreakdash-adic completion is a perfectoid ring. Let $E^+$ be a valuation ring extension of$~C^+$. Then the $E^+$-module $\Omega^1_{E^+/C^+}$ is $p$\nobreakdash-torsionfree and the cotangent complex $\mathbb{L}_{(E^+/p)/(C^+/p)}$ is a flat $E^+/p$\nobreakdash-module in degree $0$.
\end{proposition}

\begin{proof}
    Note the equivalence $\mathbb{L}_{(E^+/p)/(C^+/p)} \simeq \mathbb{L}_{E^+/C^+} \otimes_{E^+}^\mathbb{L} E^+/p$. The cotangent complex of any extension of valuation rings with characteristic $0$ fraction fields is concentrated in degree $0$ (\cite[Theorem~$6.5.8\,(ii)$]{gabber_almost_2003}\footnote{More precisely, \cite[Theorem~$6.5.8\,(ii)$]{gabber_almost_2003} is formulated for extensions of valued fields. An extension of valuation rings $C^+ \hookrightarrow E^+$ is the composition of the localisation $C^+ \hookrightarrow C^+_\mathfrak{p}$ at the prime ideal $\mathfrak{p} := C^+ \cap \mathfrak{m}_{E^+}$ and the morphism $C^+_\mathfrak{p} \hookrightarrow E^+$. The cotangent complex of the first morphism is trivial and the second morphism is induced by an extension of valued fields.}). In particular the cotangent complex $\mathbb{L}_{E^+/C^+}$ is concentrated in degree $0$, given by the $E^+$-module $\Omega^1_{E^+/C^+}$. If the $E^+$-module $\Omega^1_{E^+/C^+}$ is $p$\nobreakdash-torsionfree, then the cotangent complex $\mathbb{L}_{(E^+/p)/(C^+/p)}$ is a flat $E^+/p$-module in degree $0$ (Lemma~\ref{Lemmaptorsionfreeimpliesmodpisflat}).
    
    We prove now that the $E^+$-module $\Omega^1_{E^+/C^+}$ is $p$\nobreakdash-torsionfree. Let $E$ be the fraction field of the valuation ring $E^+$, and $\overline{E}$ an algebraic closure of $E$. We fix a valuation on $\overline{E}$ extending the valuation of the valued field $E$, and denote by $\overline{E}^+$ the corresponding ring of integers. Applying again \cite[Theorem~$6.5.8\,(ii)$]{gabber_almost_2003} to $C^+$, $E^+$ and $\overline{E}^+$ and because the morphism $E^+ \rightarrow \overline{E}^+$ is flat, the transivity sequence for the cotangent complex can be rewritten as a short exact sequence of $\overline{E}^+$-modules:
    $$0 \longrightarrow \Omega^1_{E^+/C^+} \otimes_{E^+} \overline{E}^+ \longrightarrow \Omega^1_{\overline{E}^+/C^+} \longrightarrow \Omega^1_{\overline{E}^+/E^+} \longrightarrow 0.$$
    An $E^+$-module (resp. $\overline{E}^+$-module) $M$ is $p$\nobreakdash-torsionfree if and only if the multiplication by $p$ morphism $p : M \rightarrow M$ is injective. The morphism $E^+ \rightarrow \overline{E}^+$ being faithfully flat, the $E^+$-module $\Omega^1_{E^+/C^+}$ is thus $p$\nobreakdash-torsionfree if and only if the $\overline{E}^+$-module $\Omega^1_{\overline{E}^+/C^+} \otimes_{E^+} \overline{E}^+$ is $p$\nobreakdash-torsionfree. By the previous short exact sequence, it then suffices to prove the $\overline{E}^+$-module $\Omega^1_{\overline{E}^+/C^+}$ is $p$\nobreakdash-torsionfree. 
    The $p$\nobreakdash-adic completions $(C^+)^\wedge_p$ and $(\overline{E}^+)^\wedge_p$ of $C^+$ and $\overline{E}^+$ are perfectoid rings, respectively by assumption and because the field $\overline{E}$ is algebraically closed. In particular the cotangent complex $\mathbb{L}_{((\overline{E}^+)^\wedge_p)/p / ((C^+)^\wedge_p)/p}$ is zero (\cite[Lemma~$3.14$]{bhatt_integral_2018}). Equivalently, the cotangent complex $\mathbb{L}_{(\overline{E}^+/p)/(C^+/p)}$ is zero. In particular, the derived reduction $\Omega^1_{\overline{E}^+/C^+} \otimes_{\overline{E}^+}^\mathbb{L} \overline{E}^+/p$ modulo $p$ of the $\overline{E}^+$-module $\Omega^1_{\overline{E}^+/C^+}$ is in degree $0$, or equivalently the $\overline{E}^+$-module $\Omega^1_{\overline{E}^+/C^+}$ is $p$\nobreakdash-torsionfree.
\end{proof}

It remains to prove the Cartier isomorphism for the reduction modulo $p$ of valuation ring extensions $C^+ \rightarrow E^+$, where $C^+$ is a mixed-characteristic valuation ring whose $p$\nobreakdash-adic completion is a perfectoid ring. 
The proof of the Cartier isomorphism for positive charateristic valuation rings \cite[Corollary~$A.4$]{kerz_towards_2021} relies on subtle approximation results of Gabber, which do not immediately pass to mixed-characteristic. Our strategy of proof in mixed-characteristic is to reduce to this result in positive characteristic. We are immediately faced with the following issue: if $C^+$ is a mixed-characteristic valuation ring, the ring $C^+/p$ is in general not an integral domain, and in particular not a valuation ring.

Here we use the perfectoid assumption on the base $C^+$ to remark that there is a perfect valuation ring $C^{+\flat}$ of characteristic $p$ (the tilt of $C^+$) whose reduction modulo some element $d \in C^{+\flat}$ is naturally isomorphic to $C^+/p$. The Cartier isomorphism is preserved by base change (Lemma~\ref{LemmabasechangeforCartiersmoothness}). So it would suffice to find a valuation ring extension $E^{+\flat}$ of $C^{+\flat}$ whose reduction modulo $d$ is $E^+/p$ to prove the Cartier isomorphism for $C^+/p \rightarrow E^+/p$. To construct such a valuation ring $E^{+\flat}$ over the $d$\nobreakdash-adically complete lift $C^{+\flat}$ of $C^+/p$, we use the deformation theory of Illusie \cite[III$.2.1.2.3$]{illusie_complexe_1971} (see also \cite[Theorem~$5.11$]{scholze_perfectoid_2012}). Namely, we will need the following result, where $R$ is a ring, $I\subset R$ is an ideal such that $I^2 = 0$, and $S_0$ is a flat $R_0:=R/I$-algebra.

\begin{theorem}[\cite{illusie_complexe_1971}]\label{TheoremIllusieondeformations}
    There is an obstruction class $\omega(S_0)$ in the group $\emph{Ext}^2_{S_0}(\mathbb{L}_{S_0/R_0},S_0\otimes_{R_0}I)$ which vanishes precisely when there exists a flat $R$-algebra $S$ such that $S \otimes_R R_0 = S_0$. If there exists such a deformation, then the set of all isomorphism classes of such deformations forms a torsor under $\emph{Ext}^1_{S_0}(\mathbb{L}_{S_0/R_0},S_0\otimes_{R_0}I)$, and every deformation has automorphism group $\emph{Hom}_{S_0}(\mathbb{L}_{S_0/R_0},S_0\otimes_{R_0} I)$. 
\end{theorem}

To apply this deformation result recursively up to a $d$\nobreakdash-adic deformation $E^{+\flat}$ of the flat $C^+/p$\nobreakdash-alge\-bra~$E^+/p$, we need to have control on the cotangent complex $\mathbb{L}_{(E^+/p)/(C^+/p)}$. More precisely, if the cotangent complex $\mathbb{L}_{(E^+/p)/(C^+/p)}$ was a projective $E^+/p$-module in degree $0$ then the higher $\text{Ext}$\nobreakdash-groups in this deformation result would vanish and we could construct the lift $E^{\flat}$ in a unique way. By Proposition~\ref{Propositionquasismoothnessvaluationrings}, the cotangent complex $\mathbb{L}_{(E^+/p)/(C^+/p)}$ is concentrated in degree $0$, given by the flat $E^+/p$-module~$\Omega^1_{(E^+/p)/(C^+/p)}$. The inverse Cartier map commuted with filtered colimits, so we can assume the field extension $C \rightarrow E$ induced by the valuation ring extension $C^+ \rightarrow E^+$ is of finite type to prove the Cartier isomorphism. But even for such valuation ring extensions, the $E^+/p$\nobreakdash-module $\Omega^1_{(E^+/p)/(C^+/p)}$ will not be projective in general. For instance, when $E \cong C(X)$, it follows from the proof of \cite[Proposition~$6.5.6$]{gabber_almost_2003} that the $E^+$-module $\Omega^1_{E^+/C^+}$ can be isomorphic to $\mathfrak{m}_{C^+}E^+$, where~$\mathfrak{m}_{C^+}$ is the maximal ideal of $C^+$. The next result will ensure the vanishing of the relevant $\text{Ext}^2$\nobreakdash-groups in the deformation theory, and thus the existence of a lift $E^{+\flat}$ (see also Remark~\ref{remarkvaluationringliftuniqueness} about the uniqueness of such a lift).

\begin{lemma}\label{lemmacgvaluationring}
    Let $E^+$ be a finite rank valuation ring, with fraction field $E$. If $M$ is a torsionfree $E^+$\nobreakdash-module such that $M \otimes_{E^+} E$ is a finite dimensional $E$-vector space, then the $E^+/p$-module $M/p$ is a countable filtered colimit of free $E^+/p$\nobreakdash-modules of rank at most $\emph{dim}_E(M\otimes_{E^+} E)$. In particular, the $E^+/p$-module $M/p$ has projective dimension at most $1$.
\end{lemma}

\begin{proof}
    The morphism of $E^+$-modules $M \rightarrow M\otimes_{E^+} E$ is injective because the $E^+$-module $M$ is torsionfree. We identify the $E^+$-module $M$ with a submodule of the $E^+$-module $M \otimes_{E^+} E$ via this morphism. We prove by induction on the dimension $d$ of the $E$-vector space $M \otimes_{E^+} E$ that $M$ is a countable filtered union of free $E^+$\nobreakdash-modules.

    For $d=0$, the $E^+$-module $M$ is free of rank $0$. For $d=1$, the $E^+$-module $M$ is isomorphic to a submodule of $E$, and we consider two cases. If $M$ is equal to $E$, then $M$ is the filtered union of the free $E^+$-submodules of rank $1$ of $E$. If $M$ is not equal to $E$, then up to shifting $M$ multiplicatively by an element of $E^+$, we can assume that $M$ is contained in $E^+$, {\it i.e.}, that $M$ is an ideal of $E^+$. Every ideal of a ring is the filtered union of its finitely generated subideals, and every such ideal is principal in a valuation ring. So the $E^+$-module $M$ can be written as a filtered union of free $E^+$-submodules of rank at most $1$. In both cases, and because the valuation ring $E^+$ has finite rank, one can assume that the filtered colimit is countable. So the $E^+$-module $M$ is a countable filtered union of free $E^+$-modules of rank at most $1$.

    Fix an integer $d \geq 1$, assume the result is proved for all integers $\leq d$, and let $M$ be an $E^+$-submodule of $(d+1)$-dimensional $E$-vector space $\oplus_{i=1}^{d+1} Ee_i$. Let $M_{e_{d+1}}$ be the image of $M$ under the projection $p_{d+1} : \oplus_{i=1}^{d+1} Ee_i \rightarrow Ee_{d+1}$. By the previous paragraph, the $E^+$-module $M_{e_{d+1}}$ is the countable filtered union of a system $(M_{e_{d+1}}^{(n)})_{n \in \N}$ of free $E^+$-modules of rank at most $1$. Let $M^{(n)} := p_{d+1}^{-1}(M_{e_{d+1}}^{(n)})$, so that $M$ is the filtered union of the $E^+$-modules $M^{(n)}$. For each integer $n \geq 0$, there is a short exact sequence of $E^+$-modules
    $$0 \longrightarrow M^{(n)} \cap \oplus_{i=1}^d Ee_i \longrightarrow M^{(n)} \longrightarrow M_{e_{d+1}}^{(n)} \longrightarrow 0,$$
    which is split since the $E^+$-module $M_{e_{d+1}}^{(n)}$ is free. The induction hypothesis implies that, for each integer $n \geq 0$, the $E^+$-module $M^{(n)} \cap \oplus_{i=1}^d Ee_i$ is a countable filtered union of free $E^+$-modules. Taking the direct sum with the $E^+$-module $M^{(n)}_{e_{d+1}}$, this implies that $M^{(n)}$ is the countable filtered union of a system $(M^{(n,m)})_{m \in \N}$ of free $E^+$-modules. The $E^+$-module $M$ is thus the countable union of the free $E^+$-modules $M^{(n,m)}$, $n,m \in \N$, and this union is filtered by construction. This concludes the induction.

    In particular, the $E^+/p$-module $M/p$ is a countable filtered colimit $(P^{(n)})_{n \in \N}$ of free $E^+/p$-modules of rank at most $\text{dim}_E(M\otimes_{E^+}E)$. To prove the last claim, consider the Milnor exact sequence 
    $$0 \longrightarrow \oplus_{n \in \N} P^{(n)} \xlongrightarrow{\partial} \oplus_{n \in \N} P^{(n)} \longrightarrow M/p \longrightarrow 0$$
    of $E^+$-modules, where $f_n : P^{(n)} \rightarrow P^{(n+1)}$ is the transition map and $$\partial : (x_n)_{n \geq 0} \longmapsto (x_n-f_{n-1}(x_{n-1})),$$ where $f_{-1}(x_{-1}):=0$. For every $E^+/p$-module $Q$, the Ext-groups $\text{Ext}^i_{E^+/p}(P^{(n)},Q)$ vanish for all integers $n \geq 0$ and $i \geq 1$. Applying the long exact sequence of Ext-groups to the previous short exact sequence of $E^+/p$-modules implies that 
    $$\text{Ext}^i_{E^+/p}(M/p,Q) \cong 0$$
    for every integer $i \geq 2$ and every $E^+/p$-module $Q$, {\it i.e.}, that the $E^+/p$-module $M/p$ has projective dimension at most $1$.
\end{proof}

\begin{corollary}\label{Corollaryalmostprojective}
     Let $C^+$ be a mixed-characteristic valuation ring whose $p$\nobreakdash-adic completion is a perfectoid ring. Let $E^+$ be a finite rank valuation ring extension of $C^+$. Assume that $C^+$ and $E^+$ are $p$\nobreakdash-adically separated, and that the induced field extension $C \rightarrow E$ is of finite type. Then the $E^+/p$\nobreakdash-module $\Omega^1_{(E^+/p)/(C^+/p)}$ has projective dimension at most $1$.
\end{corollary}

\begin{proof}
    We apply Lemma~\ref{lemmacgvaluationring} to the $E^+$-module $\Omega^1_{E^+/C^+}$. The $E^+$-module $\Omega^1_{E^+/C^+}$ is $p$\nobreakdash-torsionfree (Proposition~\ref{Propositionquasismoothnessvaluationrings}), and thus torsionfree since $E^+$ is $p$\nobreakdash-adically separated. The localisation at a prime ideal has trivial cotangent complex, so the canonical map $\Omega^1_{E^+/C^+} \otimes_{E^+} E \rightarrow \Omega^1_{E/C}$ is an isomorphism. Moreover, the $E$-vector space $\Omega^1_{E/C}$ is finite dimensional since $C \rightarrow E$ is a field extension of finite type. The reduction modulo $p$ of the $E^+$-module $\Omega^1_{E^+/C^+}$ is the $E^+/p$-module $\Omega^1_{(E^+/p)/(C^+/p)}$, hence the conclusion.
\end{proof}

We now use this result on the projective dimension of $\Omega^1_{(E^+/p)/(C^+/p)}$ to lift the $C^+/p$-algebra $E^+/p$ to a valuation ring extension $E^{+\flat}$ of $C^{+\flat}$.

\begin{notation}\label{Notationalmostmaths}
    Let $C^+$ be a mixed-characteristic valuation ring whose $p$\nobreakdash-adic completion is a perfectoid ring. Denote by $$C^{+\flat}:=\lim\limits_{x\mapsto x^p} C^+/p$$ the tilt of this perfectoid ring. The ring $C^{+\flat}$ is a perfect valuation ring of characteristic $p$, which is $d$\nobreakdash-adically complete for some element $d \in C^{+\flat}$ such that there is a natural ring isomorphism $$C^+/p \cong C^{+\flat}/d.$$
\end{notation}

We will need the following to ensure that the flat lift $E^{+\flat}$ produced by deformation theory is actually a valuation ring.

\begin{lemma}\label{Lemmaflatimpliesvaluationring}
    Let $C^+$ be a mixed-characteristic valuation ring whose $p$\nobreakdash-adic completion is a perfectoid ring. Let $E^+$ be a valuation ring extension of $C^+$ and $E^{+\flat}$ a $d$\nobreakdash-complete flat $C^{+\flat}$-algebra. Assume there is an isomorphism of \hbox{$C^{+\flat}/d=C^+/p$-algebras} $E^{+\flat}/d \cong E^+/p$. Then $E^{+\flat}$ is a valuation ring extension of $C^{+\flat}$.
\end{lemma}

\begin{proof}
    We want to prove that $E^{+\flat}$ is an integral domain such that for any elements $f$ and $g$ in $E^{+\flat}$, either $f$ divides $g$ or $g$ divides $f$.

    The valuation ring $C^{+\flat}$ is perfect by definition. Denote by $(d^{1/p^n})_{n \in \N} \in (C^{+\flat})^{\N}$ a compatible system of $p$\nobreakdash-power roots of $d \in C^{+\flat}$. A module over the valuation ring $C^{+\flat}$ is flat if and only if it is torsionfree, so the $C^{+\flat}$-module $E^{+\flat}$ is $d^{1/p^n}$-torsionfree for each $n \geq 0$. In the following we identify the rings $C^+/p$ and $C^{+\flat}/d$; in particular any (rational) power $d^\alpha$ of $d \in C^{+\flat}$ defines an element of~$C^+/p$, which we also denote by $d^\alpha$. 

    Let $f$ and $g$ be elements of the ring $E^{+\flat}$. Let us prove that either $f$ divides $g$ or $g$ divides $f$. The ring $E^{+\flat}$ is $d$-torsionfree, so we can assume that $d$ does not divide one of $f$ and $g$. We first assume that $d$ does not divide $f$ in $E^{+\flat}$, and prove that $f$ divides $d$. Because $E^+$ is a valuation ring, and for every $\alpha \in \Z[\tfrac{1}{p}]$, $\alpha < 1$, either $f$ divides $d^\alpha$, or $d^{\alpha}$ divides $f$ in the ring $E^{+\flat}/d$. If $f$ divides $d^\alpha$ for such an $\alpha$ in $E^{+\flat}/d$, then we can write $fh=d^\alpha + dt$ for some elements $h,t \in E^{+\flat}$. Because $E^{+\flat}$ is $d$-adically complete it is also $d^{1-\alpha}$-adically complete. In particular $1 + d^{1-\alpha}t$ is a unit in $E^{+\flat}$, so $f$ divides $d^{\alpha}$, and thus $f$ divides $d$. If $f$ does not divide $d^{\alpha}$ for any $\alpha \in \Z[\tfrac{1}{p}]$, $\alpha < 1$, then in particular $d^{1/p}$ divides $f$ in $E^{+\flat}/d$, and thus also in $E^{+\flat}$. Because $E^{+\flat}$ is $d^{1/p}$-torsionfree, we can consider the element $f/d^{1/p} \in E^{+\flat}$, which divides $d^{1-1/p}$ in $E^{+\flat}/d$. And we argue as in the previous case to prove that $f/d^{1/p}$ divides $d^{1-1/p}$ in $E^{+\flat}$.
    
    Going back to the elements $f$ and $g$ of $E^{+\flat}$, we deduce the following: if $f$ is not divisible by $d$ and $g$ is divisible by $d$, then $f$ divides $g$. We now assume that $d$ does not divide $f$ or $g$. Without loss of generality and because $E^+$ is a valuation ring, there is an element $h \in E^+/p$ such that $f=gh$ in $E^{+\flat}/d$. We can then write $f=g\tilde{h}+dt$ for some elements $\tilde{h},t \in E^{+\flat}$. By the previous paragraph, and because $d$ does not divide $g$, we know that $g$ divides $d$. So $g$ divides $f$.

    We now prove that the ring $E^{+\flat}$ is an integral domain. Given two elements $f$ and $g$ of the ring~$E^{+\flat}$, then without loss of generality we can assume that $g$ divides $f$ using the divisibility property we just proved. Assume now $fg=0$. The ring $E^{+\flat}$ is $d$\nobreakdash-adically separated, so if $f$ is nonzero then $f$ divides $d^n$ for some integer $n \geq 1$. But then $d^{2n}=0$, which is impossible since $E^{+\flat}$ is $d$\nobreakdash-torsionfree. So $f=0$, and the ring $E^{+\flat}$ is an integral domain. By assumption $E^{+\flat}$ is a flat $C^{+\flat}$-algebra, so $E^{+\flat}$ is a valuation ring extension of $C^{+\flat}$.
\end{proof}

\begin{theorem}\label{Propositiondeformationvaluationringwhenfinitetypestarid}
    Let $C^+$ be a mixed-characteristic valuation ring whose $p$\nobreakdash-adic completion is a perfectoid ring, and $E^+$ be a valuation ring extension of $C^+$. If the $E^+/p$-module $\Omega^1_{(E^+/p)/(C^+/p)}$ has projective dimension at most $1$, then there exists a $d$\nobreakdash-complete valuation ring extension $E^{+\flat}$ of~$C^{+\flat}$ such that there is a ring isomorphism $E^{+\flat}/d \cong E^+/p$.
\end{theorem}

The projective dimension hypothesis is satisfied in particular if the induced field extension $C \rightarrow E$ is of finite type and $E^+$ is of finite rank (Corollary~\ref{Corollaryalmostprojective}). It is also satisfied for the $p$\nobreakdash-adic completion of such an extension, as it depends only on its reduction modulo $p$.

\begin{proof}
    We prove by induction on~$n \geq 1$ that there are, for all $1 \leq m \leq n$, flat $C^{+\flat}/d^m$-algebras~$E^{+\flat}_m$ with isomorhisms $E^{+\flat}_m/d \xrightarrow{\cong} E^+/p$ and with discrete cotangent complexes $\mathbb{L}_{E^{+\flat}_m/(C^{+\flat}/d^m)}$ having projective dimension at most $1$. By Theorem~\ref{TheoremIllusieondeformations}, the vanishing of the obstructions in this process can be expressed in terms of the cotangent complexes $\mathbb{L}_{(E^{+\flat}_n)/(C^{+\flat}/d^n)}$. For $n=1$, define $E_1^{+\flat}$ as the flat $C^{+\flat}/d$-algebra $E^+/p$.

    Now let $n \geq 1$ be an integer, and assume we are given for each $1 \leq m \leq n$ a flat $C^{+\flat}/d^m$\nobreakdash-algebra~$E^{+\flat}_m$, with isomorphisms $E^{+\flat}_m/d^{m-1} \cong E^{+\flat}_{m-1}$ for all $2 \leq m \leq n$. We claim that the cotangent complex $\mathbb{L}_{E^{+\flat}_n/(C^{+\flat}/d^n)}$ is in degree $0$, given by a $E^{+\flat}_n$-module of projective dimension at most $1$. Tensoring the short exact sequence
    $$0 \longrightarrow E_1^{+\flat} \xlongrightarrow{d^{n-1}} E_n^{+\flat} \longrightarrow E^{+\flat}_{n-1} \longrightarrow 0$$
    of $C^{+\flat}/d^n$-modules by the cotangent complex $\mathbb{L}_{E^{+\flat}_n/(C^{+\flat}/d^n)}$ induces a distinguished triangle
    $$\mathbb{L}_{(E^+/p)/(C^+/p)} \longrightarrow \mathbb{L}_{E^{+\flat}_n/(C^{+\flat}/d^n)} \longrightarrow \mathbb{L}_{E^{+\flat}_{n-1}/(C^{+\flat}/d^{n-1}} \longrightarrow$$
    in the derived category $\mathcal{D}(E^{+\flat}_n)$. The discreteness and the projective dimension of $\mathbb{L}_{E^{+\flat}_n/(C^{+\flat}/d^n)}$ being at most $1$ thus reduce inductively to the case $n=1$. For $n=1$, $\mathbb{L}_{(E^+/p)/(C^+/p)}$ is in degree $0$ by Proposition~\ref{Propositionquasismoothnessvaluationrings}, and the $E^+/p$-module $\Omega^1_{(E^+/p)/(C^+/p)}$ has projective dimension at most $1$.

    In particular, the $\text{Ext}^2$-group in the deformation theory Theorem~\ref{TheoremIllusieondeformations} vanishes, and so does the deformation class. This implies that there is a flat $C^{+\flat}/d^{n+1}$-algebra $E^{+\flat}_{n+1}$ with an isomorphism $E^{+\flat}_{n+1}/d^n \xrightarrow{\cong} E^{+\flat}_n$. This concludes the induction.

    Let $E^{+\flat}$ be the inverse limit of the system $(E^{+\flat}_n)_{n \geq 1}$. In particular, the $C^{+\flat}$-algebra $E^{+\flat}$ is $d$\nobreakdash-adically complete. For each integer $n \geq 2$, there is a natural exact sequence of $C^{+\flat}$-modules:
    $$0 \longrightarrow E^{+\flat}_{n-1} \xlongrightarrow{d} E^{+\flat}_n \longrightarrow E^{+\flat}_1.$$
    Passing to the inverse limit over $n \geq 2$ implies that the $C^{+\flat}$-modules $E^{+\flat}$ has no $d$-torsion, and is thus flat. By Lemma~\ref{Lemmaflatimpliesvaluationring}, $E^{+\flat}$ is thus a valuation ring extension of $C^{+\flat}$.
\end{proof}

\begin{proof}[Proof of Theorem~\ref{mainthmvaluationrings}.]
    As a flat morphism of integral domains, the valuation ring extension \hbox{$C^+ \rightarrow E^+$} is $p$\nobreakdash-Cartier smooth if and only if the morphism $C^+/p \rightarrow E^+/p$ is Cartier smooth (Remark~\ref{RemarkvaluationringssufficesmodpCSm}). The result in characteristic $p$ is Theorem~\ref{TheoremGabberandGabberRamero}, and the result when $p$ is invertible in $E^+$ is trivial. In mixed-characteristic, the cotangent smoothness of the morphism $C^+/p \rightarrow E^+/p$ is Proposition~\ref{Propositionquasismoothnessvaluationrings}. The inverse Cartier map depends only on the morphism \hbox{$C^+/p \rightarrow E^+/p$}, so we can replace $C^+$ and $E^+$ by their $p$\nobreakdash-adically separated quotients to prove the Cartier isomorphism. We assume now that $C^+$ is a mixed-characteristic $p$\nobreakdash-adically separated valuation ring whose $p$\nobreakdash-adic completion is a perfectoid ring, that $E^+$ is a $p$\nobreakdash-adically separated valuation ring extension of $C^+$, and prove the Cartier isomorphism for the morphism $C^+/p \rightarrow E^+/p$. The fraction field $E$ of the valuation ring $E^+$ is the filtered union of the finite type field extensions $E_i$ of $C$ contained in $E$. The valuation ring $E^+$ is the filtered union of the associated valuation rings $E^+_i$ and the inverse Cartier map commutes with filtered colimits, so we can assume that the field extension $C \rightarrow E$ is of finite type. 
    
    If the valuation ring $E^+$ is of finite rank, then there exists a valuation ring extension $C^{+\flat} \rightarrow E^{+\flat}$ such that $E^{+\flat}/d$ is isomorphic to $E^+/p$ (Theorem~\ref{Propositiondeformationvaluationringwhenfinitetypestarid}). The Cartier isomorphism for $C^{+\flat} \rightarrow E^{+\flat}$ (Theorem~\ref{TheoremGabberandGabberRamero}) then implies the Cartier isomorphism for $C^+/p \rightarrow E^+/p$ by base change along the morphism \hbox{$C^{+\flat} \rightarrow C^{+\flat}/d$} (Lemma~\ref{LemmabasechangeforCartiersmoothness}). 
    
    In general ({\it i.e.}, if the valuation ring $E^+$ is not necessarily of finite rank), the ideal $\mathfrak{p} := \sqrt{(p)}$ of~$E^+$ is a prime ideal, the localisation $E^+_{\mathfrak{p}}$ of $E^+$ at this prime ideal is a rank $1$ valuation ring extension of $C^+$ and the commutative diagram
    $$\xymatrix{ 
    	\relax
    	E^+ \ar[r] \ar[d] & E^+_{\mathfrak{p}} \ar[d]
    	\\
    	E^+/\mathfrak{p} \ar[r]& E^+_{\mathfrak{p}}/\mathfrak{p}E^+_{\mathfrak{p}} 	
    }$$
    is a Milnor square (\cite[Lemma~$3.12$]{huber-klawitter_differential_2018}). By \cite[Lemma~$3.14$]{huber-klawitter_differential_2018} and its proof, this Milnor square induces, for all integers $n \geq 0$, compatible short exact sequences
    $$0 \rightarrow \Omega^n_{(E^+/p)/(C^+/p)} \rightarrow \Omega^n_{(E^+_{\mathfrak{p}}/p)/(C^+/p)} \oplus \Omega^n_{(E^+/\mathfrak{p})/(C^+/\mathfrak{p}} \rightarrow \Omega^n_{(E^+_{\mathfrak{p}}/\mathfrak{p}E^+_{\mathfrak{p}})/(C^+/\mathfrak{p})} \rightarrow 0$$
    of $C^+$-modules by \cite[Lemma $3.14$]{huber-klawitter_differential_2018}. These short exact sequences induce a long exact sequence of cohomology groups:
    $$\cdots \rightarrow \text{H}^n(\Omega^{\bullet}_{(E^+/p)/(C^+/p)}) \rightarrow \text{H}^n(\Omega^\bullet_{(E^+_{\mathfrak{p}}/p)/(C^+/p)}) \oplus \text{H}^n(\Omega^\bullet_{(E^+/\mathfrak{p})/(C^+/\mathfrak{p})}) \rightarrow \text{H}^n(\Omega^\bullet_{(E^+_{\mathfrak{p}}/\mathfrak{p}E^+_{\mathfrak{p}})/(C^+/\mathfrak{p})}) \rightarrow \cdots$$
    The morphisms $C^+/\mathfrak{p} \rightarrow E^+/\mathfrak{p}$ and $C^+/\mathfrak{p} \rightarrow E^+_{\mathfrak{p}}/\mathfrak{p}E^+_{\mathfrak{p}}$ are extensions of characteristic $p$ valuation rings, and in particular satisfy the Cartier isomorphism (Theorem~\ref{TheoremGabberandGabberRamero}). This implies, via the inverse Cartier map relating the previous two exact sequences, that the morphism $C^+/p \rightarrow E^+/p$ satisfies the Cartier isomorphism if and only if the morphism $C^+/p \rightarrow E^+_{\mathfrak{p}}/p$ does. We proved in the previous paragraph that the morphism $C^+/p \rightarrow E^+_{\mathfrak{p}}/p$ satisfies the Cartier isomorphism. Hence the morphism $C^+/p \rightarrow E^+/p$ satisfies the Cartier isomorphism, which concludes the proof.
\end{proof}

\begin{remark}\label{remarkvaluationringliftuniqueness}
    Let $C^+$ be a mixed-characteristic rank $1$ perfectoid valuation ring, and $E^+$ a rank $1$ valuation ring extension of $C^+$, such that the induced field extension $C \rightarrow E$ is of finite type. In this remark, we sketch the proof of the fact that the $d$-adically complete valuation ring $E^{+\flat}$, introduced in Theorem~\ref{mainthmvaluationrings}, is unique up to isomorphism. To do so, we consider almost mathematics (\cite{gabber_almost_2003, abbes_p-adic_2016}) with respect to the pair $(C^{+\flat},\sqrt{(d)})$ (Notation~\ref{Notationalmostmaths}). The proof of Lemma~\ref{lemmacgvaluationring}, where we use \cite[Lemma~$2.4.15$]{gabber_almost_2003}, can be adapted to prove that the $E^+/p$-module $\Omega^1_{(E^+/p)/(C^+/p)}$ is almost free, and in particular almost projective. The almost deformation theory \cite[Corollary~$3.2.11$]{gabber_almost_2003} then implies that the $d$-adically complete flat lift $E^{+\flat}$ of Theorem~\ref{Propositiondeformationvaluationringwhenfinitetypestarid} must be unique when seen in the category of almost $C^{+\flat}$-algebras. One can check that the functor
    $$(-)_\ast : C^+\text{-Alg} \xlongrightarrow{(-)^a} (C^+)^a\text{-Alg} \xlongrightarrow{(-)_\ast} C^+\text{-Alg},$$
    where the second map is the right adjoint to the localisation functor from $C^+$-algebras to almost $C^+$\nobreakdash-algebras (\cite[Proposition~$2.2.13\,(ii)$]{gabber_almost_2003}), is the identity on rank $1$ valuation rings extensions of~$C^{+\flat}$. By construction, the valuation ring $E^{+\flat}$ is of rank $1$ because the valuation ring $E^+$ is of rank~$1$, hence the result.
\end{remark}

\section{The syntomic-étale comparison theorem}\label{Sectionsyntomicetalecomparison}

\vspace{-\parindent}
\hspace{\parindent}

In this section we prove Theorem~\ref{Theoremsyntomicetalecomparisonintro} (Theorems~\ref{Theoremsyntomicetalecomparison} and \ref{Theoremsyntomicetalecomparison2}), comparing the syntomic cohomology of a $p$\nobreakdash-Cartier smooth algebra over a perfectoid ring to the étale cohomology of its generic fibre. This comparison theorem was also proved in \cite{kelly_k-theory_2021,luders_milnor_2021} in characteristic $p$, and in \cite{bhatt_syntomic_2022} for $p$\nobreakdash-torsionfree $F$\nobreakdash-smooth schemes, using different methods.

\subsection{Syntomic cohomology}

\vspace{-\parindent}
\hspace{\parindent}

The syntomic complexes $\Z_p(i)^\text{syn}$ were first defined in \cite[Section $7.4$]{bhatt_topological_2019} for $p$\nobreakdash-complete $p$\nobreakdash-quasi\-syn\-tomic rings in terms of $p$\nobreakdash-completed topological cyclic homology. Another equivalent definition was given in \cite{bhatt_prisms_2022} in terms of absolute prismatic cohomology, and a generalisation for general schemes was developed in \cite[Section $8.4$]{bhatt_absolute_2022}. We will be interested mainly in $p$\nobreakdash-complete algebras over a perfectoid ring, whose syntomic complexes can be defined in terms of relative prismatic cohomology (Definition~\ref{Definitionsyntomiccomplexesoverperfectoid}).

The ideal $I \subset A$ of a perfect prism (subsection \ref{subsectionFsmoothness} or \cite[Definition~$3.2\,(2)$]{bhatt_prisms_2022}) is necessarily principal, generated by a nonzerodivisor $d \in A$. When defining a perfect prism $(A,(d))$ in 
this section, we implicitly fix a choice of generator $d \in A$. When the base prism $(A,I)$ is perfect, one can define a Nygaard filtration $\mathcal{N}^{\geq \star} \Prism_{S/A}$ on the 
prismatic complex $\Prism_{S/A}$ (without Frobenius twist), and we denote by 
$$\phi : \mathcal{N}^{\geq \star} \Prism_{S/A} \xlongrightarrow{\sim} \mathcal{N}^{\geq \star} \Prism^{(1)}_{S/A} \xlongrightarrow{\phi} I^\star \Prism_{S/A}$$
the Frobenius on the relative prismatic complex (see subsection \ref{subsectionreviewrelativeprismaticcohomology} for the definition of the second map).
Following \cite[Section $12$]{bhatt_prisms_2022} and for each $i \geq 0$, a divided Frobenius map 
$$\phi_i = \text{\textquotedblleft} \frac{\phi}{d^i} \text{\textquotedblright} : \mathcal{N}^{\geq i} \Prism_{S/A} \longrightarrow \Prism_{S/A}$$
is defined and sits in the commutative diagram
$$\xymatrix{
	\relax
	\mathcal{N}^{\geq i} \Prism_{S/A} \ar[rd]_{\phi} \ar[r]^{\phi_i} & \Prism_{S/A} \ar[d]^{d^i} \\ 
	& d^i \Prism_{S/A}. \\
}$$

\begin{definition}[Syntomic complexes]\label{Definitionsyntomiccomplexesoverperfectoid}
    Let $(A,(d))$ be a perfect prism, and $S$ a $p$\nobreakdash-complete $A/d$-algebra. For each integer $i \geq 0$, the syntomic complex $\Z_p(i)^\emph{syn}(S) \in \mathcal{D}^{\geq 0}(\Z_p)$ is 
    $$\Z_p(i)^\emph{syn}(S) := \emph{hofib}(\phi_i - 1 : \mathcal{N}^{\geq i} \Prism_{S/A} \longrightarrow \Prism_{S/A}).$$
    For each integer $n \geq 1$, also define $\Z/p^n(i)^\emph{syn}(S):=\Z_p(i)^\emph{syn}(S) \otimes_{\Z_p}^\mathbb{L} \Z/p^n$.
\end{definition}

Over a scheme in which $p$ is invertible, the object $\Z_p(i)$, called {\it $p$\nobreakdash-adic étale Tate twist}, will denote the (pro-)étale sheaf defined as the inverse limit over $n\geq 1$ of the étale sheaves $\mu_{p^n}^{\otimes i}$. Following \cite[Section~$8.3$]{bhatt_absolute_2022}, there is a map comparing the syntomic complexes and the $p$-adic étale Tate twists.

\begin{construction}[The syntomic-étale comparison map, \cite{bhatt_absolute_2022}]\label{Constructionsyntomicétalecomparison}
    Let $S$ be a $p$\nobreakdash-complete ring ({\it e.g.}, a $p$\nobreakdash-complete algebra over a perfectoid ring). For every integer~\hbox{$i \geq 0$}, there is a canonical map
    $$\gamma_\emph{syn}^\emph{ét}\{i\} : \Z_p(i)^\emph{syn}(S) \longrightarrow R\Gamma_\emph{proét}(\emph{Spec}(S[\tfrac{1}{p}]),\Z_p(i)).$$
	For every integer $n \geq 1$, one can also consider the derived reduction modulo $p^n$ of this canonical map
	$$\gamma_\emph{syn}^\emph{ét}\{i\}/p^n : \Z/p^n(i)^\emph{syn}(S) \longrightarrow R\Gamma_\emph{ét}(\emph{Spec}(S[\tfrac{1}{p}]),\mu_{p^n}^{\otimes i}).$$
\end{construction}

\begin{remark}\label{RemarkNygaarddinverted}
    Let $(A,(d))$ be a perfect prism, and $S$ an $A/d$-algebra. For each $i\geq 0$, the canonical map $$(\mathcal{N}^{\geq i} \Prism_{S/A})[\tfrac{1}{d}]/p \longrightarrow \Prism_{S/A}[\tfrac{1}{d}]/p$$
    is an equivalence in the derived category $\mathcal{D}(A)$. By left Kan extension and $p$\nobreakdash-quasisyntomic descent, it suffices to prove it for quasiregular semiperfectoid $A/d$-algebras $S$, for which there are natural inclusions of $(p,d)$-completely flat $A$-modules $\phi^{-1}_A(d)^i \Prism_{S/A} \subseteq \mathcal{N}^{\geq i} \Prism_{S/A} \subseteq \Prism_{S/A}$.
    The result follows by using the equality $(\phi^{-1}_A(d))^p = d$ in $A/p$.
\end{remark}

\begin{theorem}[Prismatic-étale comparison]\label{Theoremprismaticétalecomparison}
    Let $(A,(d))$ be a perfect prism, and $S$ a $p$\nobreakdash-complete $A/d$\nobreakdash-algebra. Then for any integers $i\geq 0$ and $n\geq 1$, there are canonical identifications
    $$R\Gamma_{\emph{ét}}(\emph{Spec}(S[\tfrac{1}{p}]),\mu_{p^n}^{\otimes i}) \xlongrightarrow{\sim} \emph{hofib}(\phi_i - 1 : \Prism_{S/A}[\tfrac{1}{d}]/p^n \longrightarrow \Prism_{S/A}[\tfrac{1}{d}]/p^n)$$
    and
    $$R\Gamma_{\emph{proét}}(\emph{Spec}(S[\tfrac{1}{p}]),\Z_p(i)) \xlongrightarrow{\sim} \emph{hofib}(\phi_i - 1 : \Prism_{S/A}[\tfrac{1}{d}]^\wedge_p \longrightarrow \Prism_{S/A}[\tfrac{1}{d}]^\wedge_p).$$
\end{theorem}

\begin{proof}
    For $i=0$ this is \cite[Theorem~$9.1$]{bhatt_prisms_2022}. When $(A,(d))$ is the perfected $q$-de Rham prism $(\Z[q^{1/p^{\infty}}]^{\wedge}_{(p,q-1)},(q-1))$, \cite[Theorem~$8.5.1$]{bhatt_absolute_2022} and its proof extend \cite[Theorem~$9.1$]{bhatt_prisms_2022} to the general case $i \geq 0$. By independence of the perfect base prism for prismatic cohomology \cite[Theorem~$5.6.2$]{bhatt_absolute_2022}, the same result also holds for any perfect base prism $(A,(d))$ such that the perfectoid ring $A/d$ admits a compatible system of $p$-powers roots of unity.

    Assume now that $(A,(d))$ is a general perfect prism. Consider a perfect prism $(A',(d))$ over $(A,(d))$ such that $A'/d$ admits a compatible system of $p$\nobreakdash-power roots of unity and such that $A/d \rightarrow A'/d$ is a $p$\nobreakdash-quasisyntomic cover (\cite[Theorem~$7.14$]{bhatt_prisms_2022}). We prove the result for torsion coefficients, by reduction to the case of $(A',(d))$ instead of $(A,(d))$; taking limits over $n \geq 1$ implies the result for integral coefficients. For any object $C \in \mathcal{D}(\Z/p^n)$ equipped with a map $F : C \rightarrow C$, denote by $C^{F=1}$ the homotopy fibre of the map $F-1 : C \rightarrow C$. 
    
    The constructions
    $$S \longmapsto R\Gamma_{\text{ét}}(\text{Spec}(S[\tfrac{1}{p}]),\mu_{p^n}^{\otimes i}), \text{ }(\Prism_{S/A}[\tfrac{1}{d}]/p^n)^{\phi_i=1}$$
    from $p$-complete $p$-quasisyntomic $A/d$-algebras $S$ to $\mathcal{D}(\Z/p^n)$, are $p$-quasisyntomic sheaves. For the first one, this is a consequence of arc-descent (\cite[Corollary~$6.17$]{bhatt_arc-topology_2021}). For the second one, this is a consequence of $p$-quasisyntomic descent for prismatic cohomology ({\it e.g.}, \cite[Construction~$1.16\,(2)$]{bhatt_prisms_2022}). In particular, to construct a natural map
    $$R\Gamma_{\text{ét}}(\text{Spec}(S[\tfrac{1}{p}]),\mu_{p^n}^{\otimes i}) \longrightarrow (\Prism_{S/A}[\tfrac{1}{d}]/p^n)^{\phi_i=1}$$
    for all $p$-complete $A/d$-algebra $S$, and proving that it is an equivalence in the derived category $\mathcal{D}(\Z/p^n)$, it suffices, by left Kan extension, to do so for $p$\nobreakdash-quasisyntomic $A/d$-algebras, and then locally in the $p$-quasisyntomic topology, {\it e.g.}, for all $p$-complete $p$-quasisyntomic $A'/d$-algebras. By \cite[Theorem~$5.6.2$]{bhatt_absolute_2022}, there are canonical equivalences
    $$\Prism_{-/A} \xlongleftarrow{\sim} \Prism_{-} \xlongrightarrow{\sim} \Prism_{-/A'}$$ 
    on $A'/d$-algebras, which concludes the proof.
\end{proof}

\subsection{The syntomic-étale comparison theorem in characteristic $p$}

\vspace{-\parindent}
\hspace{\parindent}

In this subsection we review the description of the syntomic complexes in characteristic $p$ in terms of logarithmic de Rham--Witt forms (following \cite{kelly_k-theory_2021,luders_milnor_2021}), and its consequence on the syntomic-étale comparison map (Construction~\ref{Constructionsyntomicétalecomparison}).

\begin{proposition}[Syntomic-étale comparison theorem in characteristic $p$]\label{propositionsyntomicetalecomparisonincharp}
    Let $R$ be perfect $\F_p$-algebra, and $S$ a Cartier smooth $R$-algebra. Then for any integers $i \geq 0$ and $n\geq 1$, there are canonical identifications
    $$\Z/p^n(i)^\emph{syn}(S) \simeq R\Gamma_\emph{ét}(\emph{Spec}(S),W_n\Omega^i_\emph{log})[-i]$$
    and 
    $$\Z_p(i)^\emph{syn}(S) \simeq R\Gamma_\emph{proét}(\emph{Spec}(S),W\Omega^i_\emph{log})[-i].$$
\end{proposition}

\begin{proof}
	The first equivalence follows from \cite[Proposition~$5.1\,(ii)$]{luders_milnor_2021}, the second by taking limits over~$n \geq 1$.
\end{proof}

\begin{corollary}\label{corollarysyntomicétalecomparisoncharp}
    Let $R$ be a perfect $\F_p$-algebra, and $S$ a Cartier smooth $R$-algebra. Then for any integers $i\geq 0$ and $n\geq 1$, the comparison maps
	$$\gamma_\emph{syn}^\emph{ét}\{i\}/p^n : \Z/p^n(i)^\emph{syn}(S) \longrightarrow R\Gamma_\emph{ét}(\emph{Spec}(S[\tfrac{1}{p}]),\mu_{p^n}^{\otimes i})$$
	and
	$$\gamma_\emph{syn}^\emph{ét}\{i\} : \Z_p(i)^\emph{syn}(S) \longrightarrow R\Gamma_\emph{proét}(\emph{Spec}(S[\tfrac{1}{p}]),\Z_p(i))$$
	have homotopy cofibres in degrees $\geq i-1$.
\end{corollary}

\begin{proof}
    The syntomic complexes $\Z/p^n(i)^\text{syn}(S)$ and $\Z_p(i)^\text{syn}(S)$ are in degrees $\geq i$ by Proposition~\ref{propositionsyntomicetalecomparisonincharp}, and the generic fibre $S[\tfrac{1}{p}]$ of $S$ is zero.
\end{proof}

\subsection{The syntomic-étale comparison theorem}

\vspace{-\parindent}
\hspace{\parindent}

In this subsection we prove the syntomic-étale comparison theorem over a general perfectoid base ring (Theorem~\ref{Theoremsyntomicetalecomparison}), generalising Corollary~\ref{corollarysyntomicétalecomparisoncharp}.

The following result is a direct consequence of Theorem~\ref{TheoremCartiersmoothMain}\,$(\mathcal{N}^\geq)$ when the base prism is perfect.

\begin{lemma}\label{Lemmasyntomicétalecomparisonendofproof}
    Let $(A,(d))$ be a perfect prism, and $S$ a $p$\nobreakdash-Cartier smooth $A/d$-algebra. Then for every integer $i \geq 0$, the map
    $$\tau^{\leq i} \mathcal{N}^{\geq i} \Prism_{S/A} \xlongrightarrow{\phi_i} \tau^{\leq i} \Prism_{S/A}$$ is an equivalence.
\end{lemma}

\begin{proof}
    The diagram
    $$\xymatrix{
		\relax
		\mathcal{N}^{\geq i} \Prism_{S/A} \ar[rd]^{\phi} \ar[r]^{\phi_i} & \Prism_{S/A} \ar[d]^{d^i} \\ 
		& d^i \Prism_{S/A} \\
	}$$
    is commutative for any $A/d$-algebra $S$ and the map $d^i : \Prism_{S/A} \rightarrow d^i \Prism_{S/A}$ is an equivalence. Locally on the $p$\nobreakdash-quasisyntomic site $\Prism_{S/A}$ is a $d$\nobreakdash-torsionfree $A$-module in degree $0$ and this is by definition of the divided Frobenius map $\phi_i$, and in general this is true by descent on the $p$\nobreakdash-quasisyntomic site and left Kan extension. When $S$ is $p$\nobreakdash-Cartier smooth over $A/d$, the Frobenius map $\phi : \mathcal{N}^{\geq i} \Prism_{S/A} \rightarrow d^i \Prism_{S/A}$ is an isomorphism in degrees $\leq i$, hence the result.
\end{proof}

\begin{lemma}\label{lemmakeyfrobeniuszeroondtorsion}
    Let $(A,(d))$ be a perfect prism, and $S$ a $p$-Cartier smooth $A/d$-algebra. For all integers~$i \geq 0$ and $k \leq i-1$, define the map $$\phi_i^{-1} : \emph{H}^k(\Prism_{S/A}/p) \longrightarrow \emph{H}^k(\Prism_{S/A}/p)$$
    as the composite
    $$\emph{H}^k(\Prism_{S/A}/p) \xlongrightarrow{\phi_i^{-1}} \emph{H}^k(\mathcal{N}^{\geq i} \Prism_{S/A}/p) \xlongrightarrow{\emph{can}} \emph{H}^k(\Prism_{S/A}/p)$$
    where the first map is given by the derived reduction modulo $p$ of Lemma~\ref{Lemmasyntomicétalecomparisonendofproof}. Then for every $k \leq i-2$, the map $\phi_i^{-1}$ is zero on the $d$-torsion subgroup $\emph{H}^k(\Prism_{S/A}/p)[d]$.\footnote{Equivalently, by Lemma~\ref{Lemmasyntomicétalecomparisonendofproof}, the canonical map
    $$\text{can} : \text{H}^k(\mathcal{N}^{\geq i} \Prism_{S/A}/p) \longrightarrow \text{H}^k(\Prism_{S/A}/p)$$
    is zero on the $d^{1/p}$-torsion subgroup $\text{H}^k(\mathcal{N}^{\geq i}\Prism_{S/A}/p)[d^{1/p}]$.} For $k=i-1$, the map $\phi_i^{-1} \circ \phi_i^{-1}$ is zero on the $d$-torsion subgroup $\emph{H}^{i-1}(\Prism_{S/A}/p)[d]$.
\end{lemma}

\begin{proof}
    First assume that $k \leq i-2$. In this case, the map $$\phi_i^{-1} : \text{H}^k(\Prism_{S/A}/p) \longrightarrow \text{H}^k(\Prism_{S/A}/p)$$
    can be rewritten as the map $d^{1/p} \phi_{i-1}^{-1}$, where $\phi_{i-1}^{-1}$ is defined as $\phi_i^{-1}$. Let $x$ be a $d$-torsion element of the $A/p$-module $\text{H}^k(\Prism_{S/A}/p)$. Then $$\phi_i^{-1}(x) = d^{1/p} \phi_{i-1}^{-1}(x) = \phi_{i-1}^{-1}(dx) = 0,$$
    hence the result.
    
    Now assume that $k=i-1$, and let us prove that $(\phi_i^{-1})^2 := \phi_i^{-1} \circ \phi_i^{-1}$ is the zero map on the $A/p$-module $\text{H}^{i-1}(\Prism_{S/A}/p)[d]$. There is a natural short exact sequence of $A/p$-modules
    $$0 \longrightarrow \text{H}^{i-1}(\Prism_{S/A})/p \longrightarrow \text{H}^{i-1}(\Prism_{S/A}/p) \longrightarrow \text{H}^i(\Prism_{S/A})[p] \longrightarrow 0,$$
    and compatible maps $\phi_i^{-1}$ on each of its terms. This short exact sequence induces, by the snake lemma, an exact sequence of $A/p$-modules 
    $$0 \longrightarrow (\text{H}^{i-1}(\Prism_{S/A})/p)[d] \longrightarrow \text{H}^{i-1}(\Prism_{S/A}/p)[d] \longrightarrow (\text{H}^i(\Prism_{S/A})[p])[d],$$
    and the maps $\phi_i^{-1}$ restrict to these $A/p$-modules. It suffices to prove that the map $\phi_i^{-1}$ is zero on both $(\text{H}^{i-1}(\Prism_{S/A})/p)[d]$ and $(\text{H}^i(\Prism_{S/A})[p])[d]$. Indeed, if $x$ is an element of $\text{H}^{i-1}(\Prism_{S/A}/p)[d]$, then $\phi_i^{-1}(x)$ is naturally in the kernel of the canonical map
    $$\text{H}^{i-1}(\Prism_{S/A}/p)[d] \longrightarrow (\text{H}^i(\Prism_{S/A})[p])[d].$$
    So $\phi_i^{-1}(x)$ is in the subgroup $(\text{H}^{i-1}(\Prism_{S/A})/p)[d] \subseteq \text{H}^{i-1}(\Prism_{S/A}/p)[d]$, and thus $\phi_i^{-1}(\phi_i^{-1}(x))=0$.
    
    On the $A/p$-module $\text{H}^{i-1}(\Prism_{S/A})/p$, the map $\phi_i^{-1}$ can be rewritten as the map $d^{1/p} \phi_{i-1}^{-1}$, and is thus zero on $(\text{H}^{i-1}(\Prism_{S/A})/p)[d]$, arguing as in the first paragraph of this proof. We now prove that the map $\phi_i^{-1}$ is zero on the $A/p$-module $(\text{H}^i(\Prism_{S/A})[p])[d]$. By definition of the map $\phi_i^{-1}$, it is equivalent to prove that the canonical map
    $$\text{can} : \text{H}^i(\mathcal{N}^{\geq i} \Prism_{S/A})[p] \longrightarrow \text{H}^i(\Prism_{S/A})[p]$$
    is zero on the $d^{1/p}$-torsion subgroup $(\text{H}^i(\mathcal{N}^{\geq i} \Prism_{S/A})[p])[d^{1/p}]$. By Theorem~\ref{TheoremCartiersmoothMain}, there is a commutative diagram
    $$\xymatrix{ 
		\relax
		\text{H}^i(\mathcal{N}^{\geq i} \Prism_{S/A})[p] \ar[r]^{\phi}_{\cong} \ar[d]^{\text{can}} & \text{H}^i(d^i\Prism_{S/A})[p] \ar[d]^{\text{can}}
		\\
		\text{H}^i(\Prism_{S/A})[p] \ar[r]^{\phi}_{\cong} & \text{H}^i(L\eta_d\Prism_{S/A})[p],
	}$$
    where the right vertical map is defined in \cite[Lemma~$6.9$]{bhatt_integral_2018}. By \cite[Lemma~$6.4$]{bhatt_integral_2018} and its proof, the canonical map 
    $$\text{can} : \text{H}^i(d^i\Prism_{S/A}) \longrightarrow \text{H}^i(L\eta_d \Prism_{S/A})$$
    is surjective, with kernel given by the $d$-torsion subgroup~$\text{H}^i(d^i\Prism_{S/A})[d]$. In particular, the right vertical map in the previous diagram is surjective, with kernel given by the $d$-torsion subgroup $(\text{H}^i(d^i\Prism_{S/A})[p])[d]$. The left vertical map of this diagram is thus also surjective, with kernel given by the $d^{1/p}$-torsion subgroup $(\text{H}^i(\mathcal{N}^{\geq i} \Prism_{S/A})[p])[d^{1/p}]$. Hence the result.
\end{proof}

\begin{remark}
    In the previous lemma, the result can be slightly improved if the perfectoid base ring $A/d$ is $p$-torsionfree. In this case, the map $\phi_i^{-1}$ is zero on $\text{H}^k(\Prism_{S/A}/p)[d]$, for every integer $k \leq i-1$; see the proof of Theorem~\ref{Theoremsyntomicetalecomparison2} below.
\end{remark}

\begin{theorem}[Syntomic-étale comparison theorem]\label{Theoremsyntomicetalecomparison}
    Let $R$ be a perfectoid ring, and $S$ a $p$\nobreakdash-Cartier smooth $R$-algebra. Then for any integers $i\geq 0$ and $n\geq 1$, the homotopy cofibres of the maps 
    $$\gamma_\emph{syn}^\emph{ét}\{i\}/p^n : \Z/p^n(i)^\emph{syn}(S) \longrightarrow R\Gamma_{\emph{ét}}(\emph{Spec}(S[\tfrac{1}{p}]),\mu_{p^n}^{\otimes i})$$
    and
    $$\gamma_\emph{syn}^\emph{ét}\{i\} : \Z_p(i)^\emph{syn}(S) \longrightarrow R\Gamma_{\emph{proét}}(\emph{Spec}(S[\tfrac{1}{p}]),\Z_p(i))$$ 
    are in degrees $\geq i-1$.
\end{theorem}

\begin{proof}
    We first prove the result modulo $p$. Let $i\geq 0$ be an integer, $(A,(d))$ the perfect prism corresponding to the perfectoid ring $R$ and $R\Phi(S,\F_p(i)) \in \mathcal{D}(\F_p)$ the homotopy cofibre of the comparison map $$\gamma_\text{syn}^\text{ét}\{i\}/p : \F_p(i)^\text{syn}(S) \longrightarrow R\Gamma_{\text{ét}}(\text{Spec}(S[\tfrac{1}{p}]),\mu_p^{\otimes i}).$$ 
    
    Following \cite[Section~$8.4$]{bhatt_absolute_2022}, the syntomic complex $\F_p(i)^{\text{syn}}(S)$, where $S$ is a not necessarily $p$-complete ring, is defined by the cartesian square 
    $$\xymatrix{ 
		\relax
		\F_p(i)^{\text{syn}}(S) \ar[r] \ar[d] & R\Gamma_{\text{ét}}(\text{Spec}(S[\tfrac{1}{p}]),\mu_{p}^{\otimes i}) \ar[d]
		\\
		\F_p(i)^{\text{syn}}(S^\wedge_p) \ar[r] & R\Gamma_{\text{ét}}(\text{Spec}(S^\wedge_p[\tfrac{1}{p}]),\mu_{p}^{\otimes i}),
	}$$
    where the bottom horizontal map is the syntomic-étale comparison map (Construction~\ref{Constructionsyntomicétalecomparison}). In particular, the homotopy cofibre of the top horizontal map is naturally identified with that of the bottom horizontal map. The desired statement depending only on this homotopy cofibre, we assume for the rest of the proof that the ring $S$ is $p$-complete.
    
    By Theorem~\ref{Theoremprismaticétalecomparison} and Remark~\ref{RemarkNygaarddinverted}, there is a commutative diagram
    $$\xymatrix{ 
		\relax
		\F_p(i)^\text{syn}(S) \ar[r] \ar[d] & \mathcal{N}^{\geq i} \Prism_{S/A}/p \ar[r]^{\phi_i -1} \ar[d]^{\lambda_i} & \Prism_{S/A}/p \ar[d]^{\lambda}
		\\
		R\Gamma_{\text{ét}}(\text{Spec}(S[\tfrac{1}{p}]),\mu_p^{\otimes i}) \ar[r] \ar[d] & \Prism_{S/A}[\tfrac{1}{d}]/p \ar[r]^{\phi_i -1} \ar[d] & \Prism_{S/A}[\tfrac{1}{d}]/p \ar[d]
		\\
		R\Phi(S,\F_p(i)) \ar[r] & \text{hocofib}(\lambda_i) \ar[r]^{\phi_i -1} & \text{hocofib}(\lambda)
	}$$
    where all the horizontal maps are homotopy fibre sequences. The maps $\lambda_i$ and $\lambda$ correspond to inverting~$d$. 
    
    We want to prove that $R\Phi(S,\F_p(i)) \in \mathcal{D}^{\geq i-1}(\F_p)$, {\it i.e.}, that $R\Phi(S,\F_p(i))$ is zero in degrees $\leq i-2$. This statement depends only on
    $$\tau^{\leq i-2} \text{hocofib}(\lambda_i) \xlongrightarrow{\phi_i-1} \tau^{\leq i-2} \text{hocofib}(\lambda),$$
    and thus only on the commutative diagram
    $$\xymatrix{ 
		\relax
		\tau^{\leq i-1} (\mathcal{N}^{\geq i} \Prism_{S/A}/p) \ar[r]^{\phi_i -1} \ar[d]^{\tau^{\leq i-1}\lambda_i} & \tau^{\leq i-1}(\Prism_{S/A}/p) \ar[d]^{\tau^{\leq i-1}\lambda}
		\\
		\tau^{\leq i-1} (\Prism_{S/A}[\tfrac{1}{d}]/p) \ar[r]^{\phi_i -1} & \tau^{\leq i-1} (\Prism_{S/A}[\tfrac{1}{d}]/p).
	}$$
	In terms of this diagram, we want to prove that the map $$\text{hocofib}(\tau^{\leq i-1} \lambda_i) \xlongrightarrow{\phi_i-1} \text{hocofib}(\tau^{\leq i-1} \lambda)$$ is an isomorphism in degrees $\leq i-3$ and is injective in degree $i-2$.
    
    By Lemma~\ref{Lemmasyntomicétalecomparisonendofproof}, the divided Frobenius map $\phi_i : \mathcal{N}^{\geq i} \Prism_{S/A}/p \rightarrow \Prism_{S/A}/p$ is an isomorphism in degrees~$\leq i-1$. Define $$1-\phi_i^{-1} : \tau^{\leq i-1} (\Prism_{S/A}/p) \longrightarrow \tau^{\leq i-1} (\Prism_{S/A}/p)$$ as the map $\phi_i - 1 : \tau^{\leq i-1} (\mathcal{N}^{\geq i} \Prism_{S/A}/p) \rightarrow \tau^{\leq i-1} (\Prism_{S/A}/p)$ precomposed with the inverse of the divided Frobenius map $\phi_i : \tau^{\leq i-1} (\mathcal{N}^{\geq i} \Prism_{S/A}/p) \xrightarrow{\sim} \tau^{\leq i-1} (\Prism_{S/A}/p)$. Similarly the divided Frobenius map $\phi_i : \Prism_{S/A}[\tfrac{1}{d}]/p \rightarrow \Prism_{S/A}[\tfrac{1}{d}]/p$ is an equivalence by Theorem~\ref{TheoremCartiersmoothMain}\,$(L\eta)$, and we define $$1 - \phi_i^{-1} : \tau^{\leq i-1}(\Prism_{S/A}[\tfrac{1}{d}]/p) \longrightarrow \tau^{\leq i-1}(\Prism_{S/A}[\tfrac{1}{d}]/p)$$ as the map $\phi_i-1 : \tau^{\leq i-1}(\Prism_{S/A}[\tfrac{1}{d}]/p) \rightarrow \tau^{\leq i-1}(\Prism_{S/A}[\tfrac{1}{d}]/p)$ precomposed with the inverse of the divided Frobenius map $\phi_i : \tau^{\leq i-1} (\Prism_{S/A}[\tfrac{1}{d}]/p) \xrightarrow{\sim} \tau^{\leq i-1} (\Prism_{S/A}[\tfrac{1}{d}]/p)$. We thus want to prove that the induced map $$\text{hocofib}(\tau^{\leq i-1} \lambda) \xlongrightarrow{1-\phi_i^{-1}} \text{hocofib}(\tau^{\leq i-1} \lambda)$$ is an isomorphism in degrees $\leq i-3$ and is injective in degree $i-2$. We prove that it is an isomorphism in degrees $\leq i-2$.
    
    Let $k$ be an integer $\leq i-2$. There is a short exact sequence of $A/p$-modules:
    $$0 \longrightarrow \text{Coker}(\text{H}^k(\lambda)) \longrightarrow \text{H}^k(\text{hocofib}(\lambda)) \longrightarrow \text{Ker}(\text{H}^{k+1}(\lambda)) \longrightarrow 0.$$
    To prove that $1-\phi_i^{-1}$ is an isomorphism on $\text{H}^k(\text{hocofib}(\lambda))$, let us prove that it is an isomorphism on both $\text{Ker}(\text{H}^{k+1}(\lambda))$ and $\text{Coker}(\text{H}^k(\lambda))$.
    
    First consider the map
    $$\text{H}^{k+1}(\lambda) : \text{H}^{k+1}(\Prism_{S/A}/p) \longrightarrow \text{H}^{k+1}(\Prism_{S/A}/p)[\tfrac{1}{d}].$$
    Its kernel is given by the $d$-power torsion subgroup of $\text{H}^{k+1}(\Prism_{S/A}/p)$, so we want to prove that the map $1-\phi_i^{-1}$ is an isomorphism on the $A/p$-module $\text{H}^{k+1}(\Prism_{S/A}/p)[d^\infty]$. The relation $\phi_A(d)=d^p$ holds in the ring $A/p$. So for any integer $j \geq 1$ and any element $x \in \text{H}^{k+1}(\Prism_{S/A}/p)[d^{p^j}]$:
    $$d^{p^{j-1}} \phi_i^{-1}(x) = \phi_i^{-1}(d^{p^j}x) = 0,$$
    and thus $\phi_i^{-1}(x) \in \text{H}^{k+1}(\Prism_{S/A}/p)[d^{p^{j-1}}]$. By Lemma~\ref{lemmakeyfrobeniuszeroondtorsion}, the map $(\phi_i^{-1})^2$ is zero on the $A/p$-module $\text{H}^{k+1}(\Prism_{S/A}/p)[d]$. So the map $1-\phi_i^{-1}$ is an isomorphism on the $A/p$-module $\text{H}^{k+1}(\Prism_{S/A}/p)[d^{p^j}]$, for each integer $j \geq 1$, with inverse given by the map $$1 + \phi_i^{-1} + \dots + (\phi_i^{-1})^{j+1}.$$
    So the map $1-\phi_i^{-1}$ is an isomorphism on the $A/p$-module $\text{H}^{k+1}(\Prism_{S/A}/p)[d^\infty]$.
    
    Now consider the map
    $$\text{H}^k(\lambda) : \text{H}^k(\Prism_{S/A}/p) \longrightarrow \text{H}^k(\Prism_{S/A}/p)[\tfrac{1}{d}].$$
    It is the filtered colimit over $m \geq 0$ of the maps $$\lambda^{(m)} : \Prism_{S/A}/p \longrightarrow \Prism_{S/A}/p$$ given by multiplication by $d^m$. The $A/p$-module $\text{Coker}(\text{H}^k(\lambda))$ and the map $1-\phi_i^{-1}$ acting on it can be rewritten as the colimit over $m \geq 0$ of the $A/p$-modules $\text{H}^k(\Prism_{S/A}/p)/d^m$ with maps $1 - d^{\tfrac{(p-1)m}{p}}\phi_i^{-1}$. Let $m \geq 0$ be an integer. We claim that the map
    $$d^{\tfrac{(p-1)m}{p}}\phi_i^{-1} : \text{H}^k(\Prism_{S/A}/p)/d^m \longrightarrow \text{H}^k(\Prism_{S/A}/p)/d^m$$ is nilpotent. Because $k \leq i-2$, this map $d^{\tfrac{(p-1)m}{p}}\phi_i^{-1}$ is naturally identified with the map $$d^{\tfrac{(p-1)m}{p}+\tfrac{1}{p}}\phi_{i-1}^{-1} : \text{H}^k(\Prism_{S/A}/p)/d^m \longrightarrow \text{H}^k(\Prism_{S/A}/p)/d^m,$$ where $$\phi_{i-1} : \tau^{\leq i-1} \mathcal{N}^{\geq i-1} \Prism_{S/A} \longrightarrow \tau^{\leq i-1} \Prism_{S/A}$$ is the equivalence of Lemma~\ref{Lemmasyntomicétalecomparisonendofproof}. Composing with itself $k$ times for some integer $k\geq 1$ and using the $\phi_A^{-1}$-linearity of $\phi_{i-1}^{-1}$ gives the map $d^{\tfrac{(p^k-1)m}{p^k} + \tfrac{p^k-1}{p^k(p-1)}}(\phi_{i-1}^{-1})^k$. This map is zero for any integer $k\geq 1$ satisfying $\tfrac{(p^k-1)m}{p^k} + \tfrac{p^k-1}{p^k(p-1)} \geq m$, that is $p^k \geq m(p-1)+1$. The map $1 - d^{\frac{(p-1)m}{p}} \phi_i^{-1}$ is thus a sum of an isomorphism and a nilpotent map, so it is an equivalence. Taking colimit over $m \geq 0$, the map $1-\phi_i^{-1}$ is an isomorphism on $\text{Coker}(\text{H}^k(\lambda))$. This concludes the proof of the result modulo $p$.

    We now prove the result for integral coefficients. Let $R\Phi(S,\Z_p(i)) \in \mathcal{D}(\Z_p)$ be the homotopy cofibre of the comparison map
    $$\gamma_\text{syn}^\text{ét}\{i\} : \Z_p(i)^\text{syn}(S) \longrightarrow R\Gamma_{\text{proét}}(\text{Spec}(S[\tfrac{1}{p}]),\Z_p(i)).$$
    We want to prove that $R\Phi(S,\Z_p(i)) \in \mathcal{D}^{\geq i-1}(\Z_p)$, {\it i.e.}, that $\tau^{\leq i-2} R\Phi(S,\Z_p(i)) \simeq 0$. The truncation of a derived $p$-complete object is derived $p$-complete (\cite[091N]{stacks_project_authors_stacks_2019}). By derived Nakayama, it thus suffices to prove that $$(\tau^{\leq i-2} R\Phi(S,\Z_p(i)))/p \simeq 0.$$ For every integer $k \leq i-3$, the natural map
    $$\text{H}^k((\tau^{\leq i-2} R\Phi(S,\Z_p(i)))/p) \longrightarrow \text{H}^k(R\Phi(S,\Z_p(i))/p) \cong \text{H}^k(R\Phi(S,\F_p(i)))$$
    is an isomorphism, and its target is zero by the first part of the proof. In degree $i-2$, the cohomology group $\text{H}^{i-2}((\tau^{\leq i-2} R\Phi(S,\Z_p(i)))/p)$ is naturally identified with the (classical) reduction modulo $p$ of the cohomology group $\text{H}^{i-2}(R\Phi(S,\Z_p(i)))$, and there is a short exact sequence of abelian groups:
    $$0 \longrightarrow \text{H}^{i-2}(R\Phi(S,\Z_p(i)))/p \longrightarrow \text{H}^{i-2}(R\Phi(S,\F_p(i))) \longrightarrow \text{H}^{i-1}(R\Phi(S,\Z_p(i)))[p] \longrightarrow 0.$$
    The middle term of this short exact sequence is zero by the first part of the proof; so the left one also is, which concludes the proof for integral coefficients. The result modulo $p^n$ can be proved like the result for integral coefficients, or deduced from it by reduction modulo $p^n$.
\end{proof}

\subsection{The syntomic-étale comparison theorem over a $p$\nobreakdash-torsionfree base}

\vspace{-\parindent}
\hspace{\parindent}

In this subsection we prove a refined version of the syntomic-étale comparison theorem (Theorem~\ref{Theoremsyntomicetalecomparison2}), assuming the perfectoid base ring is $p$\nobreakdash-torsionfree.

\begin{lemma}\label{LemmadividedFrobptorsionfree}
    Let $(A,(d))$ be a perfect prism such that $A/d$ is $p$\nobreakdash-torsionfree, and $S$ a $p$\nobreakdash-Cartier smooth $A/d$-algebra. Then for every integer $i\geq 0$, the Frobenius maps and divided Frobenius map
    $$\phi : \Prism_{S/A}/p \longrightarrow L\eta_d (\Prism_{S/A}/p)$$
    $$\phi : \tau^{\leq i} (\mathcal{N}^{\geq i} \Prism_{S/A}/p) \longrightarrow \tau^{\leq i}(d^i\Prism_{S/A}/p)$$
    $$\phi_i : \tau^{\leq i} (\mathcal{N}^{\geq i} \Prism_{S/A}/p) \longrightarrow \tau^{\leq i}(\Prism_{S/A}/p)$$
    are equivalences.
\end{lemma}

\begin{proof}
    By Theorem~\ref{TheoremCartiersmoothMain}\,$(L\eta)$, the Frobenius map $$\phi : \Prism_{S/A} \longrightarrow L\eta_d \Prism_{S/A}$$
    is an equivalence, thus so is its derived reduction modulo $p$. Note that $p$ is a nonzerodivisor in the ring~$A$ (\cite[Lemma~$2.28\,(1)$]{bhatt_prisms_2022}). Moreover, $S$ is $p$\nobreakdash-cotangent smooth over $A/d$, so the groups $(\Omega^n_{S/(A/d)})^\wedge_p$ are $p$-flat modules over the $p$\nobreakdash-torsionfree ring $A/d$ and are in particular $p$\nobreakdash-torsionfree. The groups~$\text{H}^n(\overline{\Prism}_{S/A})$ are thus also $p$\nobreakdash-torsionfree (Proposition~\ref{propositionpquasismoothproperties}\,$(3)$), and the natural map $$(L\eta_d \Prism_{S/A})/p \longrightarrow L\eta_d (\Prism_{S/A}/p)$$ is an equivalence in the derived category $\mathcal{D}(A/p)$ (\cite[Lemma~$5.16$]{bhatt_specializing_2018}), which proves the first statement.
    
    The proof of Theorem~\ref{TheoremCartiersmoothMain}\,$(L\eta)\Rightarrow(\mathcal{N}^{\geq})$, where we use that the short exact sequence 
    $$0 \longrightarrow \text{H}^{-1}(\Prism_{S/A}/p)/d \longrightarrow \text{H}^{-1}(\Prism_{S/A}/(p,d)) \longrightarrow \text{H}^0(\Prism_{S/A}/p)[d] \longrightarrow 0$$
    to prove that $\text{H}^0(\Prism_{S/A}/p)$ is $d$\nobreakdash-torsionfree, then adapts readily to prove that the Frobenius map
    $$\phi : \tau^{\leq i} (\mathcal{N}^{\geq i} \Prism_{S/A}/p) \longrightarrow \tau^{\leq i}(d^i\Prism_{S/A}/p)$$
    is an equivalence in the derived category $\mathcal{D}(A/p)$. The proof of the third statement is the same as in Lemma~\ref{Lemmasyntomicétalecomparisonendofproof}.
\end{proof}

\begin{theorem}[Syntomic-étale comparison theorem over a $p$\nobreakdash-torsionfree base]\label{Theoremsyntomicetalecomparison2}
    Let $R$ be a $p$\nobreakdash-torsionfree perfectoid ring, and $S$ a $p$\nobreakdash-Cartier smooth $R$-algebra. Then for any integers $i\geq 0$ and $n\geq 1$, the homotopy cofibres of the maps 
    $$\gamma_\emph{syn}^\emph{ét}\{i\}/p^n : \Z/p^n(i)^\emph{syn}(S) \longrightarrow R\Gamma_{\emph{ét}}(\emph{Spec}(S[\tfrac{1}{p}]),\mu_{p^n}^{\otimes i})$$
    and
    $$\gamma_\emph{syn}^\emph{ét}\{i\} : \Z_p(i)^\emph{syn}(S) \longrightarrow R\Gamma_{\emph{proét}}(\emph{Spec}(S[\tfrac{1}{p}]),\Z_p(i))$$ are in degrees $\geq i$.
\end{theorem}

\begin{proof}
    As in the proof of Theorem~\ref{Theoremsyntomicetalecomparison}, we first reduce to the case where $S$ is a $p$-complete ring and, by derived Nakayama, it suffices to prove the result modulo $p$. We keep the same notation as in the proof of Theorem~\ref{Theoremsyntomicetalecomparison}. We want to prove that $$R\Phi(S,\F_p(i)) \in \mathcal{D}^{\geq i}(\F_p),$$ {\it i.e.}, that $R\Phi(S,\F_p(i))$ is zero in degrees $\leq i-1$. This statement depends only on
    $$\tau^{\leq i-1} \text{hocofib}(\lambda_i) \xlongrightarrow{\phi_i-1} \tau^{\leq i-1} \text{hocofib}(\lambda),$$
    and thus only on the commutative diagram
    $$\xymatrix{ 
		\relax
		\tau^{\leq i} (\mathcal{N}^{\geq i} \Prism_{S/A}/p) \ar[r]^{\phi_i -1} \ar[d]^{\tau^{\leq i}\lambda_i} & \tau^{\leq i}(\Prism_{S/A}/p) \ar[d]^{\tau^{\leq i}\lambda}
		\\
		\tau^{\leq i} (\Prism_{S/A}[\tfrac{1}{d}]/p) \ar[r]^{\phi_i -1} & \tau^{\leq i} (\Prism_{S/A}[\tfrac{1}{d}]/p).
	}$$
	In terms of this diagram, we want to prove that the map $$\text{hocofib}(\tau^{\leq i}\lambda_i) \xlongrightarrow{\phi_i-1} \text{hocofib}(\tau^{\leq i}\lambda)$$ is an isomorphism in degrees $\leq i-2$ and is injective in degree $i-1$.
	
	The divided Frobenius map $\phi_i : \mathcal{N}^{\geq i} \Prism_{S/A}/p \rightarrow \Prism_{S/A}/p$ is an isomorphism in degrees $\leq i$ by Lemma~\ref{LemmadividedFrobptorsionfree}. Define $$1-\phi_i^{-1} : \tau^{\leq i} (\Prism_{S/A}/p) \longrightarrow \tau^{\leq i} (\Prism_{S/A}/p)$$ as the map $\phi_i - 1 : \tau^{\leq i} (\mathcal{N}^{\geq i} \Prism_{S/A}/p) \rightarrow \tau^{\leq i} (\Prism_{S/A}/p)$ precomposed with the inverse of the divided Frobenius map $\phi_i : \tau^{\leq i} (\mathcal{N}^{\geq i} \Prism_{S/A}/p) \xrightarrow{\sim} \tau^{\leq i} (\Prism_{S/A}/p)$. Similarly the divided Frobenius map \hbox{$\phi_i : \Prism_{S/A}[\tfrac{1}{d}]/p \rightarrow \Prism_{S/A}[\tfrac{1}{d}]/p$} is an equivalence (Theorem~\ref{TheoremCartiersmoothMain}\,$(L\eta)$), and we define $$1 - \phi_i^{-1} : \tau^{\leq i}(\Prism_{S/A}[\tfrac{1}{d}]/p) \longrightarrow \tau^{\leq i}(\Prism_{S/A}[\tfrac{1}{d}]/p)$$ as the map $\phi_i-1 : \tau^{\leq i}(\Prism_{S/A}[\tfrac{1}{d}]/p) \rightarrow \tau^{\leq i}(\Prism_{S/A}[\tfrac{1}{d}]/p)$ precomposed with the inverse of the divided Frobenius map $\phi_i : \tau^{\leq i} (\Prism_{S/A}[\tfrac{1}{d}]/p) \xrightarrow{\sim} \tau^{\leq i} (\Prism_{S/A}[\tfrac{1}{d}]/p)$. We thus want to prove that the induced map $$\text{hocofib}(\tau^{\leq i}\lambda) \xlongrightarrow{1-\phi_i^{-1}} \text{hocofib}(\tau^{\leq i}\lambda)$$ is an isomorphism in degrees $\leq i-2$ and is injective in degree $i-1$. We prove that it is an isomorphism in degrees $\leq i-1$.
	
	Let $k$ be an integer $\leq i-1$. There is a short exact sequence of $A/p$-modules:
    $$0 \longrightarrow \text{Coker}(\text{H}^k(\lambda/p)) \longrightarrow \text{H}^k(\text{hocofib}(\lambda/p)) \longrightarrow \text{Ker}(\text{H}^{k+1}(\lambda/p)) \longrightarrow 0.$$
    To prove that $1-\phi_i^{-1}$ is an isomorphism on $\text{H}^k(\text{hocofib}(\lambda/p))$, let us prove that it is an isomorphism on both $\text{Coker}(\text{H}^k(\lambda/p))$ and $\text{Ker}(\text{H}^{k+1}(\lambda/p))$.

    For $\text{Coker}(\text{H}^k(\lambda/p))$, the argument is the same as in the proof of Theorem~\ref{Theoremsyntomicetalecomparison}, where we need Lemma~\ref{LemmadividedFrobptorsionfree} for the case of $k=i-1$.
    
    For $\text{Ker}(\text{H}^{k+1}(\lambda/p))$, we follow the lines of the proof of Theorem~\ref{Theoremsyntomicetalecomparison}. It suffices to prove the case of $k=i-1$, {\it i.e.}, that $1-\phi_i^{-1}$ is an isomorphism on the $A/p$-module $\text{H}^i(\Prism_{S/A}/p)[d^\infty]$. It then suffices to prove that $\phi_i^{-1}$ is nilpotent on the $A/p$-module $\text{H}^i(\Prism_{S/A}/p)[d]$; we prove that it is zero. By definition of the map $\phi_i^{-1}$, it is equivalent to proving that the canonical map
    $$\text{can} : \text{H}^i(\mathcal{N}^{\geq i}\Prism_{S/A}/p) \longrightarrow \text{H}^i(\Prism_{S/A}/p)$$
    is zero on the $d^{1/p}$-torsion subgroup $\text{H}^i(\mathcal{N}^{\geq i}\Prism_{S/A}/p)[d^{1/p}]$. By Lemma~\ref{LemmadividedFrobptorsionfree}, there is a commutative diagram
    $$\xymatrix{ 
		\relax
		\text{H}^i((\mathcal{N}^{\geq i} \Prism_{S/A})/p) \ar[r]^{\phi}_{\cong} \ar[d]^{\text{can}} & \text{H}^i(d^i\Prism_{S/A}/p) \ar[d]^{\text{can}}
		\\
		\text{H}^i(\Prism_{S/A}/p) \ar[r]^{\phi}_{\cong} & \text{H}^i(L\eta_d(\Prism_{S/A}/p)),
	}$$
    where the right vertical map is defined in \cite[Lemma~$6.9$]{bhatt_integral_2018}. As in the proof of \ref{lemmakeyfrobeniuszeroondtorsion}, the right vertical map of this diagram is surjective, with kernel given by the $d$-torsion subgroup of $\text{H}^i(d^i\Prism_{S/A}/p)$. So the left vertical map is also surjective, with kernel given by the $d^{1/p}$-torsion subgroup $\text{H}^i((\mathcal{N}^{\geq i} \Prism_{S/A})/p)[d^{1/p}]$. Hence the result.
\end{proof}

\begin{remark}[Comparison with \cite{bhatt_topological_2019}]\label{RemarkcomparisonBMS2smoothoverOC}
    Let $C$ be a complete and algebraically closed extension of $\Q_p$, and $\mathcal{O}_C$ be its ring of integers. In particular, $\mathcal{O}_C$ is a $p$-torsionfree perfectoid ring. When $S$ is a smooth $\mathcal{O}_C$-algebra, the previous result was already proved by Bhatt--Morrow--Scholze (\cite[Theorem~$10.1$]{bhatt_topological_2019}). In this situation, their result is slighlty stronger, as they prove that the syntomic-étale comparison map is an isomorphism in degree $i$ (and is thus not only injective). Note that this fact does not hold for general $p$-Cartier smooth $\mathcal{O}_C$-algebras, {\it e.g.}, for general valuation ring extensions of~$\mathcal{O}_C$. Their method uses crucially the functor $L\eta_{\mu}$, and in particular the existence of a compatible system of $p$-power roots of unity in the perfectoid ring $\mathcal{O}_C$.
\end{remark}

\begin{remark}
    Without assuming that the perfectoid base is $p$\nobreakdash-torsionfree, the previous result would be false: for instance, in characteristic $p$, the homotopy cofibre of the syntomic-étale comparison map is typically nonzero in degree $i-1$ (Proposition~\ref{propositionsyntomicetalecomparisonincharp}).
\end{remark}

\nocite{*}

\bibliographystyle{alpha}

\begin{thebibliography}{AMMN20}
	
	\bibitem[AGT16]{abbes_p-adic_2016}
	Ahmed Abbes, Michel Gros, and Takeshi Tsuji.
	\newblock {\em The {$p$}-adic {S}impson correspondence}, volume 193 of {\em Annals of Mathematics Studies}.
	\newblock Princeton University Press, Princeton, NJ, 2016.
	
	\bibitem[AMMN20]{antieau_beilinson_2020}
	Benjamin Antieau, Akhil Mathew, Matthew Morrow, and Thomas Nikolaus.
	\newblock On the {B}eilinson fiber square.
	\newblock {\em Duke Math. J.}, 2020.
	
	\bibitem[Avr99]{avramov_locally_1999}
	Luchezar~L. Avramov.
	\newblock Locally complete intersection homomorphisms and a conjecture of {Q}uillen on the vanishing of cotangent homology.
	\newblock {\em Ann. of Math.}, 150(2):455--487, 1999.
	
	\bibitem[Bha12]{bhatt_p-adic_2012}
	Bhargav Bhatt.
	\newblock {$p$-adic derived de Rham cohomology}.
	\newblock \url{https://arxiv.org/abs/1204.6560}, 2012.
	
	\bibitem[Bha13]{bhatt_imperfect_2013}
	Bhargav Bhatt.
	\newblock {An imperfect ring with trivial cotangent complex}.
	\newblock \url{https://www.math.ias.edu/~bhatt/math/trivial-cc.pdf}, 2013.
	
	\bibitem[Bha18]{bhatt_specializing_2018}
	Bhargav Bhatt.
	\newblock Specializing varieties and their cohomology from characteristic 0 to characteristic $p$.
	\newblock {\em Algebraic Geometry: Salt Lake City 2015}, 2018.
	
	\bibitem[BL22]{bhatt_absolute_2022}
	Bhargav Bhatt and Jacob Lurie.
	\newblock Absolute prismatic cohomology.
	\newblock \url{https://arxiv.org/abs/2201.06120}, 2022.
	
	\bibitem[BLM21]{bhatt_revisiting_2021}
	Bhargav Bhatt, Jacob Lurie, and Akhil Mathew.
	\newblock Revisiting the de {R}ham--{W}itt complex.
	\newblock {\em Ast\'{e}risque}, 424, 2021.
	
	\bibitem[BM21]{bhatt_arc-topology_2021}
	Bhargav Bhatt and Akhil Mathew.
	\newblock The arc topology.
	\newblock {\em Duke Math. J.}, 170(9):1899--1988, 2021.
	
	\bibitem[BM22]{bhatt_syntomic_2022}
	Bhargav Bhatt and Akhil Mathew.
	\newblock {Syntomic complexes and $p$-adic {\'e}tale Tate twists}.
	\newblock \url{https://arxiv.org/abs/2202.04818}, 2022.
	
	\bibitem[BMS18]{bhatt_integral_2018}
	Bhargav Bhatt, Matthew Morrow, and Peter Scholze.
	\newblock Integral {$p$}-adic {H}odge theory.
	\newblock {\em Publ. Math. Inst. Hautes \'{E}tudes Sci.}, 128:219--397, 2018.
	
	\bibitem[BMS19]{bhatt_topological_2019}
	Bhargav Bhatt, Matthew Morrow, and Peter Scholze.
	\newblock Topological {H}ochschild homology and integral {$p$}-adic {H}odge theory.
	\newblock {\em Publ. Math. Inst. Hautes \'{E}tudes Sci.}, 129:199--310, 2019.
	
	\bibitem[BS17]{bhatt_projectivity_2017}
	Bhargav Bhatt and Peter Scholze.
	\newblock Projectivity of the {W}itt vector affine {G}rassmannian.
	\newblock {\em Invent. Math.}, 209(2):329--423, 2017.
	
	\bibitem[BS21]{bhatt_prismatic_2021}
	Bhargav Bhatt and Peter Scholze.
	\newblock {Prismatic $F$-crystals and crystalline Galois representations}.
	\newblock \url{https://arxiv.org/abs/2106.14735}, 2021.
	
	\bibitem[BS22]{bhatt_prisms_2022}
	Bhargav Bhatt and Peter Scholze.
	\newblock Prisms and prismatic cohomology.
	\newblock {\em Ann. of Math.}, 2022.
	
	\bibitem[CMM21]{clausen_k-theory_2021}
	Dustin Clausen, Akhil Mathew, and Matthew Morrow.
	\newblock {$K$}-theory and topological cyclic homology of henselian pairs.
	\newblock {\em J. Amer. Math. Soc.}, 34(2):411--473, 2021.
	
	\bibitem[{\v C}S21]{cesnavicius_purity_2021}
	K{\k e}stutis {\v C}esnavi{\v c}ius and Peter Scholze.
	\newblock {Purity for flat cohomology}.
	\newblock \url{https://arxiv.org/abs/1912.10932}, 2021.
	
	\bibitem[DI87]{deligne_relevements_1987}
	Pierre Deligne and Luc Illusie.
	\newblock Rel{\`e}vements modulo {$p^2$} et d\'{e}composition du complexe de de {R}ham.
	\newblock {\em Invent. Math.}, 89(2):247--270, 1987.
	
	\bibitem[EHIK21]{elmanto_cdh_2021}
	Elden Elmanto, Marc Hoyois, Ryomei Iwasa, and Shane Kelly.
	\newblock Cdh descent, cdarc descent, and {M}ilnor excision.
	\newblock {\em Math. Ann.}, 379(3-4):1011--1045, 2021.
	
	\bibitem[Gab92]{gabber_k-theory_1992}
	Ofer Gabber.
	\newblock {$K$}-theory of {H}enselian local rings and {H}enselian pairs.
	\newblock In {\em Algebraic {$K$}-theory, commutative algebra, and algebraic geometry ({S}anta {M}argherita {L}igure, 1989)}, volume 126 of {\em Contemp. Math.}, pages 59--70. Amer. Math. Soc., Providence, RI, 1992.
	
	\bibitem[GL00]{geisser_k-theory_2000}
	Thomas Geisser and Marc Levine.
	\newblock The {$K$}-theory of fields in characteristic {$p$}.
	\newblock {\em Invent. Math.}, 139(3):459--493, 2000.
	
	\bibitem[GL01]{geisser_bloch-kato_2001}
	Thomas Geisser and Marc Levine.
	\newblock The {B}loch--{K}ato conjecture and a theorem of {S}uslin--{V}oevodsky.
	\newblock {\em J. Reine Angew. Math.}, 530:55--103, 2001.
	
	\bibitem[GR03]{gabber_almost_2003}
	Ofer Gabber and Lorenzo Ramero.
	\newblock {\em Almost ring theory}, volume 1800 of {\em Lecture Notes in Math.}
	\newblock Springer-Verlag, Berlin, 2003.
	
	\bibitem[HKK18]{huber-klawitter_differential_2018}
	Annette Huber-Klawitter and Shane Kelly.
	\newblock Differential forms in positive characteristic, {II}: cdh-descent via functorial {Riemann}–{Zariski} spaces.
	\newblock {\em Algebra \& Number Theory}, 12(3):649--692, 2018.
	
	\bibitem[Hub96]{huber_adic_1996}
	Roland Huber.
	\newblock Adic spaces.
	\newblock In {\em \'{E}tale cohomology of rigid analytic varieties and adic spaces}, volume E30 of {\em Aspects of Mathematics}, pages 36--107. Friedr. Vieweg \& Sohn, Braunschweig, 1996.
	
	\bibitem[Ill71]{illusie_complexe_1971}
	Luc Illusie.
	\newblock {\em Complexe cotangent et d\'{e}formations. {I}}, volume 239 of {\em Lecture Notes in Math.}
	\newblock Springer-Verlag, Berlin-New York, 1971.
	
	\bibitem[KM21]{kelly_k-theory_2021}
	Shane Kelly and Matthew Morrow.
	\newblock {$K$}-theory of valuation rings.
	\newblock {\em Compos. Math.}, 157(6):1121--1142, 2021.
	
	\bibitem[KST21]{kerz_towards_2021}
	Moritz Kerz, Florian Strunk, and Georg Tamme.
	\newblock Towards {V}orst's conjecture in positive characteristic.
	\newblock {\em Compos. Math.}, 157(6):1143--1171, 2021.
	
	\bibitem[LM21]{luders_milnor_2021}
	Morten L{\"u}ders and Matthew Morrow.
	\newblock {Milnor {$K$}-theory of $p$-adic rings}.
	\newblock \url{https://arxiv.org/abs/2101.01092}, 2021.
	
	\bibitem[Nyg81]{nygaard_slopes_1981}
	Niels~O. Nygaard.
	\newblock Slopes of powers of {F}robenius on crystalline cohomology.
	\newblock {\em Ann. Sci. \'{E}cole Norm. Sup.}, 14(4):369--401, 1981.
	
	\bibitem[Sch12]{scholze_perfectoid_2012}
	Peter Scholze.
	\newblock Perfectoid spaces.
	\newblock {\em Publ. Math. Inst. Hautes \'{E}tudes Sci.}, 116:245--313, 2012.
	
	\bibitem[{Sta}19]{stacks_project_authors_stacks_2019}
	The {Stacks Project Authors}.
	\newblock \textit{Stacks Project}.
	\newblock \url{https://stacks.math.columbia.edu}, 2019.
	
	\bibitem[Sus83]{suslin_k-theory_1983}
	Andrei Suslin.
	\newblock On the {$K$}-theory of algebraically closed fields.
	\newblock {\em Invent. Math.}, 73(2):241--245, 1983.
	
	\bibitem[SV00]{suslin_bloch-kato_1998}
	Andrei Suslin and Vladimir Voevodsky.
	\newblock Bloch--{K}ato conjecture and motivic cohomology with finite coefficients.
	\newblock In {\em The arithmetic and geometry of algebraic cycles ({B}anff, {AB}, 1998)}, volume 548 of {\em NATO Sci. Ser. C Math. Phys. Sci.}, pages 117--189. Kluwer Acad. Publ., Dordrecht, 2000.
	
\end{thebibliography}

\end{document}